\documentclass[12pt]{amsart}
\usepackage[top=30truemm,bottom=30truemm,left=25truemm,right=25truemm]{geometry}
\usepackage{mathrsfs}
\usepackage{enumitem}
\usepackage{amsmath, amsthm, amssymb}
\usepackage{color}
\usepackage{bm}
\usepackage{amsfonts,amssymb}
\usepackage{dsfont}
\usepackage{amscd}
\usepackage{extarrows}
\usepackage{amsmath}
\usepackage{mathtools}
\usepackage{amssymb}
\usepackage{amsthm}
\usepackage{mathrsfs}
\usepackage{amscd}
\usepackage[all]{xy}
\usepackage{geometry}
\usepackage[colorlinks]{hyperref}

\hypersetup{pdfencoding=auto}
\usepackage[capitalize]{cleveref}
\newtheorem{theorem}{Theorem}[section]
\newtheorem{definition}[theorem]{Definition}
\newtheorem{lemma}[theorem]{Lemma}
\newtheorem{corollary}[theorem]{Corollary}
\newtheorem{proposition}[theorem]{Proposition}
\newtheorem{remark}[theorem]{Remark}

\newcommand{\be}{\begin{equation}}
\newcommand{\ee}{\end{equation}}
\newcommand{\bea}{\begin{eqnarray}}
\newcommand{\eea}{\end{eqnarray}}
\newcommand{\ben}{\begin{eqnarray*}}
	\newcommand{\een}{\end{eqnarray*}}
\newcommand{\bt}{\begin{split}}
	\newcommand{\et}{\end{split}}
\newcommand{\bet}{\begin{equation}}

\begin{document}

\title[Degenerate Complex Hessian type equations and Applications]
{Degenerate Complex Hessian type equations on compact Hermitian manifolds and Applications}

\author[K. Pang]{Kai Pang}
\address{Kai Pang: School of Mathematical Sciences\\ Beijing Normal University\\ Beijing 100875\\ P. R. China}
\email{202331130031@mail.bnu.edu.cn}
\author[H. Sun]{Haoyuan Sun}
\address{Haoyuan Sun: School of Mathematical Sciences\\ Beijing Normal University\\ Beijing 100875\\ P. R. China}
\email{202321130022@mail.bnu.edu.cn}
\author[Z. Wang]{Zhiwei Wang}
\address{Zhiwei Wang: Laboratory of Mathematics and Complex Systems (Ministry of Education)\\ School of Mathematical Sciences\\ Beijing Normal University\\ Beijing 100875\\ P. R. China}
\email{zhiwei@bnu.edu.cn}

\author[X. Zhou]{Xiangyu Zhou}
\address{Xiangyu Zhou: Institute of Mathematics\\ Academy of Mathematics and Systems Science\\
and Hua Loo-Keng Key Laboratory of Mathematics\\ Chinese Academy of Sciences\\ Beijing 100190\\ P. R. China}

\email{xyzhou@math.ac.cn}

\begin{abstract}
   The aim of this paper is to further develop the theory of the degenerate complex Hessian equations on compact Hermitian manifolds. Building upon the generalization of the Bedford-Taylor pluripotential theory to complex Hessian equations by Ko\l odziej-Nguyen, we solve these equations in the $(\omega, m)$-positive cone, $(\omega, m)$-big classes  and in nef classes, where $\omega$ is a reference Hermitian metric. These results are also new  in the K\"ahler case.  Moreover, we adapt our techniques to solve complex Monge-Amp\`ere equations in nef classes with mild singularities. The solutions we obtain, in the compact K\"ahler case, coincide with those for the complex Monge-Amp\`ere equations in the sense of the non-pluripolar product introduced by Boucksom-Eyssidieux-Guedj-Zeriahi.  One of the key ingredients in the proof is the adaption, to the Hermitian setting, of a new a priori $L^\infty$-estimate established by Guo-Phong-Tong and Guo-Phong-Tong-Wang. 
\end{abstract}

\maketitle
\tableofcontents
\section{Introduction}
Since Yau's resolution of the Calabi conjecture \cite{Yau78}, the study of complex Monge–Ampère equations has become a central subject in the intersection of partial differential equations and complex geometry. Many impressive results have been established in the literature; see, for example, \cite{Yau78, Aub78, BT76, BT82, Che87, Kol98, TW10, EGZ09, BEGZ10, EGZ11, KN15, Ngu16, GL23, LWZ24} and the references therein.

The complex Hessian equation also has a strong background in the study of important problems in complex geometry and related fields, as demonstrated in works such as \cite{FY08,LS15,PPZ19, Song20,CJY20,Chen21, Chen22} and the references therein.

On compact K\"ahler manifolds,  Dinew-Ko\l odziej \cite{DK14} established the Calabi-Yau type theorem for  non-degenerate complex Hessian equations. Later  Sz\'ekelyhidi \cite{Sze18} and Zhang \cite{Zh15}   solve the non-degenerate complex Hessian equation (Calabi-Yau type) on compact Hermitian manifolds. 

The main concern of this paper is the degenerate complex Hessian equation. 

Let $(X,\omega)$ be a compact Hermitian manifold of complex dimension $n$ and $\omega$ be a Hermitian metric on $X$. Let $\beta$ be a real $(1,1)$-form on $X$ with some proper positivity. Let $0\leq f\in L^p(X,\omega^n)$ such that $\int_Xf\omega^n>0$. Let $m$ be a positive integer such that $1\leq m\leq n$. The degenerate complex Hessian equation is written as follows:

\begin{align}\label{equ: Hessian}
(\beta+dd^c \varphi)^m \wedge \omega^{n-m}=cf \omega^n.
\end{align}
In \cite{KN16}, Ko\l odziej-Nguyen show that the equation (\ref{equ: Hessian}) has a continuous solution  for the case $\beta=\omega$.
  When $\beta$ is semipositive and big (i.e. there is a quasi-plurisubharmonic function $\rho$ on $X$ such that $\rho$ has analytic singularities and  $\beta+dd^c\rho$ dominates a Hermitian form) Guedj-Lu \cite{GL25} obtained that there is a bounded solution to (\ref{equ: Hessian}), which is continuous on the ample locus of $\beta$.  

Both the work of Ko\l odziej-Nguyen and that of Guedj-Lu require strong positivity of the given real $(1,1)$-form $\beta$, because at that time, the pluripotential theory for complex Hessian equations had only been developed for $m$-subharmonic functions that admit  decreasing approximations by  smooth $m$-subharmonic functions \cite{KN16}.

More recently, the groundbreaking work of Ko\l odziej and Nguyen \cite{KN25a} established a comprehensive pluripotential framework for complex Hessian equations, generalizing the Bedford-Taylor theory to this setting. This breakthrough paved the way for studying the degenerate complex Hessian equation in more general settings, which serves as the starting point of the present paper.

The study of the Hessian equation (\ref{equ: Hessian}) will proceed in two main directions: in the case where $\beta$ is $m$-positive and in the case where $\beta$ is nef.

As a warm up, we first study the Hessian equation (\ref{equ: Hessian})  in the case when $\beta$ is $m$-positive. By adapting the argument of \cite{KN16}, we can get the following two theorems.

\begin{theorem}\label{thm: m-positive beta}
Let $(X,\omega)$ be a complex Hermitian manifold of complex dimension $n$ with a Hermitian metric $\omega$.   Let $\beta$ be an $(\omega,m)$-positive form and let $0\leq f\in L^p(X,\omega^n)$ with $p>\frac{n}{m}$ and $\int_Xf\omega^n>0$, where $m$ is a positive integer such that $1\leq m\leq n$. Then there exists a unique constant $c>0$ and a continuous function $\varphi\in \operatorname{SH}_m(X,\beta,\omega)\cap C^{0}(X)$ such that
        $$
(\beta+dd^c\varphi)^m\wedge\omega^{n-m}=cf\omega^n.
        $$
 Furthermore, we have
    $$
    \operatorname{osc} \varphi\leq C,
    $$ 
    where $C$ is the constant depends on $\beta,n,m,B,X,\omega$ and the upper bound of  $\|f\|_p$.
\end{theorem}

\begin{theorem}\label{thm: twist for m-positive beta}
     Let $(X,\omega)$ be a complex Hermitian manifold of complex dimension $n$ with a Hermitian metric $\omega$.   Let $\beta$ be an $(\omega,m)$-positive form and let $0\leq f\in L^p(X,\omega^n)$ with $p>\frac{n}{m}$ and $\int_Xf\omega^n>0$, where $m$ is a positive integer such that $1\leq m\leq n$.  Then for any positive number $\lambda>0$, there exists a unique continuous function $\varphi\in \operatorname{SH}_m(X,\beta,\omega)\cap C^{0}(X)$ such that
        $$
(\beta+dd^c\varphi)^m\wedge\omega^{n-m}=e^{\lambda\varphi}f\omega^n.
        $$
    Furthermore, we have
    $$
    \operatorname{osc} \varphi\leq C,
    $$ 
    where $C$ is a constant depending on $X,\beta,\omega,\lambda$ and $\|f\|_p$.
\end{theorem}
The proof follows \cite{KN16} closely, with the only possible mild novelty being the establishment of a modified comparison principle for $(\omega,m)$-non-collapsing forms (see \cref{modified comparison principle}), which is itself inspired by \cite{GL22}. 

As an application, we derive the following approximation theorem:
\begin{theorem}\label{introduction thm: decreasing approximation}Let $(X,\omega)$ be a complex Hermitian manifold of complex dimension $n$ with a Hermitian metric $\omega$.   Let $\beta$ be an $(\omega,m)$-positive form, where $m$ is a positive integer such that $1\leq m\leq n$.  Then, for any $u\in \operatorname{SH}_m(X,\beta,\omega)$, there exists a decreasing sequence of smooth $(\beta,\omega,m)$-subharmonic functions on $X$ converging pointwise to $u$.
\end{theorem}

When $\beta$ is only $(\omega,m)$-semipositive, following the arguments of \cite{GL23}, we can prove the following theorem.

\begin{theorem}\label{introduction twist equation for beta m-big}
 Let $(X,\omega)$ be a complex Hermitian manifold of complex dimension $n$ with a Hermitian metric $\omega$.   Let $\beta$ be an $(\omega,m)$-semi-positive and $(\omega,m)$-big form and let $0\leq f\in L^p(X,\omega^n)$ with $p>\frac{n}{m}$ and $\int_Xf\omega^n>0$, where $m$ is a positive integer such that $1\leq m\leq n$. Fix a constant $\lambda>0$. Then there exists a unique bounded function $\varphi\in \operatorname{SH}_m(X,\beta,\omega)\cap L^\infty(X)$ such that
        $$
(\beta+dd^c\varphi)^m\wedge\omega^{n-m}=e^{\lambda\varphi}f\omega^n.
        $$
    Furthermore, we have
    $$
    \operatorname{osc} \varphi\leq C,
    $$ 
    where $C$ is a constant depends on $\beta,\omega,p,X,\|f\|_p$.
\end{theorem}
We expect that the  equation (\ref{equ: Hessian}) can be solved under the assumption of the above theorem. The main difficulty lies in establishing an $L^\infty$-estimate in this case.

Next, we study the Hessian equation (\ref{equ: Hessian}) in the case when $\beta$ is nef. 

The main idea proceeds in three steps. First, solving the non-degenerate Hessian equation for $\beta + t\omega$ with $t > 0$ to obtain a sequence of smooth solutions $\varphi_t$ by applying the results of Székelyhidi \cite{Sze18} and Zhang \cite{Zh15}. Next, establishing a uniform bound for $\varphi_t$ and extracting a convergent subsequence to obtain a limit function $\varphi$. Finally, proving that $\varphi$ is the desired solution via an argument based on the domination principle and an envelope construction.

A crucial aspect of the second step is the derivation of a priori $L^\infty$-estimates for solutions in nef Bott-Chern classes. This is inspired by the recent breakthrough of Guo-Phong-Tong \cite{GPT23} and Guo-Phong-Tong-Wang \cite{GPTW24} for compact K\"ahler manifolds.

To achieve the  a priori $L^\infty$-estimates in our situation,   we introduce a  new invariant $SL_{\omega,a}({\beta})$, called the sup-slope (see \cref{def of slope in nef class}),  which  plays a pivotal role in the control of the twist constant $c$ both in the a priori estimates and in solving the equations. 

We remark that the sup-slope was first introduced by Guo-Song \cite{GS24} in Hermitian classes, we extend their definition to nef classes and introduce a new invariant called $p$-slope, which is much more precise than the lower volume studied in \cite{GL23,BGL25,SW25}, as it is strictly positive when $\beta$ is a Hermitian form and, even when $\beta$ is semi-positive and big. For further properties of sup-slopes, see \cref{section:sup slopes}. 

We summarize the obtained a priori $L^\infty$-estimate in the following theorem.

\begin{theorem}\label{thm: a priori in nef class}
 Let $(X,\omega)$ be a complex Hermitian manifold of complex dimension $n$ with a Hermitian metric $\omega$. Let $m$ be a positive integer such that $1\leq m\leq n$. Let $\{\beta\}\in BC^{1,1}(X)$ be a nef Bott-Chern class satisfying $SL_{\omega,a}(\{\beta\})>0$ for some constant $1\leq a<\frac{n}{n-m}$. Assume also $\varphi_t\in \operatorname{SH}_m(X,\beta+t\omega,\omega)\cap C^\infty(X)$ satisfying
\begin{equation}
(\beta+t\omega+dd^c\varphi_t)^m\wedge\omega^{n-m}= c_te^{F_t}\omega^n,\quad\sup_X\varphi_t=0.
\end{equation}
Fix a constant $p>na$. Then there exists a uniform constant $C$ depending on $\beta,\omega,m,n,p$, the upper bound of $\|e^{\frac{an}{m}F_t}\|_{L^1(\log L)^p}:=\int_{X}e^{\frac{an}{m}F_t(z)}(1+\frac{an}{m}|F_t(z)|)^p\omega^n,E_t:=\int_X(-\varphi_t+V_t)^ae^{\frac{an}{m}F_t}\omega^n$ and the lower bound of  $SL_{\omega,a}(\{\beta\})$ such that
$$
0\leq-\varphi_t+V_t\leq C.
$$
\end{theorem}

Using the a priori $L^\infty$-estimates above, we can solve a class of degenerate Hessian equations in a nef Bott-Chern class. Our main result is as follows.

\begin{theorem}\label{thm: solve nef 1}
    Let $(X,\omega)$ be a complex Hermitian manifold of complex dimension $n$ with a Hermitian metric $\omega$. Let $m$ be a positive integer such that $1\leq m\leq n$. Let $\{\beta\}\in BC^{1,1}(X)$ be a nef Bott-Chern class satisfying $SL_{\omega,a}(\{\beta\})>0$ for some constant $1\leq a<\frac{n}{n-m}$. Assume moreover that there is a function $\rho\in \operatorname{SH}_m(X,\beta,\omega)\cap L^\infty(X)$ and that $\{\beta\}$ is $(\omega,m)$-non-collapsing. Fix $0\leq f\in L^p(\omega^n)$ with $p>\frac{n^2}{m^2-(1-\frac{1}{a})mn}$ satisfying $\int_Xf\omega^n>0$. Then, there exists $\varphi\in \operatorname{SH}_m(X,\beta,\omega)\cap L^\infty(X)$ and $c>0$ solving
$$
(\beta+dd^c\varphi)^m\wedge\omega^{n-m}=cf\omega^n,\quad\sup_X\varphi=0.
$$
Moreover, $osc_X\varphi$ is bounded by a constant $C$ depending on $\beta,\omega,\|f\|_p,n,m$, the bound of $\rho$ and the lower bound of $SL_{\omega,a}(\{\beta\})$.
\end{theorem}

\begin{theorem}\label{thm: solve nef 2}
   Let $(X,\omega)$ be a complex Hermitian manifold of complex dimension $n$ with a Hermitian metric $\omega$. Let $m$ be a positive integer such that $1\leq m\leq n$.  Let $\{\beta\}\in BC^{1,1}(X)$ be a nef Bott-Chern class satisfying $SL_{\omega,a}(\{\beta\})>0$ for some positive constant $1\leq a<\frac{n}{n-m}$. Assume moreover that there is a function $\rho\in \operatorname{SH}_m(X,\beta,\omega)\cap L^\infty(X)$. Fix $0\leq f\in L^p(\omega^n)$ with $p>\frac{n^2}{m^2-(1-\frac{1}{a})mn}$ satisfying $\int_Xf\omega^n>0$. Then, there exists a unique $\varphi\in \operatorname{SH}_m(X,\beta,\omega)\cap L^\infty(X)$ solving
$$
(\beta+dd^c\varphi)^m\wedge\omega^{n-m}=e^{\varphi}f\omega^n.
$$
Moreover, the lower bound of $\varphi$ is bounded by a constant $C$ depending on $\beta,\omega,\|f\|_p,n,m$, the bound of $\rho$ and the lower bound of $SL_{\omega}(\{\beta\})$. Furthermore, if $\beta$ is $(\omega,m)$-non-collapsing, the solution is unique.
\end{theorem}

\begin{remark}
When $m=n$, \cref{thm: solve nef 1} and \cref{thm: solve nef 2} recover the results of \cite{LWZ24, SW25}. When $\beta$ is semi-positive and big, \cref{thm: solve nef 1} and \cref{thm: solve nef 2} recover the   results of  \cite{GL25} for $p>\frac{n^2}{m^2-(1-\frac{1}{a})mn}$. When the metric $\omega$ is K\"ahler,  the condition $SL_{\omega,a}(\{\beta\})>0$ reads $\int_X\beta^n>0$, \cref{thm: solve nef 1} and \cref{thm: solve nef 2} also produces new results on K\"ahler manifolds (see \cref{Corollary Kahler} and \cref{Corollary Kahler twist}).
\end{remark}

The final part of this paper is devoted to solving a class of degenerate complex Monge-Amp\`ere equations in a nef class with mild singularities.

The nef cone and the pseudoeffective cone in $H^{1,1}(X, \mathbb{R})$ are crucial objects of study, which are closely related to the geometric structure of the manifold $X$.  A more general space, called  the Bott-chern space $BC^{p,q}(X)$,  is defined as the cokernel of $dd^c:\Omega^{p-1,q-1}(X)\rightarrow\Omega^{p,q}(X)$.
Let $\{\beta\}\in BC^{1,1}(X)$ be a nef Bott-Chern class on $X$.
In \cite{LWZ24} and \cite{SW25}, the authors studied complex Monge-Amp\`ere equations in a nef Bott-Chern class admitting a bounded potential $\rho\in\operatorname{PSH}(X,\beta)\cap L^\infty(X)$ and obtained bounded solutions. It is natural to relax the boundedness assumption of the potential $\rho$ and consider more singular potentials. 
We say a $\beta$-plurisubharmonic function $u$ belongs to the class $D(X,\beta)$ if for any local coordinate ball $\Omega\subset X$, there is a smooth function $\rho\in C^\infty(\Omega)$ (we usually choose $dd^c\rho\geq\beta$) such that $u+\rho\in D(\Omega)$. As in \cite{Sal25}, the measure $(\beta+dd^cu)^n$ for $u\in D(X,\beta)$ can be defined as follows: fix a coordinate ball $B\subset X$ and select a potential $\rho\in C^\infty(B)$ such that $dd^c\rho\geq\beta$ in $B$, then
$$
(\beta+dd^cu)^n:=\sum_{j=0}^nC_n^j(dd^c(u+\rho))^j\wedge(\beta-dd^c\rho)^{n-j}.
$$
Since $u+\rho\in\mathcal{E}(B)$, we get a well-defined Radon measure on $B$ thanks to \cite[Theorem 4.2]{Ceg04}. 

Let $\{\beta\}\in BC^{1,1}(X)$ be a nef Bott-Chern class with \textbf{mild singularities}, i.e.,  $D(X,\beta)\neq\emptyset$. That is to say, $V_\beta\in D(X,\beta)$, where $V_\beta:=\sup\{u\in\operatorname{PSH}(X,\beta),u\leq0\}$. An important example is due to \cite[Theorem 2.5, Proposition 2.9]{Dem93-1}, \cite[Corollary 3.6]{FS95}, \cite[Theorem 3.1.2, Theorem 3.1.3]{Ra01}, when $u\in\operatorname{PSH}(X,\beta)$ and $H_2(L(u))=0$, we have $u\in D(X,\beta)$. Here $H_2$ denotes the 2- dimensional Hausdorff measure. 

A function $u\in\operatorname{PSH}(X,\beta)$ is said to be of minimal singularity type if there is a constant $C$ such that $|u-V_\beta|\leq C$. Recall that the unbounded locus $L(V_\beta)$ of $V_\beta$ is defined as the set of all the points $x\in X$ such that $V_\beta$ is unbounded in each neighborhood of $x$, which is a closed subset of $X$. Set $P(V_\beta):=\{V_\beta=-\infty\}$ , i.e. the polar locus of $V_\beta$.

We derive the following $L^\infty$ estimates analogous to \cref{thm: a priori in nef class} for the Monge-Amp\`ere equations, which is a generalization of \cite[Theorem 3.1]{SW25} and \cite[Theorem 1]{GPTW24}.
\begin{theorem}\label{thm: a priori in nef class for MA}
  Let $(X,\omega)$ be a complex Hermitian manifold of complex dimension $n$ with a Hermitian metric $\omega$. Let $\{\beta\}\in BC^{1,1}(X)$ be a nef Bott-Chern class satisfying $SL_{\omega,a}(\{\beta\})>0$ for some constant $a\geq1$. Assume also $\varphi_t\in \operatorname{PSH}(X,\beta+t\omega)\cap C^\infty(X)$ satisfying
\begin{equation}
(\beta+t\omega+dd^c\varphi_t)^n= c_te^{F_t}\omega^n,\quad\sup_X\varphi_t=0.
\end{equation}
Here $F_t\in C^\infty(X)$. We also fix a constant $p>na$. Then there exists a uniform constant $C$ depending on $\beta,\omega,n,p$, the upper bound of $\int_{X}e^{aF_t(z)}[\log(1+e^{aF_t(z)})]^p\omega^n$, the lower bound of $\int_Xe^{\frac{F_t}{n}}\omega^n$ and the lower bound of  $SL_{\omega,a}(\{\beta\})$ such that
$$
0\leq-\varphi_t+V_t\leq C.
$$
Here $V_t:=\sup\{u\;|\;u\in \operatorname{PSH}(X,\beta+t\omega),u\leq0\}$ is the largest non-positive $(\beta+t\omega)$-psh function.
\end{theorem}

  Applying above $L^\infty$-estimate and an approximation argument, we first provide the following resolution of CMA equations with an exponential twist:
\begin{theorem}\label{thm: small unbounded locus twist equation}
   Let $(X,\omega)$ be a complex Hermitian manifold of complex dimension $n$ with a Hermitian metric $\omega$.   Let $\{\beta\}\in BC^{1,1}(X)$ be a nef Bott-Chern class satisfying $SL_{\omega,a}(\{\beta\})>0$ and $ D(X,\beta)\neq\emptyset$. Then, for any $f\in L^p(X), p>a$ satisfying $\int_Xf\omega^n>0$, there exists a function $\varphi$ of minimal singularity type solving the following equation:
    $$
    (\beta+dd^c \varphi)^n = e^ \varphi f\omega^n.
    $$
    Moreover, $\varphi$ is unique if we assume furthermore that $SL_{\omega,1}(\{\beta\})=\underline{\operatorname{Vol}}(\{\beta\})>0$ and $L(V_\beta)$ is pluripolar.
\end{theorem}
 \Cref{thm: small unbounded locus twist equation} offers a unified framework that generalizes several key results for degenerate complex Monge-Ampère equations, including \cite[Theorem 3.4]{GL23}, \cite[Theorem 1.2]{LWZ24}, and \cite[Theorem 1.1]{SW25}. The method we employ also appears to be more concise than those used in the aforementioned works.

Next, we establish a mass comparison theorem, which will play a key role in proving $(\beta+dd^cu)^n$ is non-pluripolar for $u\in\operatorname{PSH}(X,\beta)$ with minimal singularity:
\begin{theorem}\label{intro_thm: Demailly Second comparison_global}
 Let $(X,\omega)$ be a complex Hermitian manifold of complex dimension $n$ with a Hermitian metric $\omega$.   Let $\beta$ be a smooth $(1,1)$- form on $X$ satisfying $D(X,\beta)\neq\emptyset$ and $u,v\in\operatorname{PSH}(X,\beta)$ be $\beta$- psh functions with minimal singularities. Assume moreover that $(\beta+dd^cv)^n$ is non-pluripolar and that $u-v$ is bounded. Then we have that $(\beta+dd^cu)^n$ is non-pluripolar.
\end{theorem}
Combining \cref{intro_thm: Demailly Second comparison_global} with a mass concentration result for quasi-plurisubharmonic envelopes (see \cref{lem: contact for beta lsc MA}), we prove the following interesting result:
\begin{corollary}\label{intro:non-pluripolar} 
 Let $(X,\omega)$ be a complex Hermitian manifold of complex dimension $n$ with a Hermitian metric $\omega$.  Let $\{\beta\}\in BC^{1,1}(X)$ be a nef Bott-Chern class on $X$ satisfying $D(X,\beta)\neq\emptyset$. Then, for each $u\in\operatorname{PSH}(X,\beta)$ with minimal singularity, the Monge-Amp\`ere measure $(\beta+dd^cu)^n$ does not charge pluripolar sets. In particular, $u$ has zero Lelong number everywhere. 
\end{corollary}

 Now that $(\beta+dd^cu)^n$ is non-pluripolar for $u\in\operatorname{PSH}(X,\beta)$ with minimal singularity, all previous techniques, such as  envelopes and domination principle,  can be extended to our case, i,e., for unbounded $\beta$-psh function with mild singularities. Here, we shall make a technical assumption that $L(V_\beta)$ is a pluripolar set to ensure the classical convergence result to hold (see \cref{continuity} below). Then the methods of \cite{GL23} can be employed to get the folllowing results. 
\begin{theorem}\label{thm: small unbounded locus non-twist equation}
   Let $(X,\omega)$ be a complex Hermitian manifold of complex dimension $n$ with a Hermitian metric $\omega$.   Let $\{\beta\}\in BC^{1,1}(X)$ be a nef Bott-Chern class satisfying $\underline{\operatorname{Vol}}(\beta)>0$ and $D(X,\beta)\neq\emptyset$. Assume moreover that $L(V_\beta)$ is a pluripolar set. Then, for any $f\in L^p(X), p>1$ such that $\int_Xf\omega^n>0$, there exists a function $\varphi$ of minimal singularity type  and a unique constant $c>0$ solving the following equation:
    $$
    (\beta+dd^c \varphi)^n =c f\omega^n.
    $$
\end{theorem}
As applications, we present partial solutions (see \cref{thm:TW}, \cref{thm:DP}) to the Tosatti-Weinkove and Demailly-Păun conjectures, following the method developed in \cite{Ch16, Pop16, Ngu16, LWZ24, SW25}.



\subsection*{Acknowledgements}
Z. Wang is very  grateful to Ngoc Cuong Nguyen for many helpful discussions on the degenerate complex Hessian equations. This research is supported by the National Key R\&D Program of China (Grant No. 2021YFA1002600 and  No. 2021YFA1003100).  Z. Wang and X. Zhou are partially supported by grants from the National Natural Science Foundation of China (NSFC) (No. 12571085) and (No. 12288201) respectively. Z. Wang is also supported by the Fundamental Research Funds for the Central Universities.

\section{Preliminary}
\subsection{\texorpdfstring{$(\omega,m)$}{(omega, m)}-subharmonic functions}
 Let $(X, \omega)$ be a compact Hermitian manifold of dimension $n$ equipped  with a Hermitian metric $\omega$. Let $m$ be a positive integer such that $1\leq m\leq n$. In this section, we begin by recalling some basic facts about $(\omega, m)$-subharmonic functions. We also use the notation $(\Omega, \omega)$ to denote a domain in $\mathbb{C}^n$ equipped with a Hermitian form $\omega$. 

\begin{definition}\label{definition: m-positive}
    A smooth real $(1,1)$-form $\alpha$ is called $(\omega,m)$-positive in $\Omega$ if the following inequalities hold pointwise in the sense of smooth forms:
    $$
    \alpha^{k}\wedge\omega^{n-k}>0, \quad k=1,...,m.
    $$
    We denote by $\Gamma_m(\omega)$ be the open convex cone of all $(\omega,m)$-positive $(1,1)$-forms. At any fixed point of $X$, we can diagonalize $\omega$ with respect to $\alpha$ and let $\lambda_1,...,\lambda_n$ be the eigenvalues of $\alpha$ with respect to $\omega$. Then, the above condition reads  
    $$
S_{k,\omega}(\alpha)>0,\quad 1\leq k\leq m,
    $$
    where $S_{k,\omega}(\alpha):=\underset{1\leq j_1<...<j_k\leq n}{\sum}\lambda_{j_1}...\lambda_{j_k}$ is the $k$-th symmetric polynomial of the eigenvalues of $\alpha$ with respect to $\omega$.

    Accordingly, we say that $\alpha$ is $(\omega,m)$-semi-positive if
      $$
    \alpha^{k}\wedge\omega^{n-k}\geq0, \quad k=1,...,m.
    $$
    Denote by $\overline{\Gamma_m(\omega)}$ the closed convex cone of all $(\omega,m)$-semi-positive $(1,1)$-forms.

    A function u$\in C^{2}(\Omega)$ is called $(\omega,m)$-subharmonic, denoted $u\in \operatorname{SH}_{m}(\Omega,\omega)$, if $dd^{c}u$ lies in $\overline{\Gamma_m(\omega)}$ at all points of $X$. It is called strictly $(\omega,m)$-subharmonic if $dd^cu$ lies in $\Gamma_m(\omega)$ pointwise on $X$.
\end{definition}

 As observed in \cite{Błocki05}, we have the following useful characterization of $(\omega,m)$-positive forms and smooth $(\omega,m)$-subharmonic functions:
    
\begin{lemma}\label{lemma: criterion}
    Let u$\in C^{2}(\Omega)$. Then u is $(\omega,m)$-subharmonic if and only if for arbitrary $(\omega,m)$-positive $(1,1)$- forms $\alpha_1,...\alpha_{m-1}$ on $\Omega$, we have
    $$
    dd^{c}u\wedge\alpha_1\wedge...\wedge\alpha_{m-1}\wedge\omega^{n-m}\geq 0.
    $$
\end{lemma}
This motivates the following definition:
\begin{definition}
       A current $T$ on $\Omega$ is said to be $(\omega,m)$-positive if for arbitrary $(\omega,m)$-positive $(1,1)$- forms $\alpha_1,...\alpha_{m-1}$ on $\Omega$, we have
    $$
    T\wedge\alpha_1\wedge...\wedge\alpha_{m-1}\wedge\omega^{n-m}\geq 0.
    $$
\end{definition}

To introduce the notion of general $(\omega,m)$-subharmonic functions, we first define $\omega$-subharmonic functions.
\begin{definition}
    A smooth function u is called $\omega$-subharmonic if it satisfies
    $$
\Delta_{\omega}u:=dd^cu\wedge\omega^{n-1}/\omega^n\geq0.
    $$
\end{definition}

\begin{definition}[{\cite[Definition 9.2]{GN18}, \cite[Deinition 1.1.1]{GL25}}]\label{omega-sh}
    A function $u\in L_{loc}^1(\Omega,\omega^n)$ is called $\omega$-subharmonic if $u$ is upper semicontinuous and for every relative compact open subset $U\Subset \Omega$ and every $h\in \mathcal C^0(\overline U)$  and $\Delta_\omega h=0$ in $U$, if $h\geq u$ on $\partial D$, then $h\geq u$ on $D$, where $\Delta_\omega:=\sum_{0\leq i,j\leq n}\omega^{\bar ji}\frac{\partial^2}{\partial z^i\partial \bar z^j}$.
    \end{definition}
It is a fact \cite[Corollary 9.8]{GN18} that if $u$ is $\omega$-subharmonic, then for any point $p\in\Omega$, there is a neighborhood of $p$ where there is a decreasing sequence of smooth $\omega$-subharmonic functions converge pointwise to $u$. 

We have some equivalent definitions of $\omega$-subharmonic functions:
 \begin{lemma}\label{equivalence of omega sh}
       A function $u\in L_{loc}^1(\Omega,\omega^n)$ is $\omega$-subharmonic if and only if $dd^cu\wedge\omega^{n-1}\geq0$ in the distributional sense and $u$ is strongly upper semi-continuous, i.e., $\forall x\in\Omega$,
       $$
 u(x)=\underset{\Omega\ni y\rightarrow x}{\operatorname{ess}\limsup}\,u(y):=\underset{r\searrow0}{\lim}\operatorname{ess}\underset{B_r(x)}{\sup}u.
       $$
 \end{lemma}
\begin{proof}
    See \cite{HL13} or \cite[section 9]{GN18}.
\end{proof}
\begin{definition}\label{def: def of m-sh}
    A function $u\in L_{loc}^1(\Omega,\omega^n)$ is called $(\omega,m)$-subharmonic if it satisfies (i) and one of (ii), (iii), (iv):
    \begin{enumerate}
   \item $dd^cu$ is an $(\omega,m)$-positive current, i.e., for arbitrary m-positive (1,1)-forms $\alpha_1,...\alpha_{m-1}$ on $\Omega$, the following inequality holds in the weak sense of currents:
    $$
    dd^{c}u\wedge\alpha_1\wedge...\wedge\alpha_{m-1}\wedge\omega^{n-m}\geq 0.
    $$
   \item    $u$ is strongly upper semi-continuous, $$i.e. \forall x\in\Omega, u(x)=\underset{\Omega\ni y\rightarrow x}{\operatorname{ess}\limsup}\,u(y):=\underset{r\searrow0}{\lim}\operatorname{ess}\underset{B_r(x)}{\sup}u .$$
   \item $u$ is $\omega$-subharmonic.
   \item Assume $v\in L_{loc}^1(X,\Omega)$, upper semi-continuous and satisfies (i). If u and v coincide almost everywhere, then $u\leq v$.
\end{enumerate} 
\end{definition}

The proof of the equivalence in \cref{equivalence of omega sh} and \cref{def: def of m-sh} uses a deep result of linear elliptic operators in \cite{Lit63} and the existence of Poisson kernels for general elliptic operators. There are also several equivalent definitions of $(\omega,m)$-subharmonic functions in the classical sense or the viscosity sense, we refer the reader to \cite[section 9]{GN18} or more generally, \cite{HL13} for the treatment of the equivalence.

\begin{remark}\label{a.e. everywhere}
    Conditions (2), (3), and (4) in \cref{def: def of m-sh} are just used to ensure that when two $(\omega,m)$-subharmonic functions are equal almost everywhere, then they equal everywhere.
\end{remark}
\subsection{Definition of Hessian measures}
In this subsection, let $\chi$ be a smooth real $(1,1)$ form on $\Omega$, we will define the Hessian measure $H_{\chi,m}(u):=(\chi+dd^cu)^m\wedge\omega^{n-m}$ for a bounded $(\chi,\omega,m)$-subharmonic function $u$. First, we recall the definition of Hessian measures in the local setting as described in \cite{KN25a}:

\begin{definition}\label{def: Hessian measures in the local setting}
    Let $u_1,...,u_m$ be $(\omega,m)$-subharmonic functions on a bounded domain $\Omega$ in $\mathbb{C}^n$ and $\omega$ be a Hermitian metric in $\Omega$. By \cite[Theorem 3.3]{KN25a}, we can define inductively 
    $$
dd^cu_{p+1}\wedge...\wedge dd^cu_1:=dd^c(u_{p
+1}dd^cu_p\wedge...\wedge dd^cu_1)
    $$
    as closed real currents of order $0$. Then
    $$
H_p(u_1,...,u_p):=dd^cu_{p}\wedge...\wedge dd^cu_1\wedge\omega^{n-m},\quad 1\leq p\leq m
    $$
    is a well-defined  positive current (positive Radon measure when $p=m$) on $\Omega$ (when there is no confusing with the reference hermitian metric $\omega$, we will omit the subscript in $H_{p,\omega}$ for convenience of notations). When $u_1=...=u_m=u$, we write
    $$
H_p(u):=(dd^cu)^p\wedge\omega^{n-m},\quad1\leq p\leq m.
    $$
    \end{definition}

\begin{definition}
    Let $(X,\omega)$ be a compact Hermitian manifold and let $\chi$ be a smooth $(1,1)$-form. A function $u$ on $X$ is called $(\chi,\omega,m)$-subharmonic, denoted $u\in \operatorname{SH}_m(X,\chi,\omega)$, if it can be locally written as a sum of a smooth function and an $\omega$-subharmonic function (see \cref{omega-sh}), and globally for any $(m-1)$-tuple of forms $\alpha_1,...,\alpha_{m-1}$ lie in $\Gamma_m(\omega)$,
    $$
(\chi+dd^cu)\wedge\alpha_1\wedge...\wedge\alpha_{m-1}\wedge\omega^{n-m}\geq0\quad on\;X.
    $$
\end{definition}

From the definition, it follows that if $\Omega$ is any coordinate chart on $X$ and $\phi$ is a strictly plurisubharmonic function on $\Omega$ such that 
$$
dd^c\phi\geq\chi,
$$
on $\Omega$, then $u+\phi$ is $(\omega,m)$-subharmonic on $\Omega$. This leads to the following equivalent characterization of $(\chi,\omega,m)$-subharmonic functions:
\begin{lemma}\label{susc eq def}
   A function $u\in L^1(X,\omega^n)$ is $(\chi,\omega,m)$-subharmonic if and only if it is strongly upper semi-continuous and globally for any $(m-1)$-tuple of forms $\alpha_1,...,\alpha_{m-1}$ lie in $\Gamma_m(\omega)$,    
$$
(\chi+dd^cu)\wedge\alpha_1\wedge...\wedge\alpha_{m-1}\wedge\omega^{n-m}\geq0\quad on\;X.    
$$
\end{lemma}
\begin{proof}
    The only if part follows directly from \cref{equivalence of omega sh}. To prove the if part, fix a local coordinate chart $\Omega$ and a smooth strictly plurisubharmonic function $\phi$ on $\Omega$ such that $dd^c\phi\geq\chi$. Then $\phi+u$ is $(\omega,m)$-subharmonic on $\Omega$ and hence is $\omega$-subharmonic by \cref{def: def of m-sh}. It follows that $u=(u+\phi)-\phi$ is locally a sum of a smooth function and an $\omega$-subharmonic function, the proof is concluded.
\end{proof}

We move on to define the global Hessian measures with respect to a bounded potential:

\begin{definition}\label{global Hessian measures}
    Let $u$ be a bounded (assume its existence) $(\chi,\omega,m)$-subharmonic function on $X$, pick a coordinate chart $\Omega$ and a strictly psh function $\phi$ as above, we define on $\Omega$:
    \begin{align*}
(\chi+dd^cu)^m\wedge\omega^{n-m}&:=[dd^c(u+\phi)+(\chi-dd^c\phi)]^m\wedge\omega^{n-m}\\
&=\sum_{k=0}^{m}C_m^k[dd^c(u+\phi)]^k\wedge(\chi-dd^c\phi)^{m-k}\wedge\omega^{n-m}\\
&=\sum_{k=0}^{m}C_m^kH_k(u+\phi)\wedge(\chi-dd^c\phi)^{m-k}.
    \end{align*}
   It is clear that this definition is independent of the choice of $\phi$. Using a partition of unity, we can glue these local measures to a global positive Radon on $X$. To see that it is positive, it suffices to check locally. Picking a local regularization of $u$ \cite[Corollary 1.2]{GN18} and using the convergence theorem \cite[Lemma 5.1]{KN25a} we get the positivity.
\end{definition}

\begin{remark}
        When $\omega$ is K\"{a}hler and $\chi$ is closed, we may define the non-pluripolar product, as in that of \cite{LWZ24} and \cite{BEGZ10}. In this case, the full mass class $\mathcal{E}_m(X,\chi,\omega)$ can be defined and we can probably apply the variational approach to study the solvability of the Hessian equations as in \cite{DDL25} and \cite{BBGZ13}.
\end{remark}

\subsection{Properties of \texorpdfstring{$(\chi,\omega,m)$}{(omega, m)}-subharmonic functions}
In this subsection, let $\chi$ be a smooth real $(1,1)$-form admitting a Hessian potential $\rho\in \operatorname{SH}_m(X,\chi,\omega)$.
We collect some basic properties of $(\chi,\omega,m)$-subharmonic functions that will be useful later; most of them can be found in \cite{KN16} and \cite{KN25a}.

\begin{proposition}\label{usc regularization}
  Let $\chi$ be a smooth real $(1,1)$-form and let $\{u_\alpha\}_{\alpha\in I}\subset\operatorname{SH}_m(X,\chi,\omega)$ be a family of $(\chi,\omega,m)$-subharmonic functions which is uniformly bounded from above. Set $u(z):=\sup_\alpha u_\alpha(z)$. Then the upper semi-continuous regularization $u^*$ belongs to $\operatorname{SH}_m(X,\chi,\omega)$ again.
\end{proposition}
\begin{proof}
It is well known that the upper semi-continuous regularization $u^*$ is the smallest usc function lies above $u$ and hence it is easy to check that $u^*$ is strongly usc: $\forall x\in X$,
$$ 
u^*(x)=\underset{\Omega\ni y\rightarrow x}{\operatorname{ess}\limsup}\,u^*(y):=\underset{r\searrow0}{\lim}\operatorname{ess}\underset{B_r(x)}{\sup}u^*.       
$$
By the Choquet's lemma $u^*=(\lim_ju_j)^*$ for an increasing family $u_j$. It is then clear that
$$
(\chi+dd^cu^*)\wedge\alpha_1\wedge...\wedge\alpha_{m-1}\wedge\omega^{n-m}\geq0\quad on\;X.    
$$
for any $(m-1)$-tuple of forms $\alpha_1,...,\alpha_{m-1}$ lie in $\Gamma_m(\omega)$. The conclusion then follows from \cref{susc eq def}.
\end{proof}
\begin{proposition}\label{L^1 compactness}
     There is a uniform constant $C=C(X,\chi,\omega)$,  such that for any $\varphi\in \operatorname{SH}_m(X,\chi,\omega)$, $\sup_X\varphi=0$, we have    
$$
\int_X|\varphi|\omega^n\leq C.    
$$
In particular, the family $\{\varphi\in \operatorname{SH}_m(X,\chi,\omega),\;\sup_X\varphi=0\}$ is relatively compact with respect to the $L^1(\omega^n)$-topology in $\operatorname{SH}_m(X,\chi,\omega)$.
\end{proposition}
\begin{proof}
    We can assume without loss of generality that $\chi\leq\omega$ and hence $\operatorname{SH}_m(X,\chi,\omega)\subset \operatorname{SH}_m(X,\omega,\omega)$ and the result follows immediately from \cite[Lemma 3.3]{KN16}.
\end{proof}
We shall also need the following integrability property of $(\chi,\omega,m)$-subharmonic functions:
\begin{proposition}\label{integrability}
    Let $\chi$ be a smooth real $(1,1)$-form on $X$, then for each $1<p<\frac{n}{n-m}$ we have $\operatorname{SH}_m(X,\chi,\omega)\subset L^p(X,\omega^n)$.
\end{proposition}
\begin{proof}
    We may assume without loss of generality that $\chi\leq\omega$, this implies that $\operatorname{SH}_m(X,\chi,\omega)\subset \operatorname{SH}_m(X,\omega,\omega)$. The result follows immediately from the main theorem of \cite{Fang25}.
\end{proof}
The following proposition was proved using the Poisson-Jensen formula for plurisubharmonic functions (see \cite[Theorem 1.48]{GZ17}) and similarly for the local $m$-subharmonic functions, see \cite[Proposition 6.1]{Błocki05}. In our general case, it is more convenient to use the volume-capacity estimate:
\begin{proposition}\label{integrablity 2}
    Let $u_j$ be a sequence of functions in $\operatorname{SH}_m(X,\chi,\omega)$ converging in $L^1(\omega^n)$ to $u\in \operatorname{SH}_m(X,\chi,\omega)$. Then $u_j\rightarrow u$ in $L^p(\omega^n)$ for all $1< p<\frac{n}{n-m}$.
\end{proposition}
\begin{proof}
  Fix a constant $1<p<\frac{n}{n-m}$. We first claim that the sequence is uniformly bounded in $L^p(\omega^n)$. The claim is essentially contained in \cite{Fang25}. Since $\{u_j\}_j$ lies in a compact family of $\operatorname{SH}_m(X,\chi,\omega)$, $\sup_Xu_j$ is uniformly bounded, we may thus assume without loss of generality that $\sup_Xu_j=-1$.

  Using the volume-capacity estimate (\cite[Proposition 3.1]{Fang25}) and the estimate of the capacity of sublevel sets of $u_j$ (\cite[Corollary 3.3]{Fang25}) we can then estimate
  \begin{align*}
      \int_X(-u_j)^p\omega^n&=\int_X\omega^n+\int_1^{+\infty}pt^{p-1}V_{\omega}(u_j<-t)dt\\
      &\leq\int_X\omega^n+C_\tau\int_1^{+\infty}pt^{p-1}\operatorname{\widetilde{Cap}}_{\omega,m}(u_j<-t)^\tau dt\\
      &\leq\int_X\omega^n+C_\tau\int_1^{+\infty}pt^{p-1}(\frac{C}{t})^{\tau}=\int_X\omega^n+pC_\tau C^\tau\int_1^{+\infty}t^{p-\tau-1}<+\infty.
  \end{align*}
Where $\operatorname{\widetilde{Cap}}_{\omega,m}$ is defined as in \cite[Definition 2.12]{Fang25} and $p<\tau<\frac{n}{n-m}$ is a fixed constant. Since $C_\tau,C$ are both uniform constants by the above two theorems in \cite{Fang25}, we have finished the proof of the claim. 

We now show that $u_j\rightarrow u$ in $L^p(\omega^n)$. The proof can be easily adapted from \cite[Theorem 1.48]{GZ17}. If the sequence $\{u_j\}_j$ is uniformly bounded, the conclusion follows immediately from the dominated convergence theorem. In general, we assume without loss of generality that $\sup_Xu_j=0$ and set for each $m\geq1$ and $j\in\mathbb{N}$
$$
u_j^m:=\max(u_j,-m),\quad u^m:=\max(u,-m).
$$
Then the triangle inequality yields 
$$
\|u_j-u\|_{L^p}\leq\|u_j^m-u_j\|_{L^p}+\|u_j^m-u^m\|_{L^p}+\|u^m-u\|_{L^p}.
$$
By the previous discussion, the second term converges to $0$ as $j\rightarrow\infty$ and the third term converges to $0$ by the monotone convergence theorem. It remains to show that the first term converges uniformly to $0$ in $j$ as $m\rightarrow+\infty$. Fix a small constant $\delta>0$ such that $p+\delta<\frac{n}{n-m}$. It follows from the claim that the sequence $\{u_j\}_j$ is uniformly bounded in $L^{p+\delta}(\omega^n)$. Using the Markov's inequality we can write
\begin{align*}
    \int_X|u_j^m-u_j|^p\omega^n&=2\int_{\{u_j\leq-m\}}|u_j|^p\omega^n\\
    &=2\int_{\{(-u_j)^\delta\geq m^\delta\}}|u_j|^p\omega^n\\
    &\leq\frac{2}{m^\delta}\int_X|u_j|^{p+\delta}\omega^n,
    \end{align*}
    the last term $\rightarrow0$ uniformly in $j$ as $m\rightarrow+\infty$ and hence the proof is concluded.
\end{proof}

\begin{remark}
In the above proof we have used the estimates for the capacity $\widetilde{Cap}$ of sublevel sets established in \cite{Fang25}. Very recently, Ko\l odziej-Nguyen \cite[Theorem 1.1 ]{KN25b} established a sharp estimate of the decay of the general Hessian capacity, which could be used in the proof of the integrability.
\end{remark}

\begin{lemma}\label{compactness 2}
    If $\mu$ is a probability measure such that $\operatorname{SH}_m(X,\chi,\omega)\subset L^1(\mu)$, then
    $$
F_\mu:=\{u\in \operatorname{SH}_m(X,\chi,\omega) \;|\;\int_Xud\mu=0\} 
    $$
    is relatively compact in $\operatorname{SH}_m(X,\chi,\omega)$. In particular there exists $C_\mu$ such that for any $u\in \operatorname{SH}_m(X,\chi,\omega)$,
    $$
-C_\mu+\sup_Xu\leq\int_Xud\mu\leq \sup_Xu.
    $$
\end{lemma}
\begin{proof}
    The proof follows from \cite[Proposition 5.8]{GZ17} word by word.
\end{proof}
\begin{proposition}\label{Hartogs' lemma}
    Assume $\varphi_j,\varphi$ are $(\chi,\omega,m)$-subharmonic functions such that $\varphi_j\rightarrow\varphi$ in $L^1(\omega^n)$ and almost everywhere. Then we have $\psi_j:=(\underset{k\geq j}{\sup}\,\varphi_k)^*$ decreases to $\varphi$ pointwise.
\end{proposition}
\begin{proof}
Set $\psi:=\lim_j\psi_j$. Note that by our assumption $\varphi_j$ is a compact family and hence they are uniformly bounded from above by \cref{compactness 2}. Therefore, it follows from \cref{usc regularization} that all the $\psi_j$ are $(\chi,\omega,m)$-subharmonic functions. Since $\varphi_j$ converges to $\varphi$ almost everywhere, it follows that $\lim_j\underset{k\geq j}{\sup}\,\varphi_k=\varphi$ almost everywhere. Since $\{\psi_j>\underset{k\geq j}{\sup}\,\varphi_k\}$ has lebesgue measure zero, we deduce that $\varphi=\psi$ almost everywhere and hence everywhere by our definition (see \cref{a.e. everywhere}) .
\end{proof}

\begin{proposition}\cite[Corollary 5.2]{KN25a}\label{max principle}
    Let $u,v$ be bounded $(\omega,m)$-subharmonic functions on a bounded domain $\Omega$ in $\mathbb{C}^n$ and $T:=dd^cv_1\wedge...\wedge dd^cv_{m-p}\wedge\omega^{n-m}$ for bounded $(\omega,m)$-sh functions $v_1,...,v_{i}$, where $1\leq p\leq m$ and  $1\leq i\leq p$. Then
    $$
\mathds{1}_{\{u<v\}}(dd^c\max(u,v))^i\wedge T=\mathds{1}_{\{u<v\}}(dd^cv)^i\wedge T.
    $$
    as currents of order $0$.
    Consequently,
    $$
    (dd^c\max(u,v))^i\wedge T\geq \mathds{1}_{\{u\geq v\}}(dd^cu)^i\wedge T+\mathds{1}_{\{u<v\}}(dd^cv)^i\wedge T.
    $$
\end{proposition}
\begin{proof}
    With \cite[Lemma 5.1]{KN25a}, one can carry out the proof of \cite[Theorem 3.27]{GZ17} to get the above result.
\end{proof}

\begin{proposition}\label{prop: max principle for beta}
    
    Let $u,v$ be bounded $(\chi,\omega,m)$-subharmonic functions on  $X$ (we assume the existence of such bounded potentials) and $T:=(\chi+dd^cv_1)\wedge...\wedge(\chi+ dd^cv_{m-p} )\wedge\omega^{n-m}$ for some bounded $(\chi, \omega,m)$-subharmonic functions $v_1,...,v_{m-p}$, where $1\leq p\leq m$. Then
    $$
\mathds{1}_{\{u<v\}}(\chi+dd^c\max(u,v))^p\wedge T=\mathds{1}_{\{u<v\}}(\chi+ dd^cv)^p\wedge T.
    $$
    Consequently,
    $$
    (\chi+dd^c\max(u,v))^p\wedge T\geq \mathds{1}_{\{u\geq v\}}(\chi+dd^cu)^p\wedge T+\mathds{1}_{\{u<v\}}(\chi+dd^cv)^p\wedge T.
    $$
\end{proposition}
\begin{proof}
The result is local, so we work on a ball $U$. Choose a smooth potential $\phi$ such that $\beta\leq dd^c \phi$, we have
\begin{align*}
   & \mathds{1}_{\{u<v\}\cap U} (\chi+dd^c \max(u,v))^p \wedge T\\
    =&\mathds{1}_{\{u<v\}\cap U} \left[  dd^c \phi +dd^c \max(u,v) + (\chi-dd^c \phi) \right]^p \wedge T\\
   =&  \mathds{1}_{\{u<v\}\cap U} \sum_{i=1}^p (dd^c \max(\phi+u,\phi+v))^i \wedge (\chi-dd^c \phi)^{p-i} \wedge T\\
 =&    \mathds{1}_{\{u<v\}\cap U} \sum_{i=1}^p (dd^c (\phi+v))^i \wedge (\chi-dd^c \phi)^{p-i} \wedge T\\
 =&    \mathds{1}_{\{u<v\}\cap U} (\chi+dd^c v)^p \wedge T,
 \end{align*}
where the first, second, fourth equality is just definition of Hessian measure; the third follows  from a general version of \cref{max principle}.
 \end{proof}
The proof of the following corollary is quite straightforward:
\begin{corollary}\label{cor of max principle}
    Let 
$u,v$ be bounded $(\chi,\omega,m)$-subharmonic functions on $X$ (assume existence again) with $u\geq v$. Then,

$$
\mathds{1}_{\{u=v\}}H_{\chi,m}(u)\geq\mathds{1}_{\{u=v\}}H_{\chi,m}(v).
$$
\end{corollary}

\subsection{Extremal functions,  \texorpdfstring{$m$}{m}-polar set and global \texorpdfstring{$(\chi,\omega,m)$}{(chi,omega,m)}-polar set}
In this subsection, $\chi$ be a (possibly non-closed) real $(1,1)$ form with a bounded Hessian potential $\rho\in \operatorname{SH}_m(X,\chi,\omega)$.
\begin{definition}
    For a Borel set $K$ of $X$, the extremal function $V_{K,m}$ of $K$ with respect to $\chi$ is defined as 
    $$V_{K,m}:=\sup \{v\in \operatorname{SH}_{m}(X,\chi,\omega):v\leq \rho \ \text{on}\ K  \}.$$
    Denote by $V_{K,m}^*$ the upper semi-continuous regularization of $V_{K,m}$, which is a $(\chi,\omega,m)$- subharmonic function by \cref{usc regularization}. We say that $K$ is $(\chi, \omega, m)$- polar set if there exists $u\in \operatorname{SH}_{m}(X,\chi.\omega)$, such that $K\subset \{u=-\infty\}$.
\end{definition}
\begin{lemma} \label{lem: characterization m polar}
    $K$ is $(\chi,\omega ,m)-$polar if and only if $\sup_{X} V_{K,m}=+\infty$. 
\end{lemma}
    \begin{proof}
    We may assume that $\rho\geq 0$.
        Suppose that $K\subset u^{-1}(-\infty)$ for some $u\in \operatorname{SH}_{m}(X,\chi,\omega) $. By considering $u+k$ for $k=1,2,\cdots,$ one can see that $\sup_{X} V_{K,m}=+\infty$. Conversely, let $v_{k}\in \operatorname{SH}_{m}(X,\chi,\omega)$  be a sequence such that $N_{k}:= \sup_{X}v_k \geq 2^k$ and $v_k\leq \rho$ on $K$. Since $\sup_{X}(v_k-N_k)=0$, the sequence $\{v_k-N_k\}_{k\in \mathbb{N^*}}$ is  compact in $L^1(X,\omega^n)$ by \cref{L^1 compactness}. Then the function 
        $$u:= \sum_{k=1}^\infty \frac{v_k-N_k}{2^k} $$ as  a decreasing limit of 
        $$u_l:= \frac{\rho}{2^l} + \sum_{k=1}^{l} \frac{v_k-N_k}{2^k}$$
        is $(\beta,\omega,m)$ subharmonic such that $u\neq - \infty$, and $u=-\infty$ on $K$.
    \end{proof}
\begin{definition}{\cite[Definition 7.1]{KN25a}}
    A set $E$ in $\mathbb{C}^n$ is $m-$ polar if for each $z\in E$ there is an open set $z\in U$ and a $(\omega,m)$- subharmonic function $u$ in $U$ such that $E \cap U \subset \{u=-\infty \}$.

    Let $\{u_{\alpha}\}$ be a family of  $(\omega,m)$-subharmonic function in $\Omega\subseteq \mathbb{C}^n$, which is locally  bounded from above. Let $u=\sup_{\alpha} u_{\alpha}$ be the supreme of this family. It follows from \cref{usc regularization} that the function  $u^*:=(\sup_{\alpha} u_{\alpha})^*$  is $(\omega,m)$-subharmonic. A set of the form  $$ N=\{z\in \Omega: u(z)<u^*(z) \}$$ is called $m$-negligible.
\end{definition}
\begin{lemma}\cite[Theorem 7.8]{KN25a}
    $m$-negligible set is $m$-polar.
\end{lemma}

\begin{lemma}
    Assume that $\chi$ is closed. If  $K$ is a compact subset of $X$ and not $(\chi,m)$-polar, then the measure $(\chi+dd^c V_{K,m}^*)^m \wedge \omega^{n-m}$ puts no mass outside $K$.
\end{lemma}
\begin{proof}
    Fix an arbitrary small ball $U\subseteq X-K$. By \cref{lem: characterization m polar}, we have $\sup_{X}V_{K,m}<+\infty$.
    By Choquet's lemma, we can find an increasing uniformly bounded $(\chi,\omega,m)$- subharmonic function sequence $\{v_k\}_k$, such that  $v_k\leq \rho$ on $K$ and $(\sup v_k)^*=V_{K,m}^*$ belongs to $\operatorname{SH}_{m}(X,\chi,\omega)$.
Now we choose smooth $\phi$ such that $\chi= dd^c\phi$ on $U$, for every small ball $D\subset U$. By \cite[Theorem 8.2, Proposition 2.9, Lemma 5.1]{KN25a},  and the proof of \cite[Corollary 12.5]{Dem93-2}, there exists a $(m,\omega)$ subharmonic function $\tilde{v}_k$ on $U$ such that
$$
\begin{cases}\tilde{v}_k=v_k+\phi & \text { on } U \backslash D, \\ \tilde{v}_k \geq v_k +\phi& \text { on } U,\\ (dd^c\tilde{v}_k)^m\wedge \omega^{n-m} = 0 & \text{ on } D.\end{cases}
$$
Furthermore, By \cite[Section 4.1 and 4.2]{GN18} and the proof of \cite[Corollary 12.5]{Dem93-2}, we have $\tilde{v}_k$ is an increasing sequence. Construct $w_k$ as 
$$
w_k:= \begin{cases}\tilde{v}_k-\phi & \text { on } D, \\ v_k & \text { on } X \backslash D .\end{cases}
$$
We have $w_k$ increase to $V_{K,m}^*$ almost everywhere on $X$. By \cite[Lemma 5.4]{KN25a},
$$ (\chi+dd^c w_k)^m\wedge \omega^{n-m}\rightarrow(\chi+dd^c V_{K,m}^*)^m \wedge \omega^{n-m}$$ weakly on $D$, so that 
$$ (\chi+dd^c V_{K,m}^*)^m\wedge \omega^{n-m}=0 \ \text{on}\ D.$$ Since $U$ is arbitrary, the proof is completed.
\end{proof}

We have used the following lemma:
\begin{lemma}\label{lem: negligible is polar}
    Let $\left\{u_j\right\}_{j \in \mathbb{N}}$ be a sequence of $(\chi,\omega,m)$-psh functions on $X$ that is uniformly bounded above. Then the set $\left\{\left(\sup _j u_j\right)^*>\sup _j u_j\right\}$ is a $m$- polar set.
\end{lemma}
\begin{proof}
    The result is local, so we work on a ball $U$. Choose $\chi\leq dd^c \phi$ on $U$, where $\phi$ is smooth, consider the sequence $\{\phi+u_j \}_j$, then we can use \cite[Theorem 7.8]{KN25a}  to get the desired result.
\end{proof}

\section{Capacity, quasi-continuity and weak convergence theorem}
Let $(X,\omega)$ be a compact Hermitian manifold of complex dimension $n$ equipped with a Hermitian metric $\omega$. Let $m$ be a positive integer such that $1\leq m\leq n$.  Let $\beta\in\overline{\Gamma_m(\omega)}$ be an $(\omega,m)$-semi-positive form on $X$.
In this section, we will develop the capacity theory for such $\beta$.
\begin{definition}\cite[Section 4]{KN25a} \label{def: local cap}
Let $(\Omega,\omega)$ be a bounded domain equipped with a Hermitian metric in $\mathbb{C}^n$ and let $E\subset\Omega$ be a Borel subset of $\Omega$, we define the local capacity of $E$ with respect to $\Omega$ as 
    $$\operatorname{Cap}_{m,\omega}(E,\Omega):= \sup\left\{ \int_{E} (dd^c u)^m \wedge \omega^{n-m}, u\in \operatorname{SH}_{m}(\omega,\Omega), -1\leq u\leq 0\right\}.$$
    We remark that when there is no confusion of notations, we will always omit the subscript $\omega$ and writing the local capacity $\operatorname{Cap}_{m,\omega}$ as $\operatorname{Cap}_m$.
\end{definition}

\begin{definition}\label{def of cap_beta}
    Let $E$ be a Borel subset of $X$ and $\beta\in\overline{\Gamma_m(\omega)}$. We define
    $$\operatorname{Cap}_{\beta,m,\omega}(E):=\sup \left\{\int_{E} H_{\beta,m}(u)| u\in \operatorname{SH}_{m}(X, \beta, \omega), -1 \leq u\leq 0\right\},$$
     Note that since $\beta\in\overline{\Gamma_m(\omega)}$, the set $\{u\in \operatorname{SH}_{m}(X, \beta, \omega), -1 \leq u\leq 0\}$ is non-empty.
\end{definition}

We list some elementary properties of $\operatorname{Cap}_{\beta,m,\omega}$ whose proofs are quite straightforward:

\begin{proposition}\label{properties of Cap_m}
The following properties hold:
     \begin{enumerate}
        \item If $E_1\subset E_2\subset X$ are Borel subsets of $X$, then $$0\leq \operatorname{Cap}_{\beta,m,\omega}(E_1)\leq \operatorname{Cap}_{\beta,m,\omega}(E_2)\leq \operatorname{Cap}_{\beta,m,\omega}(X)\leq C.$$ Where $C=C(\beta,\omega,X)$ is a uniform constant. 
        \item If $E_j$ are Borel subsets of $X$, then $\operatorname{Cap}_{\beta,m,\omega}(\cup_j E_j)\leq \sum_j\operatorname{Cap}_{\beta,m,\omega}(E_j)$. Moreover, if $E_j\subset E_{j+1}$ is an increasing sequence of Borel subsets of $X$, then $$\operatorname{Cap}_{\beta,m,\omega}(\cup_jE_j)=\lim_j\operatorname{Cap}_{\beta,m,\omega}(E_j).$$
        \item If $\beta_1,\beta_2\in\overline{\Gamma_m(\omega)}$ are such that $\beta_2-\beta_1\in\overline{\Gamma_m(\omega)}$, then $$\operatorname{Cap}_{\beta_1,m,\omega}(\cdot)\leq \operatorname{Cap}_{\beta_2,m,\omega}(\cdot).$$ For all $A\geq1$, we have $$\operatorname{Cap}_{\beta,m,\omega}(\cdot)\leq \operatorname{Cap}_{A\beta,m,\omega}(\cdot)\leq A^m\operatorname{Cap}_{\beta,m,\omega}(\cdot).$$
        \item If $\beta_1,\beta_2\in\Gamma_m(\omega)$ are two strictly $(\omega,m)$-positive forms, then there exists a constant $C\geq1$ such that $$
        C^{-1}\operatorname{Cap}_{\beta_1,m,\omega}(\cdot)\leq \operatorname{Cap}_{\beta_2,m,\omega}(\cdot)\leq C\cdot \operatorname{Cap}_{\beta_1,m,\omega}(\cdot).$$ In particular, even when $\beta\in\Gamma_m(\omega)$ and $\omega^{\prime}$ is another Hermitian metric on $X$, the two capacities $\operatorname{Cap}_{\beta,m,\omega}(\cdot)$ and $\operatorname{Cap}_{\omega^{\prime},m,\omega} (\cdot)$ are equivalent in the sense above.
    \end{enumerate}
\end{proposition}
\begin{proof}
    We give a proof of (iv) for example. Since $\beta_1\in\Gamma_m(\omega)$, we may find a large constant $A>1$ such that $A\beta_1-\beta_2\in\Gamma_m(\omega)$. Indeed, this is equivalent to say
    $$
(A\beta_1-\beta_2)^k\wedge\omega^{n-k}>0,\quad1\leq k\leq m.
    $$
    But $(A\beta_1-\beta_2)^k\wedge\omega^{n-k}=A^k\beta_1^k\wedge\omega^{n-k}+...$ is a polynomial in $A$ with strictly positive and smooth monic coefficient $\beta_1^k\wedge\omega^{n-k}$ for each $k$. Since the manifold $X$ is compact, the constant $A$ can be derived. It thus follows from (ii) and (iii) that
    $$
\operatorname{Cap}_{\beta_2,m,\omega}(\cdot)\leq \operatorname{Cap}_{A\beta_1,m,\omega}(\cdot)\leq A^m\operatorname{Cap}_{\beta_1,m,\omega}(\cdot),
    $$
    and we get the inequality on the right hand side. The left hand side is similar.
    
\end{proof}

Following the statement in \cite[Page 238,239]{GZ17}, we introduce another capacity on $X$, which will be useful when we prove the volume-capacity estimate for $\beta\in\Gamma_m(\omega)$:
\begin{definition}
    Fix a finite open double covering $U_j\subset\subset U_j^{\prime}\subset X$ of open coordinate balls such that $\{U_j\}_j$ covers $X$. For any Borel set $E\subset X$, we define
    $$
\operatorname{Cap}_{BT,m}(E):=\sum_j\operatorname{Cap}_m(E\cap U_j,U_j^{\prime}).
    $$
    Where $\operatorname{Cap}_m(\cdot)$ is the local capacity defined in \cref{def: local cap}.
\end{definition}
Thanks to \cite[Lemma 2.8]{KN25a},the gluing construction \cite[Proposition 1.30]{GZ17} also works for $(\omega,m)$ subharmonic function. The proof of the following equivalence theorem may now proceed directly via \cite[Proposition 9.8]{GZ17}.

\begin{proposition}\label{comparison of capacity}
    For each $\beta\in\overline{\Gamma_m(\omega)}$, there exists $C>1$ such that
    $$
\operatorname{Cap}_{\beta,m,\omega}(\cdot)\leq C\cdot \operatorname{Cap}_{BT,m}(\cdot).
    $$
    When $\beta\in\Gamma_m(\omega)$, the reverse inequality holds, i,e., there exists $C^\prime>1$ such that
    $$
\operatorname{Cap}_{BT,m}(\cdot)\leq C^\prime \operatorname{Cap}_{\beta,m,\omega}(\cdot).
    $$
\end{proposition}

\begin{theorem}[{\cite[Theorem 4.9]{KN25a} \cite[Proposition 5.5]{KN25b}}]
Let $\chi$ be a smooth form admitting a $(\chi,\omega,m)$-subharmonic potential, then every $\varphi\in \operatorname{SH}_{m}(X,\beta,\omega)$ is quasi-continuous with respect to $\operatorname{Cap}_{\omega,m,\omega}(\cdot)$, i.e., for every $\varepsilon>0$, there exists an open set $U\subset X$ such that $\operatorname{Cap}_{\omega,m,\omega}(U)<\varepsilon$ and $\varphi$ is continuous on $X-U$. 
\end{theorem}

\begin{definition}
A sequence of Borel functions $f_j$ is said to converge to $f$ in capacity if for each fixed $\delta>0$, we have  $$\lim_{j\rightarrow +\infty} \operatorname{Cap}_{\omega,m,\omega}\left( \left\{|f_j-f|>\delta \right \}\right)=0.$$
\end{definition}
\begin{remark}
    Thanks to \cref{properties of Cap_m} (iv), for any $\beta\in\Gamma_m(\omega)$, $\operatorname{Cap}_{\beta,m,\omega}(\cdot)$ is equivalent to $\operatorname{Cap}_{\omega,m,\omega}(\cdot)$. Therefore, we can use either $\operatorname{Cap}_{\beta,m,\omega}(\cdot)$ or $\operatorname{Cap}_{\omega,m,\omega}(\cdot)$ to define convergence in capacity.
\end{remark}
We give a weak convergence theorem which will be useful later.

\begin{proposition}\label{prop: monotone_capcity convergence}
   Let $\varphi_j\in \operatorname{SH}_{m}(X,\beta, \omega)$ be a decreasing (resp. increasing) sequence converging (resp. almost everywhere) to a function $\varphi\in \operatorname{SH}_{m}(X,\beta, \omega)\cap L^\infty(X)$, then $\varphi_j$ converges to $\varphi$ in capacity.
\end{proposition}
\begin{proof}
  By \cref{comparison of capacity}, the result is local, hence reduced to the case of \cite[Corollary 4.11]{KN25a}.
\end{proof}

\begin{lemma} \label{lem: dominate by Capacity}
    Let $v_1, \cdots, v_m$ be bounded  $(\omega,m)$-subharmonic functions in a domain $\Omega\subset C^n$ such that $0\leq v_i\leq M$. Set $\mu=dd^c v_1 \wedge \cdots \wedge dd^c v_m \wedge \omega^{n-m}$, then there exists a uniform constant $$\mu(\cdot) \leq C\cdot \operatorname{Cap}_{m}(\cdot).$$ 
\end{lemma}
\begin{proof}
    Observe that 
    $$dd^c v_1 \wedge \cdots \wedge dd^c v_m \wedge \omega^{n-m} \leq \left[dd^c (v_1+\cdots+v_m)\right]^m \wedge \omega^{n-m},$$
    thus we have 
    $$ 
    \mu(\cdot)\leq (mM)^m \operatorname{Cap}_{m}(\cdot)
    $$
     by  \cref{def: local cap}.
\end{proof}

 \begin{lemma}\label{lem: weak convergence1}
     Let $\mathcal{P}$ be a locally uniformly bounded family of $(\omega,m)$-subharmonic functions in $\Omega$. Fix $0\leq p \leq m$, let $\mathcal{T}$ be the set of currents of the form $T:= \left( \bigwedge_{1\leq i \leq p} dd^c u_i \right) \wedge \omega^{n-m}$, where $u_1, \cdots, u_p \in \mathcal{P}$. Let $(T_j)_{j\in \mathbb{N}}$ be a sequence of currents in $\mathcal{P}$ converging weakly to $T\in \mathcal{T}$. Then, for any locally bounded quasi-continuous function f in $\Omega$, $f\ T_j\rightarrow f\ T$ in the sense of currents of order $0$. 
 \end{lemma}
\begin{proof}
The proof is adapted from \cite[Lemma 4.21]{GZ17}.

Let $\Theta$ be a positive continuous test form of bidegree $(m-p,m-p)$ in $\Omega$ with compact support $K\subset\subset \Omega$, fix $\varepsilon>0$ a small constant.
From the quasi-continuity of $f$ there is an open subset $G\subset \Omega$ with $\operatorname{Cap}_{m}(G;\Omega)< \varepsilon$ such that 
$f|_{K-G}$ is continuous. Let $g$ be a continuous function in $\Omega$ with compact support such that $g=f$ on $K-G$, the existence of such $g$ can be deduced from the Tieztze's extention lemma.

Then
$$
\int_{\Omega} f\left(T_j \wedge \Theta-T \wedge \Theta\right)=\int_{\Omega} g\left(T_j \wedge \Theta-T \wedge \Theta\right)+\int_G(f-g)\left(T_j \wedge \Theta-T \wedge \Theta\right) .
$$

Observe that $\lim _{j \rightarrow+\infty} \int_{\Omega} g\left(T_j \wedge \Theta-T \wedge \Theta\right)=0$ since $T_j$ converges weakly towards $T$. We claim that the second term is $O(\varepsilon)$. Indeed, since $|f-g|$ is bounded by a constant $M>0$ on the support of $\Theta$,
$$
\left|\int_G(f-g)\left(T_j \wedge \Theta-T \wedge \Theta\right)\right| \leq M \int_G\left(T_j \wedge \Theta+T \wedge \Theta\right).
$$

Observe that $T_j \wedge \Theta \leq C_1 T_j \wedge \beta^{m-p}$ for some $C_1>0$, where $\beta=d d^{\mathrm{c}}|z|^2$. Moreover, there exists a uniform constant $C_2>0$ such that for all Borel subsets $E \subset \Omega$,
$$
\int_E \left(  \bigwedge_{1 \leq i \leq p}d d^{\mathrm{c}} v_i \right)\wedge \omega^{n-m} \wedge \beta^{m-p} \leq C_2 \operatorname{Cap}_m(E, \Omega)
$$
as $v_1, \ldots, v_p \in \operatorname{SH}_m(\omega,\Omega) \cap L_{\text {loc }}^{\infty}(\Omega)$ with $\left|v_i\right|\leq A$ (see \cref{lem: dominate by Capacity})  . Therefore,
$$
\int_G T_j \wedge \Theta+T \wedge \Theta \leq 2 C_1 C_2 \varepsilon
$$
and the proof is complete.

\end{proof}

\begin{lemma}\label{lem: convergence in loc}
    Let $\left(f_j\right)_{j \geq 0}$ be a sequence of (locally) uniformly bounded Borel functions converging in capacity to a  (locally) bounded Borel function $f$ in $\Omega$. Then $f_j \rightarrow f$ in $L_{\text {loc }}^1(\Omega)$.
\end{lemma}
\begin{proof}
The proof follows from \cite[Lemma 4.24]{GZ17}.
    Fix a compact set $K \subset \Omega$ and $\delta>0$. Since the sequence $\left(f_j\right)_{j \in \mathbb{N}}$ and $f$ are uniformly bounded in $K$ (by some $M>0$ ), we get
$$
\int_K\left|f_j-f\right| d \lambda \leq 2 M \lambda\left(K \cap\left\{\left|f_j-f\right|>\delta\right\}\right)+\delta \lambda(K) .
$$

Since the Lebesgue measure $\lambda$ is dominated by $\operatorname{Cap}_{m}(\cdot)$ (from \cref{lem: dominate by Capacity} if we take $v_1=\cdots=v_m=|z|^2$), the first term in the right-hand side converges to $0$. The claim follows since $\delta>0$ can be chosen arbitrarily small.
\end{proof}

\begin{theorem} \label{thm: weak convergence}
     Let $U \subset \mathbb{C}^n$ be an open set. Suppose $\left\{f_j\right\}_j$ are uniformly bounded quasi-continuous functions converging in capacity to a bounded quasi-continuous function $f$ on U. Let $\left\{u_1^j\right\}_j,\left\{u_2^j\right\}_j, \ldots,\left\{u_m^j\right\}_j$ be uniformly bounded  $(\omega,m)$- subharmonic  functions on $\Omega$, converging in capacity to bounded $(\omega,m)$ subharmonic functions  $ u_1, u_2, \ldots, u_m$ respectively. Then we have the following weak convergence of measures:
$$
f_j  dd^c u_1^j \wedge dd^c u_2^j \wedge \ldots \wedge dd^c u_m^j \wedge \omega^{n-m}\rightarrow f dd^c  u_1 \wedge dd^c u_2 \wedge \ldots \wedge dd^c u_m \wedge \omega^{n-m} .
$$
\end{theorem}

\begin{proof}
    The proof is adapted from \cite[Proposition 4.26]{GZ17}.
We proceed by induction. It follows from \cref{lem: convergence in loc}    that $f_j$ (resp. $u_j$ ) converges in $L_{\text {loc }}^1(\Omega)$ to $f$ (resp. $u$ ). If $p=1$, we know that $d d^{\mathrm{c}} u_j^1 \rightarrow d d^{\mathrm{c}} u^1$ weakly. Using induction on $p$ and setting
$$
T:=\bigwedge_{1 \leq i \leq p} d d^{\mathrm{c}} u^i \wedge \omega^{n-m}, \quad T_j:=\bigwedge_{1 \leq i \leq p} d d^{\mathrm{c}} u_j^i \wedge \omega^{n-m}.
$$
Clearly it suffices to prove that if $T_j \rightarrow T$ weakly, then $f_j T_j \rightarrow f T$ weakly. 
The statement is local, so we can assume that $\Omega$ is the unit ball and all the functions are bounded between $-1$ and $0$ . Fix a positive  test form $\Theta$ of bidegree $(m-p, m-p)$ and observe that
$$
\int_{\Omega} f_j T_j \wedge \Theta-\int_{\Omega} f T \wedge \Theta=\int_{\Omega}\left(f_j-f\right) T_j \wedge \Theta+\int_{\Omega} f\left(T_j \wedge \Theta-T \wedge \Theta\right)=I_j+J_j.
$$
It follows from \cref{lem: weak convergence1} that $\lim _j J_j=0$. Thus, it remains to prove that $\lim _j I_j=0$. Fix $\delta>0$ small, let $K$ be the support of $\Theta$ and set $E_j:=K \cap\left\{\mid f_j-\right.$ $f \mid \geq \delta\}$. Then
$$
\left|I_j\right| \leq \int_{\Omega}\left|f_j-f\right| T_j \wedge \Theta \leq \int_{K \cap E_j}\left|f_j-f\right| T_j \wedge \Theta+\delta \int_K T_j \wedge \Theta.
$$
It follows from previous estimates in the proof of \cref{lem: weak convergence1} that
$$
\left|I_j\right| \leq C^{\prime} \operatorname{Cap}_{m}\left(E_j, \Omega\right)+\delta M(K),
$$
where $M(K):=C \sup _j \int_K T_j \wedge \beta^{n-p}$ is finite by the Chern-Levine-Nirenberg inequalities (see \cite[Proposition 3.7]{KN25a}). Hence, $\lim _j I_j=0$.
\end{proof}

\begin{theorem}\label{thm: monotonicity}
Let $\beta^j, j \in\{1, \ldots, m\}$ be smooth real $(1,1)$-forms on $X$ admitting bounded potentials $\rho_j\in\operatorname{SH}_m(X,\beta_j,\omega)\cap L^\infty(X)$ respectively. Suppose that for all $j \in\{1, \ldots, n\}$ we have  bounded functions (and  uniformly bounded from above with respect to $k$)  $u_j, u_j^k \in \operatorname{SH}_{m}(X,\beta^{j},\omega)$ such that $u_j^k \rightarrow u_j$ in $\operatorname{Cap}_{\omega,m,\omega}(\cdot)$ as $k \rightarrow \infty$. Let $\chi_k, \chi \geq 0$ be quasi-continuous and uniformly bounded functions such that $\chi_k \rightarrow \chi$ in $\operatorname{Cap}_{\omega,m,\omega}(\cdot)$. Then
$$
\liminf _{k \rightarrow \infty} \int_X \chi_k \beta_{u_1^k}^1 \wedge \ldots \wedge \beta_{u_m^k}^m  \wedge \omega^{n-m}    \geq 
\int_X \chi \beta_{u_1}^1 \wedge \ldots \wedge \beta_{u_m}^m \wedge \omega^{n-m}.
$$
\end{theorem}

Proof. 
The proof follows the idea of \cite[Theorem 2.6]{DDL25}.
We may assume without loss of generality that $u_j>\rho_j$ up to shifting a constant.
Fix $\varepsilon>0$, consider
$$
f_j^{k, \varepsilon}:=\frac{\max \left(u_j^k-\rho_j, 0\right)}{\max \left(u_j^k-\rho_j, 0\right)+\varepsilon}, j=1,\; \ldots, m, k \in \mathbb{N}^*
$$
and
$$
\bar{u}_j^{k}:=\max \left(u_j^k, \rho_j\right).
$$

Observe that for $j$ fixed, the functions $\bar{u}_j^{k} \geq \rho_j$ are uniformly bounded in $X$  and converge in $\operatorname{Cap}_{\omega,m,\omega}(\cdot)$ to $\bar{u}_j:=\max(u_j,\rho_j)=u_j$ as $k \rightarrow \infty$. Moreover, $f_j^{k,  \varepsilon}=0$ if $u_j^k \leq \rho_j$. By \cref{prop: max principle for beta},  we can write
$$
f^{k, \varepsilon} \chi_k \beta_{u_1^k}^1 \wedge \ldots \wedge \beta_{u_m^k}^m \wedge \omega^{n-m}=f^{k,  \varepsilon} \chi_k \beta_{\bar{u}_1^{k}}^1 \wedge \ldots \wedge \beta_{\bar{u}_m^{k}}^m \wedge \omega^{n-m}.
$$
where $f^{k, \varepsilon}=f_1^{k, \varepsilon} \cdots f_m^{k, \varepsilon}$. For each $ \varepsilon$ fixed the functions $f^{k, \varepsilon}$ are quasi-continuous, uniformly bounded (with values in $[0,1]$ ) and converging in capacity to $f^{ \varepsilon}:=f_1^{ \varepsilon} \cdots f_n^{ \varepsilon}$, where $f_j^{ \varepsilon}$ is defined by
$$
f_j^{\varepsilon}:=\frac{\max \left(u_j-\rho_j, 0\right)}{\max \left(u_j-\rho_j, 0\right)+\varepsilon}.
$$

With the information above we can apply  \cref{thm: weak convergence} to get that
$$
f^{k,  \varepsilon} \chi_k \beta_{u_1^{k}}^1 \wedge \ldots \wedge \beta_{u_m^{k}}^m \wedge \omega^{n-m} \rightarrow f^{ \varepsilon} \chi \beta_{\bar{u}_1}^1 \wedge \ldots \wedge \beta_{\bar{u}_m}^m \wedge \omega^{n-m} \text { as } k \rightarrow \infty,
$$
in the weak sense of measures on $X$. In particular, since $0 \leq f^{k, \varepsilon} \leq 1$ we have that
$$
\begin{aligned}
\liminf _{k \rightarrow \infty} \int_X \chi_k \beta_{u_1^k}^1 \wedge \ldots \wedge \beta_{u_m^k}^m \wedge \omega^{n-m} & \geq \liminf _{k \rightarrow \infty} \int_X f^{k, \varepsilon} \chi_k \beta_{u_1^{k}}^1 \wedge \ldots \wedge \beta_{u_m^{k}}^m \wedge \omega^{n-m} \\
& \geq \int_X f^{\varepsilon} \chi \beta_{\bar{u}_1}^1 \wedge \ldots \wedge \beta_{\bar{u}_m}^m \wedge \omega^{n-m}.
\end{aligned}
$$

Now, letting $\varepsilon \rightarrow 0$  we conclude the proof.

\section{\texorpdfstring{$\operatorname{SH}_{m}(X, \beta, \omega)$-envelopes}{SH\_m(X, beta, omega)-envelopes}}

Let $(X,\omega)$ be a compact Hermitian manifold of complex dimension $n$ equipped with a Hermitian metric $\omega$. Let $m$ be a positive integer such that $1\leq m\leq n$.  We also use the notation $(\Omega, \omega)$ to denote a domain in $\mathbb{C}^n$ equipped with a Hermitian form $\omega$. Let $\{\beta\}\in BC^{1,1}(X)$ be a nef Bott-Chern class admitting a bounded Hessian potential $\rho\in \operatorname{SH}_m(X,\beta,\omega)\cap L^\infty(X)$.

In this section, we study the basic properties of the envelop functions, which will play an important role in subsequent studies.

\begin{definition}
    If $h: X \rightarrow \mathbb{R}$ is a bounded function, we set
$$
P_{\beta, m}(h):=\sup \{u \in S H_{m}(X, \beta, \omega), u \leq h  \quad\text{quasi-everywhere}\}^*,
$$ where "quasi-everywhere" means outside a $m$- polar set. Note that we of course have $P_{\beta,m}(h)\in\operatorname{SH}_m(X,\beta,\omega)$ thanks to \cref{usc regularization}.
\end{definition}

\begin{theorem}\label{thm: expression of envelopes}
If $h$ is bounded, then we have 
    $$P_{\beta, m}(h):=\sup \{u \in S H_{m}(X, \beta, \omega), u \leq h  \quad\text{quasi-everywhere}\}.$$
    In particular, we have monotonicity property of $P_{\beta,m}(\cdot)$, i.e., for bounded $h_j\searrow h$, we have $P_{\beta,m}(h_j)\searrow P_{\beta,m}(h)$.
\end{theorem}
 \begin{proof}
     Denote $\sup \{u \in S H_{m}(X, \beta, \omega), u \leq h  \quad\text{quasi-everywhere}\}$ as $G$,  use \cref{lem: negligible is polar} and the proof in \cite[Proposition 2.13]{PSW25}, we have $G=G^*$, whence the desired monotonicity property. 
 \end{proof}

\begin{lemma}\label{lem: balayage}
Let $\beta$ be a Hermitian metric and let $U\subset\overline{U}\subset \Omega$ be small balls in $C^n$. For any bounded $(\beta,\omega,m)$-subharmonic function $g$ on $\Omega$, we can find another $(\beta,\omega,m)$-subharmonic function $\tilde{g}$ on $\Omega$ such that $\tilde{g}\geq g$ on $\Omega$, $\tilde{g}= g$ on $D-U $, and $(\beta+dd^c \tilde{g})^m \wedge \omega^{n-m}=0$. Moreover, for $g_1\leq g_2$, we have $\tilde{g}_1\leq \tilde{g}_2$.
\end{lemma}
\begin{proof}
Assume first that $g \in C^0(\Omega)$. By \cite[Theorem 3.15]{GN18}, there exists a continuous $(\beta,\omega,m)$-subharmonic function $u$ such that $u=g$ on $\partial U $ and $\left(\beta+d d^c u\right)^m\wedge \omega^{n-m}=0$ on $U$. Set
$$
\widetilde{g}=\left\{\begin{array}{lll}
u & \text { on } & U, \\
g & \text { on } & \Omega \backslash U.
\end{array}\right.
$$

By the comparison principle \cite[Corollary 3.11]{GN18}, we have $\widetilde{g}=u \geq g$ on $U$. Moreover, $\widetilde{g}$ is the decreasing limit of a sequence of $(\beta,\omega,m)$-subharmonic functions
$$
g_k= \begin{cases}\max \left\{u, g+\frac{1}{k}\right\} & \text { on } U\left(z_0, r\right), \\ g+\frac{1}{k} & \text { near } \Omega \backslash U\left(z_0, r\right),\end{cases}
$$
hence $\widetilde{g}$ is $(\beta,\omega,m)$-subharmonic. For the last statement, we observe that $u_1 \leq u_2$ for any $g_1\leq g_2$ and hence $\widetilde{g}_1 \leq \widetilde{g}_2$ by \cite[Corollary 3.11]{GN18} again. The proof is therefore concluded.

\end{proof}
\begin{remark}\label{rem: balayage}
    We have used the comparison principle and solvability of dirichlet problems established in \cite{GN18}. Note that although the Hermitian metric was assumed to be locally conformal K\"ahler there, the proof can be carried out exactly for a general $\omega$ thanks to the work \cite{KN25a}.  
\end{remark}
\begin{lemma}\label{lem: contact for herrmitian}
    Let $\beta$ be a Hermitian metric. If $h$ is a bounded  lower semi-continuous Lebesgue measurable function, then the Hessian measure $\left(\beta+d d^c P_{\beta,m}(h)\right)^m \wedge \omega^{n-m}$ puts no mass on the open set $\left\{P_{\beta,m}(h)<h\right\}$.
\end{lemma}
\begin{proof}
Fix  $x_0\in \left\{P_{\beta,m}(h)<h\right\}$ and $\varepsilon>0$ small, we choose a small ball $U$ containing $x_0$ such that $\forall x\in \bar{U}$, we have $$P_{\beta,m}(h)(x)<h(x_0)-\varepsilon \leq h(x_0)-\frac{\varepsilon}{2}<h(x).$$
Shrinking $U$ if necessary, we can furthermore choose a smooth potential $\phi$ such that $\beta\leq dd^c \phi$ and $0<\phi<\frac{\varepsilon}{2}$.
    By \cref{lem: balayage}, there exists $g\in\operatorname{SH}_m(X,\beta,\omega)\cap L^\infty(X)$ such that     $$
\begin{cases}g=P_{\beta,m}(h) & \text { on } X\backslash U, \\ g \geq  P_{\beta,m}(h) & \text { on } X,\\ (\beta+dd^c g)^m\wedge \omega^{n-m} = 0 & \text{ on } U.\end{cases}
$$
Fix a point $z\in\partial U$, we can write 
$$ (\phi+g)(z)=(\phi+P_{\beta,m}(h))(z) < h(x_0)-\varepsilon +\phi(z)\leq h(x_0)-\frac{\varepsilon}{2}\leq h(z),$$
The Comparison principle \cite[Corollary 6.2]{KN25a} yields that $\phi+g\leq h(x_0)-\frac{\varepsilon}{2}$ on $U$, thus $g(x)\leq h(x_0)-\frac{\varepsilon}{2}-\phi\leq h(x_0)-\frac{\varepsilon}{2}\leq h(x)$ on $U$. 

Now we have showed that $g\leq h$ quasi-everywhere on $X$ and hence $g\leq P_{\beta,m}(h)$ by the definition of the envelope, which implies that $g=P_{\beta,m}(h)$. This yields that $(\beta+dd^c P_{\beta,m}(h))^m \wedge \omega^{n-m}=0$ on $U$. Since $x_0$ is arbitrary, we get the desired result.
\end{proof}
\begin{lemma} \label{lem: contact for beta lsc}
    Let $\{\beta\}\in BC^{1,1}(X)$ be a nef Bott-Chern class admitting a bounded Hessian potential $\rho\in \operatorname{SH}_m(X,\beta,\omega)\cap L^\infty(X)$. If $h$ is a bounded  lower semi-continuous Lebesgue measurable function, then the Hessian measure $\left(\beta+d d^c P_{\beta,m}(h)\right)^m \wedge \omega^{n-m}$ puts no mass on the open set $\left\{P_{\beta,m}(h)<h\right\}$.
\end{lemma}
\begin{proof}
Since $\{\beta\}$ is nef, we may find for each $j$ a function $\rho_j\in \operatorname{PSH}(X,\beta+\frac{1}{j}\omega)\cap C^\infty(X)$ such that $\beta+\frac{1}{j}\omega+dd^c\rho_j$ is a Hermitian metric.
 Observe that 
 $$
 P_{\beta+\frac{1}{j}\omega,m}(h)= P_{\beta+\frac{1}{j}\omega+dd^c \rho_j,m}(h-\rho_j)+\rho_j
 $$
 and 
 $$
 P_{\beta+\frac{1}{j}\omega,m}(h)\searrow P_{\beta,m}(h),
 $$
as $j\rightarrow\infty$. Moreover, the existence of $\rho\in \operatorname{SH}_m(X,\beta,\omega)\cap L^\infty(X)$ ensures that the envelope $P_{\beta,m}(h)$ is bounded. By \cref{lem: contact for herrmitian} we have 
 \begin{align*}
       &\mathds{1}_{\{P_{\beta+ \frac{1}{j}\omega,m}(h)<h\}} \left(\beta+ \frac{1}{j} \omega+ d d^c P_{\beta+\frac{1}{j}\omega, m}(h)\right)^m \wedge \omega^{n-m} \\
       =&\mathds{1}_{\{P_{\beta+\frac{1}{j}\omega+dd^c \rho_j,m}(h-\rho_j)<h-\rho_j\}} \left(\beta+ \frac{1}{j} \omega+dd^c \rho_j+   d d^c P_{\beta+\frac{1}{j}\omega, m}(h-\rho_j)\right)^m \wedge \omega^{n-m}\\
    =&0.
 \end{align*}
Now we let $j\rightarrow +\infty$,  by the monotone convergence theorem of Hessian measures with bounded potentials \cite[Lemma 5.1]{KN25a},  we know 
$$
\left(\beta+ \frac{1}{j} \omega+ d d^c P_{\beta+\frac{1}{j}\omega, m}(h)\right)^m \wedge \omega^{n-k} \rightarrow (\beta+dd^c P_{\beta,m}(h))^{m} \wedge \omega^{n-m}.
$$
Since $\{P_{\beta+\frac{1}{k} \omega,m}(h)< h  \}$ is open, by semi-continuity of the weak convergence,  we have for $\forall k\in \mathbb{N}^*$,
\begin{align*}
   & \int_{ \{P_{\beta+\frac{1}{k} \omega,m}(h)< h  \}}  (\beta+dd^c P_{\beta,m}(h))^{m} \wedge \omega^{n-m}\\
 &\leq \liminf_{j\rightarrow +\infty}\int_{\{P_{\beta+\frac{1}{k} \omega,m}(h)< h  \}} \left(\beta+ \frac{1}{j} \omega+ d d^c P_{\beta+\frac{1}{j}\omega, m}(h)\right)^m \wedge \omega^{n-m} \\
 =& 0.
\end{align*}
We can finally conclude the proof by letting $k\rightarrow\infty$.
\end{proof}

\begin{theorem}\label{contact}
 If $h$ is a quasi-continuous bounded function, then the complex Hessian measure $\left(\beta+d d^c P_{\beta, m}(h)\right)^m \wedge \omega^{n-k}$ is concentrated on the contact set $\mathcal{C}_m=\left\{P_{\beta, m}(h)=h\right\}$.
\end{theorem}
\begin{proof}
The proof follows the idea of \cite[Proposition 2.5]{GLZ19}.
By definition of quasi-continuity, there exists a sequence of compact sets $\left(K_l\right)$ such that $\operatorname{Cap}_{\omega,m,\omega}\left(X \backslash K_l\right) \leq 2^{-l}$ and the restriction $\left.h\right|_{K_l}$ is a continuous function on $K_l$. By taking $\widetilde{K}_j:=\cup_{1 \leq l \leq j} K_l$, we can assume the sequence  $(K_j)$  is increasing. Using the Tietze-Urysohn Lemma, there exists a continuous function $H_j$ on $X$ such that $\left.H_j\right|_{K_j}=\left.h\right|_{K_j}$ and $H_j$ shares the same bounds as $h$.

Set
$$
h_j:=\sup \left\{H_l \mid l \geq j\right\}
$$

Then $\left\{h_j\right\}_j$ is lower semi-continuous, quasi-continuous, and uniformly bounded. Indeed, fix $j$ and $k \geq j$,
$$
\left.h_j\right|_{K_k}=\max \left\{\sup \left\{H_l \mid l \geq k\right\}, \max \left(H_j, \ldots, H_{k-1}\right)\right\}=\max \left\{h, \max \left(H_j, \ldots, H_{k-1}\right)\right\},
$$
which implies that $\left.h_j\right|_{K_k}$ is continuous and hence the quasi-continuity of $h$. 

We claim that $h_j$ converges to $h$ in capacity. Here we sketch the proof. Since $h_j=h$ on $K_j, \forall \delta>0$, we have $\left\{h_j \geq h+\delta\right\} \subseteq X \backslash K_j$, thus $\operatorname{Cap}_{\omega,m,\omega}\left(\left\{h_j \geq h+\delta\right\}\right) \leq$ $\operatorname{Cap}_{\omega,m,\omega}\left(X \backslash K_j\right) \leq 2^{-j}$, thus $h_j$ converges to $h$ in capacity.

Write $\hat{h}_j:= P_{\beta,m}(h_j)$,  we have $\hat{h}_j\searrow \hat{h}:=P_{\beta,m}(h)$ by \cref{thm: expression of envelopes} and thus $\hat{h}_j$ converges to $\hat{h}$ in capacity by \cref{prop: monotone_capcity convergence}. \cref{lem: contact for beta lsc} then yields that $$\int_{X} \frac{\max ( -\hat{h}_j+h_j,0)}{ \max (- \hat{h}_j+h_j,0) +\varepsilon   } (\beta+dd^c \hat{h}_j)^m \wedge \omega^{n-m} =0 .$$
 Now we use \cref{thm: monotonicity} to get
 $$ 
 \int_{X} \frac{\max ( -\hat{h}+h,0)}{ \max (- \hat{h}+h,0) +\varepsilon   } (\beta+dd^c \hat{h})^m \wedge \omega^{n-m} =0.
 $$
 Finally, we conclude the proof by letting $\epsilon\rightarrow0^+$.

\end{proof}
The following corollary is a direct consequence of \cref{contact}:

\begin{corollary}\label{mass concentration 2}
    Let $u,v$ be bounded $(\beta,m,\omega)$-sh functions and let $w:=P_{\beta,m}(u,v)=P_{\beta,m}(\min(u,v))$ be the rooftop envelope of $u,v$. Then
    $$
(\beta+dd^cw)^m\wedge\omega^{n-m}\leq\mathds{1}_{\{w=u<v\}}(\beta+dd^cu)^m\wedge\omega^{n-m}+\mathds{1}_{\{w=v\}}(\beta+dd^cv)^m\wedge\omega^{n-m}.
    $$
\end{corollary}

\begin{corollary}\label{contact3}
    Fix $\lambda\geq0$ and let $u,v\in \operatorname{SH}_m(X,\beta,\omega)\cap L^{\infty}(X)$. Fix two smooth $(1,1)$-forms $\beta_1,\beta_2$ such that $\beta_1\geq\beta$ and $\beta_2\geq\beta$. If $H_{\beta_1,m}(u)\leq e^{\lambda u}f\omega^n$ and $H_{\beta_2,m}(v)\leq e^{\lambda v}g\omega^n$, then
    $$
H_{\beta,m}(P_{\beta,m}(u,v))\leq e^{\lambda P_{\beta,m}(u,v)}\max(f,g)\omega^n.
    $$
\end{corollary}
\begin{proof}
    Set $\varphi:=P_{\beta,m}(u,v)$. Then we have
    \begin{align*}
        H_{\beta,m}(\varphi)&\leq\mathds{1}_{\{\varphi=u<v\}} H_{\beta,m}(\varphi)+\mathds{1}_{\{\varphi=v\}} H_{\beta,m}(\varphi)\\
        &\leq\mathds{1}_{\{\varphi=u<v\}} H_{\beta_1,m}(\varphi)+\mathds{1}_{\{\varphi=v\}} H_{\beta_2,m}(\varphi)\\
        &\leq\mathds{1}_{\{\varphi=u<v\}} H_{\beta_1,m}(u)+\mathds{1}_{\{\varphi=v\}} H_{\beta_2,m}(v)\\
        &\leq\mathds{1}_{\{\varphi=u<v\}} e^{\lambda u}f\omega^n+\mathds{1}_{\{\varphi=v\}} e^{\lambda v}g\omega^n\\
        &\leq e^{\lambda P_{\beta,m}(u,v)}\max(f,g)\omega^n.
    \end{align*}
    Where in the first inequality we have used \cref{contact} and \cref{cor of max principle} in the second inequality.
\end{proof}
\section{The comparison principle and the domination principle} \label{subsec: comparison and domination}

Let $(X,\omega)$ be a compact Hermitian manifold of complex dimension $n$ equipped with a Hermitian metric $\omega$. Let $m$ be a positive integer such that $1\leq m\leq n$.

In this section we prove the modified comparison principle and the domination principle in the Hessian setting, the method is adapted mainly from \cite{GL22} and \cite{SW25}. Firstly, we introduce an analogous concept of non-collapsing forms as in \cite{GL22}.

\begin{definition} \label{def: non-collapsing}
    Let $(X,\omega)$ be a compact Hermitian manifold and $\beta$ be a (possibly non-closed) smooth $(1,1)$-form with a bounded $(\beta,\omega,m)$-subharmonic potential $\rho$. We say that $\beta$ is $(\omega,m)$-non-collapsing if for any $u\in \operatorname{SH}_m(X,\beta,\omega)\cap L^\infty(X)$, the complex Hessian measure $H_{\beta,m}(u)=\beta_u^m\wedge\omega^{n-m}$ has positive global mass:
    $$
\int_X(\beta+dd^cu)^m\wedge\omega^{n-m}>0.
    $$
    
\end{definition}
Unless otherwise stated, we will assume throughout this section that $\beta$ is $(\omega,m)$-non-collapsing.
\begin{proposition}\label{domination for beta}
Assume $\beta$ be as above and fix a constant $0\leq c<1$. Assume $u,v\in \operatorname{SH}_m(X,\beta,\omega)\cap L^{\infty}(X)$ satisfies $\mathds{1}_{\{u<v\}}H_{\beta,m}(u)\leq c\mathds{1}_{\{u<v\}}H_{\beta,m}(v)$, then $u\geq v$.
\end{proposition}

\begin{proof}
    Fix an arbitrary positive constant $\delta>0$, we need to prove that $u\geq v-\delta$. Assume by contradiction that the set $E:=\{u<v-\delta\}$ is non-empty and hence has positive Lebesgue measure, since $u,v$ are quasi-$(\omega,m)$-sh, $E$ is quasi-open. For $b>1$, set
    $$
u_b:=P_{\beta,m}(bu-(b-1)v).
    $$

    It follows from \cref{contact} that $H_{\beta,m}(u_b)$ is concentrated on the contact set $D:=\{u_b=bu-(b-1)v\}$. Since $b^{-1}u_b+(1-b^{-1})v\leq u$ with equality on $D$, by the maximum principle  (\cref{max principle}) we get
    \begin{align*}
        &\mathds{1}_DH_{\beta,m}(u)\geq\mathds{1}_DH_{\beta,m}(b^{-1}u_b+(1-b^{-1})v)\\
        \geq&\mathds{1}_Db^{-m}H_{\beta,m}(u_b)+\mathds{1}_D(1-b^{-1})^mH_{\beta,m}(v)\\
        \geq&\mathds{1}_Db^{-m}H_{\beta,m}(u_b)+\mathds{1}_DcH_{\beta,m}(v),
    \end{align*}
    if we take $b$ sufficiently large. By our condition we have $\mathds{1}_{D\cap\{u<v\}}H_{\beta,m}(u_b)=0$. 

    Clearly $u_b$ is bounded. Since $\beta$ is $(\omega,m)$-non-collapsing, we know that the mass of $H_{\beta,m}(u_b)$ on $D$ is positive and hence the set $D\cap\{u\geq v\}$ is non-empty. On this set,
    $$
u_b=bu-(b-1)v\geq u,
    $$
thus the  $\sup_Xu_b$ is uniformly bounded from below since $u$ is assumed to be bounded. The sequence $u_b-\sup_Xu_b$ converges in $L^1$ and almost everywhere to a function $u_{\infty}\in \operatorname{SH}_m(X,\beta,\omega)$ by \cref{L^1 compactness} and hence $u_{\infty}$ must be identically $-\infty$ on $E$ with positive Lebesgue measure, this yields a contradiction that  $u_{\infty}\notin \operatorname{SH}_m(X,\beta,\omega)$. The proof is completed.
    
\end{proof}

\begin{corollary}
\label{domination for beta 2}
    Let $\beta,u,v$ be as in \cref{domination for beta}, if
    $$
e^{-\lambda u}H_{\beta,m}(u)\leq e^{-\lambda v}H_{\beta,m}(v),
    $$
    then $u\geq v$.
\end{corollary}
\begin{proof}
    For any fixed $\delta>0$, we have on the set $\{u<v-\delta\}$, we have
    $$
H_{\beta,m}(u)\leq e^{\lambda(u-v)}H_{\beta,m}(v) \leq e^{-\lambda\delta}H_{\beta,m}(v-\delta).
    $$
    It then follows from \cref{domination for beta} that $v-\delta\leq u$ and hence $v\leq u$ by letting $\delta\rightarrow0$.
    
\end{proof}

\begin{corollary}\label{domination for beta 3}
Let $\beta,u,v$ be as in \cref{domination for beta}, if
$$
H_{\beta,m}(u)\leq cH_{\beta,m}(v)
$$
for some constant $c>0$, then $c\geq 1$.
\end{corollary}
\begin{proof}
Assume by contradiction that $c<1$. For any constant $C>0$, we have
$$
\mathds{1}_{\{u<v+C\}}H_{\beta,m}(u)\leq c\mathds{1}_{\{u<v
+C\}}H_{\beta,m}(v+C).
$$
Using \cref{domination for beta} we deduce that $u\geq v+C$. Since $C$ was chosen arbitrarily and $u,v$ are bounded, this leads to a contradiction.

\end{proof}
Our next goal of this section is to establish a modified comparison principle for the degenerate complex Hessian equations, which is an extension of \cite[Theorem 3.7]{KN16}. 

By the definition, there exists a constant $B=B(\beta,\omega)>0$ satisfies, on $X$,
$$
-B\omega^2\leq dd^c\omega\leq B\omega^2,\quad-B\omega^3\leq d\omega\wedge d^c\omega\leq B\omega^3,
$$
$$
-B\omega^2\leq dd^c\beta\leq B\omega^2,\quad-B\omega^3\leq d\beta\wedge d^c\beta\leq B\omega^3.
$$
and
$$
-B\omega^3\leq d\beta\wedge d^c\omega\leq B\omega^3,\quad-B\omega^3\leq d\omega\wedge d^c\beta\leq B\omega^3.
$$
\begin{theorem}\label{modified comparison principle}
 Let $\beta\in\Gamma_m(\omega)$ be an $(\omega,m)$-positive form and let $u,v\in \operatorname{SH}_m(X,\beta,\omega)\cap L^{\infty}(X)$. Fix $0<\epsilon<1$, set $S(\epsilon):=\inf_X[u-(1-\epsilon)v]$ and $U(\epsilon,s):=\{u<(1-\epsilon)v+S(\epsilon)+s\}$ for $s>0$. Then, there exists a uniform constant $C=C(m,n,B,\omega)$ such that for any $0<s<\frac{(\lambda\epsilon)^3}{2C}$,
    $$
\left(1-\frac{Cs}{(\lambda\epsilon)^{3}}\right)^m\int_{U(\epsilon,s)}(\beta+dd^c(1-\epsilon)v)^m\wedge\omega^{n-m}\leq\int_{U(\epsilon,s)}(\beta+dd^cu)^m\wedge\omega^{n-m}.
    $$
\end{theorem}
\begin{proof}
    The idea of the proof is inspired by \cite[Theorem 1.5]{GL22}. During the proof, we will use $C$ to denote various uniform constants depending only on $m,n,B,\omega$. Set $\phi:=\max(u,(1-\epsilon)v+S(\epsilon)+s)$, it is immediate to see that $(1-\epsilon)v\in \operatorname{SH}_m(X,\beta,\omega)$ and hence $\phi\in \operatorname{SH}_m(X,\beta,\omega)$. Set $T_k:=\beta_{u}^k\wedge\beta_{\phi}^{m-k}\wedge\omega^{n-m}$ for $0\leq k\leq m$. We also take the convention that $T_j=0$ for $j<0$. Our goal is to prove that
    \begin{equation}\label{eq induction}
        \left(1-\frac{Cs}{(\lambda\epsilon)^3}\right)\int_{U(\epsilon,s)}T_k\leq\int_{U(\epsilon,s)}T_{k+1},
    \end{equation}
    for $k=0,1,...,m-1$. We will prove by induction.

    Firstly, assume $k=0$. Since $u\leq\phi$, it follows from \cref{cor of max principle} that 
    $$
\mathds{1}_{\{u=\phi\}}\beta_{\phi}^m\wedge\omega^{n-m}\geq\mathds{1}_{\{u=\phi\}}\beta_{u}\wedge\beta_{\phi}^{m-1}\wedge\omega^{n-m}.
    $$
  Observe that $U(\epsilon,s)=\{u<\phi\}$, the above inequality implies that
    \begin{align*}
        \int_{U(\epsilon,s)}(T_0-T_1)&=\int_{U(\epsilon,s)}(\beta_{\phi}-\beta_{u})\beta_{\phi}^{m-1}\wedge\omega^{n-m}\\
        &\leq\int_X(\beta_{\phi}-\beta_{u})\beta_{\phi}^{m-1}\wedge\omega^{n-m}\\
        &=\int_X(\phi-u)dd^c\left(\beta_{\phi}^{m-1}\wedge\omega^{n-m}\right).
        \end{align*}
        We next estimate $dd^c(\beta_{\phi}^{m-1}\wedge\omega^{n-m})$:
        \begin{align*}
            dd^c(\beta_{\phi}^{m-1}\wedge\omega^{n-m})=&d\left(\pm(m-1)\beta_{\phi}^{m-2}\wedge d^c\beta\wedge\omega^{n-m}\pm(n-m)\beta_{\phi}^{m-1}\wedge\omega^{n-m-1}\wedge d^c\omega\right)\\
            =&\pm(m-1)\beta_{\phi}^{m-2}\wedge dd^c\beta\wedge\omega^{n-m}\pm(m-1)(m-2)\beta_{\phi}^{m-3}\wedge d\beta\wedge d^c\beta\wedge\omega^{n-m}\\
            &\pm(m-1)(n-m)\beta_{\phi}^{m-2}\wedge d\omega\wedge
            d^c\beta\wedge\omega^{n-m-1}\\
            &\pm(n-m)(n-m-1)\beta_{\phi}^{m-1}d\omega\wedge d^c\omega\wedge\omega^{n-m-2}\\
            &\pm(n-m)\beta_{\phi}^{m-1}\wedge dd^c\omega\wedge\omega^{n-m-1}\pm(n-m)(m-1)\beta_{\phi}^{m-2}\wedge d\beta\wedge            
d^c\omega\wedge\omega^{n-m}\\
            \leq&C\left(\sum_{j=1}^3\beta_{\phi}^{m-j}\wedge\omega^{n-m+j}\right).
            \end{align*}
        Where we have used \cite[Lemma 2.4]{KN16} and the constant $C$ depends on $\omega,m,n,B$. Note that although all the forms are smooth in \cite[Lemma 2.4]{KN16}, we can use \cite[Corollary 1.2]{GN18} to regularize and check the inequalities locally. The constant $C$ depends only on $m,n,T$ there (we use $T$ to denote various form such as $dd^c\beta\wedge\omega^{n-m}$), so the limiting process does not cause any problems.
            Recall that $\beta-\lambda\omega$ is again an $(\omega,m)$-positive form for some $0<\lambda\leq1$, we therefore have
            $$
(\lambda\epsilon)\beta_{\phi}^{m-j}\wedge\omega^{n-m+j}\leq\epsilon\,\beta_{\phi}^{m-j}\wedge\beta\wedge\omega^{n-m+j-1}\leq\beta_{\phi}^{m-j}\wedge\beta_{(1-\epsilon)v}\wedge\omega^{n-m+j-1}.
            $$
for each $1\leq j\leq3$. Repeating the argument $j$ times we get
$$
(\lambda\epsilon)^j\int_X(\phi-u)\beta_{\phi}^{m-j}\wedge\omega^{n-m+j}\leq\int_X(\phi-u)\beta_{\phi}^{m-j}\wedge\beta_{(1-\epsilon)v}^j\wedge\omega^{n-m}.
$$
Notice that $\phi-u$ vanishes outside $U(\epsilon,s)$ and $\beta_{\phi}^{m-j}\wedge\beta_{(1-\epsilon)v}^j\wedge\omega^{n-m}=\beta_{\phi}^m\wedge\omega^{n-m}$ on the set $U(\epsilon,s)$ by the maximum principle. It follows that
$$
    \int_{U(\epsilon,s)}(T_0-T_1)\leq Cs\sum_{j=1}^3(\lambda\epsilon)^{-j}\int_{U(\epsilon,s)}\beta_{\phi}^m\wedge\omega^{n-m}\leq\frac{Cs}{(\lambda\epsilon)^3}\int_{U(\epsilon,s)}T_0.
$$
This yields (\ref{eq induction}) for $k=0$.

Assume now that (\ref{eq induction}) holds for all $j\leq k-1$, we need to check that it still holds for $k$. The argument is similar: by the maximum principle we can write
 \begin{align*}
        \int_{U(\epsilon,s)}(T_k-T_{k+1})&=\int_{U(\epsilon,s)}(\beta_{\phi}-\beta_{u})\beta_{u}^k\wedge\beta_{\phi}^{m-k-1}\wedge\omega^{n-m}\\
        &\leq\int_X(\beta_{\phi}-\beta_{u})\beta_{u}^k\wedge\beta_{\phi}^{m-k-1}\wedge\omega^{n-m}\\
        &=\int_X(\phi-u)dd^c\left(\beta_{u}^k\wedge\beta_{\phi}^{m-k-1}\wedge\omega^{n-m}\right).
        \end{align*}
    Next, we turn to estimate $dd^c(\beta_{u}^k\wedge\beta_{\phi}^{m-k-1}\wedge\omega^{n-m})$:
    \begin{align*}
dd^c(\beta_{u}^k\wedge\beta_{\phi}^{m-k-1}\wedge\omega^{n-m})=&d(\pm k\beta_u^{k-1}\wedge d^c\beta\wedge\beta_\phi^{m-k-1}\wedge\omega^{n-m}\pm (n-m)\beta_u^{k}\wedge d^c\omega\wedge\beta_\phi^{m-k-1}\wedge\omega^{n-m-1})\\
&\pm (m-k-1)\beta_u^{k}\wedge d^c\beta\wedge\beta_\phi^{m-k-2}\wedge\omega^{n-m}\\
\leq&C\beta_u^{k-2}\wedge\beta_\phi^{m-k-1}\wedge\omega^{n-m+3}+C\beta_u^{k-1}\wedge\beta_\phi^{m-k-1}\wedge\omega^{n-m+2}\\
&+C\beta_u^{k-1}\wedge\beta_\phi^{m-k-1}\wedge\omega^{n-m+2}+C\beta_u^{k}\wedge\beta_\phi^{m-k-1}\wedge\omega^{n-m+1}\\
&+C\beta_u^{k-1}\wedge\beta_\phi^{m-k-2}\wedge\omega^{n-m+3}+C\beta_u^{k}\wedge\beta_\phi^{m-k-2}\wedge\omega^{n-m+2}\\
&+C\beta_u^{k}\wedge\beta_\phi^{m-k-3}\wedge\omega^{n-m+3}.
    \end{align*}

  Using the same argument as above it is easy to show that
        \begin{align*}
        & \int_X(\phi-u)dd^c\left(\beta_{u}^k\wedge\beta_{\phi}^{m-k-1}\wedge\omega^{n-m}\right)\leq\frac{Cs}{(\lambda\epsilon)^3}\int_{U(\epsilon,s)}(T_k+T_{k-1}+T_{k-2})\\
         \leq&a\left(1+\frac{1}{1-a}+\frac{1}{(1-a)^2}\right)\int_{U(\epsilon,s)}T_k\leq7a\int_{U(\epsilon,s)}T_k.
        \end{align*}
        Where we take $a:=\frac{Cs}{(\lambda\epsilon)^3}$ and the last inequality holds because $s<\frac{(\lambda\epsilon)^3}{2C}$ (at each step we enlarge $C$ if necessary and then take the assumption $s<\frac{(\lambda\epsilon)^3}{2C}$). Consequently, we have finished the proof of (\ref{eq induction}).

       Since $\phi=(1-\epsilon)v+S(\epsilon)+s$ on $U(\epsilon,s)$, by the maximum principle again we arrive at
  \begin{equation}\label{eq comparison}
\left(1-\frac{Cs}{(\lambda\epsilon)^{3}}\right)^m\int_{U(\epsilon,s)}(\beta+dd^c(1-\epsilon)v)^m\wedge\omega^{n-m}\leq\int_{U(\epsilon,s)}(\beta+dd^cu)^m\wedge\omega^{n-m}.
  \end{equation}   
\end{proof}
We give a remark, which will be useful in establishing the uniform a priori estimates for complex Hessian equations when $\beta\in\Gamma_m(\omega)$.
\begin{remark}
    In the above theorem, set $y:=\frac{Cs}{(\lambda\epsilon)^{3}}$, we have $0<y<\frac{1}{2}$ because $s<\frac{(\lambda\epsilon)^3}{2C}$. A basic calculation yields that $(\frac{1}{1-y})^m\leq2^m(1+y)^m\leq4^m(1+y)$. Consequently, (\ref{eq comparison}) can be rewritten as
  \begin{align*}
\int_{U(\epsilon,s)}(\beta+dd^c(1-\epsilon)v)^m\wedge\omega^{n-m}&\leq4^m\left(1+\frac{Cs}{(\lambda\epsilon)^{3}}\right)\int_{U(\epsilon,s)}(\beta+dd^cu)^m\wedge\omega^{n-m}\\
&\leq4^{m+1}\int_{U(\epsilon,s)}(\beta+dd^cu)^m\wedge\omega^{n-m}.
  \end{align*}
  The last inequality is again due to the assumption $s<\frac{(\lambda\epsilon)^3}{2C}$.
\end{remark}
\begin{corollary}\label{cor:m-positive is non-collapsing}
 If $\beta\in\Gamma_m(\omega)$, then $\beta$ is $(\omega,m)$-non-collapsing.
\end{corollary}
\begin{proof}
    We argue by contradiction. Suppose there is a function $\varphi\in \operatorname{SH}_m(X,\beta,\omega)\cap L^\infty(X)$ such that $\beta_\varphi^m\wedge\omega^{n-m}=0$ as a measure. Take $u=\varphi,v\equiv0$ in \cref{modified comparison principle} we obtain that for small $s$,
    $$
\left(1-\frac{Cs}{(\lambda\epsilon)^{3}}\right)^m\int_{U(\epsilon,s)}\beta^m\wedge\omega^{n-m}\leq\int_{U(\epsilon,s)}(\beta+dd^c\varphi)^m\wedge\omega^{n-m}=0.
    $$
    Since $\beta\in\Gamma_m(\omega)$, it is clear that $\beta^m\wedge\omega^{n-m}$ is a smooth volume form on $X$. It follows that $U(\epsilon,s)$ is a set of Lebesgue measure zero, which is absurd since $U(\epsilon,s)=\{u<\inf_Xu+s\}$ is a non-empty quasi-open set.
\end{proof}

\begin{definition}\label{def of m-big}
    We say a smooth $(1,1)$-form $\beta$ is $(\omega,m)$-big if there is a $(\beta,\omega,m)$-subharmonic potential $\rho$ such that $\beta+dd^c\rho$ dominates an $(\omega,m)$-positive form $\gamma\in\Gamma_m(\omega)$. In this setting we have $dd^c\rho\geq\gamma-\beta$ and hence $\rho$ is a quasi-psh function, by Demailly's regularization theorem, we may assume that $\rho$ has analytic singularities.
\end{definition}
We can furthermore extend \cref{cor:m-positive is non-collapsing} in the following way:
\begin{corollary}\label{cor:m-positive is non-collapsing 2}
 If $\beta\in\overline{\Gamma_m(\omega)}$ is $(\omega,m)$-semi-positive and $(\omega,m)$-big, then $\beta$ is $(\omega,m)$-non-collapsing. 
\end{corollary}
\begin{proof}
    The proof is very similar to that in \cref{modified comparison principle} and \cref{cor:m-positive is non-collapsing}, see also \cite[Theorem 2.1]{GL25}. Let $\rho$ be a quasi-psh function with analytic singularities and $\beta+dd^c\rho\geq\gamma\in\Gamma_m(\omega)$. Fix a test function $u\in \operatorname{SH}_m(X,\beta,\omega)\cap L^\infty(X)$. Set $\Omega:=\{\rho>-\infty\}$, $S:=\inf_\Omega(u-\rho)$ and $U(s):=\{u<\rho+S+s\}\subset\subset\Omega$ for any $s>0$. We let $\phi:=\max(u,\rho+S+s)$ so that $\phi$ is actually a bounded $(\beta,\omega,m)$-sh function. Set $T_k:=\beta_{u}^k\wedge\beta_{\phi}^{m-k}\wedge\omega^{n-m}$ for $0\leq k\leq m$. We also take the convention that $T_j=0$ for $j<0$. Our first step is to prove that
    \begin{equation}\label{eq induction2}
        \left(1-\frac{Cs}{\lambda^3}\right)\int_{U(s)}T_k\leq\int_{U(s)}T_{k+1},
    \end{equation}
    for $k=0,1,...,m-1$, where $0<\lambda<1$ is a positive constant such that $\gamma-\lambda\omega\in\Gamma_m(\omega)$. We will prove by induction.

    Firstly, assume $k=0$. Since $u\leq\phi$, it follows from \cref{cor of max principle} that 
    $$
\mathds{1}_{\{u=\phi\}}\beta_{\phi}^m\wedge\omega^{n-m}\geq\mathds{1}_{\{u=\phi\}}\beta_{u}\wedge\beta_{\phi}^{m-1}\wedge\omega^{n-m}.
    $$
  Since $U(s)=\{u<\phi\}$, the above inequality implies that
    \begin{align*}
        \int_{U(s)}(T_0-T_1)&=\int_{U(s)}(\beta_{\phi}-\beta_{u})\beta_{\phi}^{m-1}\wedge\omega^{n-m}\\
        &\leq\int_X(\beta_{\phi}-\beta_{u})\beta_{\phi}^{m-1}\wedge\omega^{n-m}\\
        &=\int_X(\phi-u)dd^c\left(\beta_{\phi}^{m-1}\wedge\omega^{n-m}\right)\\
        &=\int_{U(s)}(\phi-u)dd^c\left(\beta_{\phi}^{m-1}\wedge\omega^{n-m}\right).
        \end{align*}
        The same argument as in \cref{modified comparison principle} shows that
        \begin{align*}
            dd^c(\beta_{\phi}^{m-1}\wedge\omega^{n-m})
            \leq&C\left(\sum_{j=1}^3\beta_{\phi}^{m-j}\wedge\omega^{n-m+j}\right).
            \end{align*}
           On the set $U(s)$ we have $\beta_\phi-\lambda\omega=\beta_\rho-\lambda\omega\geq\gamma-\lambda\omega\in\Gamma_m(\omega)$. We therefore have
            $$
\lambda\,\beta_{\phi}^{m-j}\wedge\omega^{n-m+j}\leq\beta_{\phi}^{m-j+1}\wedge\omega^{n-m+j-1}.
            $$
for each $1\leq j\leq3$. Repeating the argument $j$ times we get
$$
\lambda^j\int_X(\phi-u)\beta_{\phi}^{m-j}\wedge\omega^{n-m+j}\leq\int_X(\phi-u)\beta_{\phi}^{m}\wedge\omega^{n-m}.
$$
It follows that
$$
\int_{U(s)}(T_0-T_1)\leq Cs\sum_{j=1}^3\lambda^{-j}\int_{U(s)}\beta_{\phi}^m\wedge\omega^{n-m}\leq\frac{Cs}{\lambda^3}\int_{U(s)}T_0.
$$
This yields (\ref{eq induction2}) for $k=0$.

Continuing the induction as in \cref{modified comparison principle} we arrive at
$$
\left(1-\frac{Cs}{\lambda^3}\right)^m\int_{U(s)}\beta_\phi^m\wedge\omega^{n-m}\leq\int_{U(s)}\beta_u^m\wedge\omega^{n-m}.
$$
If $u$ satisfies $\int_{X}\beta_u^m\wedge\omega^{n-m}=0$, then $\int_{U(s)}\beta_\phi^m\wedge\omega^{n-m}=0$. Since $\beta_\phi\geq\gamma$ on $U(s)$, $\beta_\phi^m\wedge\omega^{n-m}\geq\gamma^m\wedge\omega^{n-m}>0$ on $U(s)$. This implies that $U(s)$ is a set of Lebesgue measure zero, which is a contradiction.
\end{proof}
\begin{remark}
    From the above proof we see that if $\beta$ admits a bounded potential $\rho\in \operatorname{SH}_m(X,\beta,\omega)\cap L^\infty(X)$ and $\beta$ is $(\omega,m)$-big, then \cref{cor:m-positive is non-collapsing 2} still holds true. The semi-positivity is in fact not necessary.
\end{remark}
\begin{corollary}\label{m-positive have domination}
If $\beta$ is $(\omega,m)$-semi-positive and $(\omega,m)$-big, in particular if $\beta\in\Gamma_m(\omega)$, the domination principles \cref{domination for beta}, \cref{domination for beta 2}, \cref{domination for beta 3} are true for $\beta$.
\end{corollary}
\begin{proof}
    It follows from \cref{cor:m-positive is non-collapsing 2} that $\beta$ is $(\omega,m)$-non-collapsing, hence the conditions in \cref{domination for beta}, \cref{domination for beta 2}, \cref{domination for beta 3} are fulfilled.
\end{proof}

\section{Mixed type inequalities}

In this section, we establish the mixed type inequality for complex Hessian measures with respect to an arbitrary form on a domain. 

Let $(X,\omega)$ be a compact Hermitian manifold of complex dimension $n$ equipped with a Hermitian metric $\omega$. Let $m$ be a positive integer such that $1\leq m\leq n$.  We also use  $(\Omega,\omega)$ to denote  a bounded $m$-pseudoconvex domain with smooth boundary equipped with a Hermitian metric $\omega$ (i.e. it admits a smooth strictly exhaustive $(\omega,m)$-subharmonic function) and let $\chi$ be an arbitrary smooth real $(1,1)$- form defined in a neighborhood of $\Omega$. We first establish two preliminary lemmas, whose proofs can be copied along the lines in \cite{GN18}:

\begin{lemma}\label{stability for Hessian}
 Let  $\Omega$ be a bounded strictly $m$-pseudoconvex domain in $\mathbb{C}^n$ and let $f_1,f_2$ be non-negative functions in $L^p(\Omega)$, $p>\frac{n}{m}$. Let $\phi_1,\phi_2\in C^0(\partial\Omega)$. Suppose $u_1,u_2$ solves the following dirichlet problems:
    $$
\begin{cases}
(\chi+dd^cu_j)^m\wedge\omega^{n-m}=f_jdV_X,  & on\quad\Omega \\
u_j=\phi_j,    & on\quad\partial\Omega\\
u_j\in C(\overline{\Omega})\cap \operatorname{SH}_m(\Omega,\chi,\omega)
\end{cases}.
$$
for $j=1,2$. Then,
$$
\|u_1-u_2\|_{L^{\infty}(\Omega)}\leq\underset{\partial\Omega}{\sup}\,|\phi_1-\phi_2|+C\|f_1-f_2\|_{L^p(\Omega)}^\frac{1}{m}.
$$
Where $C$ depends only on $p$ and $\Omega$.
\end{lemma}
\begin{proof}
    The proof is similar to that in \cite[Lemma 3.13]{GN18}. Suppose that $\Omega\subset U:=U(0,R)$ for a ball $U$ in $\mathbb{C}^n$. Extend $f_1,f_2$ by zero to $U$ and set $h:=|f_1-f_2|^{\frac{n}{m}}$ in $U$. It follows that $h\in L^{\frac{pm}{n}}(U)$ and $\|h\|_{L^{\frac{pm}{n}}(U)}^{\frac{1}{n}}=\|f_1-f_2\|_{L^p(\Omega)}^\frac{1}{m}$. Note moreover that $\frac{pm}{n}>1$. By \cite{Kol96} there exists $\rho\in\operatorname{PSH}(U)\cap C^0(\overline{U})$ solving
    $$
(dd^c\rho)^n=hdV_X,\quad\rho|{\partial U}=0
$$
and 
$$
\|\rho\|_{\infty}\leq C\|h\|_{L^{\frac{pm}{n}}(U)}^{\frac{1}{n}}=C\|f_1-f_2\|_{L^p(\Omega)}^\frac{1}{m}.
$$
Where $C=C(m,n,p,B,\omega)$ is a uniform constant. The mixed type inequality in the Monge-Amp\`ere case tells us that
$$
(dd^c\rho)^m\wedge\omega^{n-m}\geq h^{\frac{m}{n}}dV_X=|f_1-f_2|dV_X.
$$
We can thus write
\begin{align*}
    (\chi+dd^cu_1+dd^c\rho)^m\wedge\omega^{n-m}&\geq(\chi+dd^cu_1)^m\wedge\omega^{n-m}+(dd^c\rho)^m\wedge\omega^{n-m}\\
    &\geq f_1dV_X+|f_1-f_2|dV_X\\
    &\geq f_2dV_X=(\chi+dd^cu_2)^m\wedge\omega^{n-m}.
\end{align*}
Since $\rho\leq0$ in $\Omega$, it follows that $u_1+\rho\leq u_2+\underset{\partial\Omega}{\sup}\,|\phi_1-\phi_2|$ on $\partial\Omega$. The comparison principle \cite[Proposition 2.3]{KN16} gives that $u_1+\rho\leq u_2+\underset{\partial\Omega}{\sup}\,|\phi_1-\phi_2|$ on $\Omega$ and hence
$$
u_1-u_2\leq-\rho+\underset{\partial\Omega}{\sup}\,|\phi_1-\phi_2|\leq\underset{\partial\Omega}{\sup}\,|\phi_1-\phi_2|+C\|f_1-f_2\|_{L^p(\Omega)}^\frac{1}{m}.
$$
   Similarly we get the reverse inequality and the proof is concluded. 
\end{proof}
\begin{lemma}\label{local dirichlet for omega} 
    Let $0\leq f\in L^p(\Omega)$ for some $p>\frac{n}{m}$, where $\Omega$ is a bounded strictly $m$-pseudoconvex domain in $\mathbb{C}^n$. Assume $\varphi\in C^0(\partial\Omega)$. Then there exists a unique continuous function $u\in \operatorname{SH}_m(\Omega,\chi,\omega)\cap C^0(\overline{\Omega})$ solves
    $$
H_{\chi,m}(u):=(\chi+dd^cu)^m\wedge\omega^{n-m}=f\omega^n,\quad u=\varphi\;on\;\partial\Omega.
    $$
\end{lemma}
\begin{proof}
    Since $\Omega$ is strictly $m$-pseudoconvex, the subsolution condition in \cite[Theorem 1.1]{CP22} is automatically satisfied. Choose smooth positive functions $f_j\rightarrow f$ in $L^p$ and $\phi_j\in C^{\infty}(\partial\Omega)$ converges uniformly to $\varphi$ on $\partial\Omega$. We can apply \cite[Theorem 1.1]{CP22} to obtain $u_j\in \operatorname{SH}_m(\Omega,\chi,\omega)\cap C^{\infty}(\overline{\Omega})$ solving
    $$
(\chi+dd^cu_j)^m\wedge\omega^{n-m}=f_j\omega^n,\quad u=\varphi\;on\;\partial\Omega.
    $$
    The stability theorem \cref{stability for Hessian} yields that
$$
\|u_j-u_k\|_{\infty}\leq \underset{\partial\Omega}{\sup}\,|u_j-u_k|+C\|f_j-f_k\|_{L^p}^\frac{1}{m}.
$$
By our assumption $u_j$ is a Cauchy sequence in $C^0(\overline{\Omega})$, let $u\in \operatorname{SH}_m(\Omega,\chi,\omega)\cap C^0(\overline{\Omega})$ be the uniform limit of $u_j$, then it is clear that
$u$ satisfies the desired equation.
\end{proof}
We can now establish the following mixed Hessian inequalities:
\begin{proposition}\label{prop:mixed type for chi}
  Let $\chi_,...,\chi_m$ be smooth $(1,1)$-forms and let $u_j$ be bounded $(\chi_j,\omega,m)$-sh functions for $1\leq j\leq m$ satisfying $H_{\chi,m}(u_k)= f_kdV_X$ for some $f_k\in L^1(X)$. Then, 
    $$
(\chi_1+dd^cu_1)\wedge...\wedge(\chi_m+dd^cu_m)\wedge\omega^{n-m}\geq (f_1...f_m)^{\frac{1}{m}}dV_X.
    $$
\end{proposition}
\begin{proof}
  The problem is local; hence we may assume without loss of generality that $\Omega$ is a unit ball $U$. We first suppose that there is a constant $p>1$ such that all the functions $f_k$ lie in $L^p(X)$.
When the functions $u_1,...,u_m$ are smooth, the statement follows immediately from Garding's inequality. 

We first claim that we only need to assume that all $u_k$ are continuous on $B$. Indeed, since $u_k$ is upper semi-continuous, we can find sequences of continuous functions $\psi_k^j$ decreasing to $u_k$ on $\overline{U}$ for each $k$. Thanks to \cref{local dirichlet for omega}, we can solve the following system of dirichlet problems:
     $$
\begin{cases}
(\chi_k+dd^cu_k^j)^m\wedge\omega^{n-m}=f_kdV_X,  & on\quad U \\
u_k^j=\psi_k^j,    & on\quad\partial U\\
u_k^j\in C(\overline{U})\cap \operatorname{SH}_m(U,\chi_k,\omega)
\end{cases}.
    $$
 Since $\psi_k^j$ decreases to $u_k$,  it follows from the comparison principle \cite[Corollary 3.11]{GN18} that $u_k^j$ is decreasing with respect to $j$ for each $k$ and $u_k^j\geq u_k$. Then it is clear that $u_k^j$ decreases to $u_k$. Thus, the claim follows from the decreasing convergence theorem \cite[Lemma 5.1]{KN25a} for uniformly bounded potentials.
      
    Thus, we may assume that all the $u_k$ are continuous near $\overline{U}$. Selecting sequences of smooth positive functions $f_k^j\rightarrow f_k$ in $L^p(\overline{U})$ and smooth functions $\phi_k^j$ converge uniformly to $u_k$ on $\partial U$ for each $k$. Since we are working on a ball $U$, the subsolution condition in \cite[Theorem 1.1]{CP22} is automatically satisfied, thus we can find sequences $v_k^j\in \operatorname{SH}_m(U,\chi_k,\omega)\cap C^{\infty}(X)$ solving the following system of dirichlet problems:
    $$
\begin{cases}
(\chi_k+dd^cv_k^j)^m\wedge\omega^{n-m}=f_k^jdV_X,  & on\quad U \\
v_k^j=\phi_k^j,    & on\quad\partial U
\end{cases}.
    $$
    By the stability estimate \cref{stability for Hessian},
    $$
\|u_k-v_k^j\|_{\infty}\leq \underset{\partial U}{\sup}\,|u_k-v_k^j|+C(p,m,n,B)\|f_k-f_k^j\|_{L^p}^{\frac{1}{m}}.
    $$
    and hence $v_k^j\rightarrow u_k$ uniformly as $j\rightarrow\infty$ for all $k$.  Since
    $$
(\chi_1+dd^cv_1^j)\wedge...\wedge (\chi_m+dd^cv_m^j)\wedge\omega^{n-m}\geq (f_1^j...f_m^j)^{\frac{1}{m}}dV_X 
    $$
    by Garding's inequality, letting $j\rightarrow\infty$ we get the desired inequality.

    Finally, when all $f_k$ lie in $L^1(X)$, we can approximate $f_k$ by a sequence of increasing $L^2$ functions and proceed similarly as in the previous steps, as described in \cite[Lemma 6.3]{Kol05}. The proof is thus complete.
\end{proof}

\begin{remark}
We have used the comparison principle \cite[Corollary 3.11]{GN18}. In \cite[section 3]{GN18} the Hermitian metric $\omega$ was supposed to be locally conformal to K\"ahler, that condition there was just used to define the Hessian measure. Thanks to the work in \cite{KN25a}, Hessian measures are also well-defined for a general Hermitian metric $\omega$, and the arguments in \cite[Corollary 3.11]{GN18} remain true in the general case. 
\end{remark}

We can now state the following version of global mixed type inequalities, which is a direct consequence of \cref{prop:mixed type for chi}:
\begin{proposition}\label{prop:mixed type for beta}
    Let $\beta\in BC^{1,1}(X)$ be a Bott-Chern class admitting a bounded potential $\rho\in\operatorname{SH}_m(X,\beta,\omega)\cap L^\infty(X)$ on $X$ and let $u_1,...u_m$ be bounded $(\beta,\omega,m)$-sh functions satisfying $H_{\beta,m}(u_k)=f_kdV_X$ for some $f_k\in L^1(X,\omega^n)$. Then, 
    $$
(\beta+dd^cu_1)\wedge...\wedge(\beta+dd^cu_m)\wedge\omega^{n-m}\geq (f_1...f_m)^{\frac{1}{m}}dV_X.
    $$
    
\end{proposition}

\section{A priori estimates in the \texorpdfstring{$m$}{m}-positive cone}
Let $(X,\omega)$ be a compact Hermitian manifold of complex dimension $n$ equipped with a Hermitian metric $\omega$. Let $m$ be a positive integer such that $1\leq m\leq n$. Let $\beta\in\Gamma_m(\omega)$ be an $(\omega,m)$-positive form.
Fix a constant $0<\lambda\leq1$ such that 
$$
\beta-\lambda\omega\in\Gamma_m(\omega).
$$
This can be done since $\Gamma_m(\omega)$ is an open cone. 

The goal of this section is to establish a general $L^{\infty}$-estimate for the equation 
$$
(\beta+dd^cu)^m\wedge\omega^{n-m}=f\omega^n,\quad f\in L^{p}(\omega^n),\;p>\frac{n}{m}.
$$

\begin{lemma}\label{decay of volume}
    For any borel set $E\subset X$, define $V_{\omega}(E):=\int_E\omega^n$. Let $\varphi\in \operatorname{SH}_m(X,\beta,\omega)$ with $\sup_X\varphi=0$ . Then, for any $t>0$,
    $$
V_{\omega}(\{\varphi<-t\})\leq\frac{C}{t}.
    $$
    Where $C=C(X,\beta,\omega)$.
\end{lemma}
\begin{proof}
    $$
V_{\omega}(\{\varphi<-t\})=\int_{\{\varphi<-t\}}\omega^n\leq\frac{1}{t}\int_X|\varphi|\omega^n\leq\frac{C}{t}.
    $$
    Here $C$ is the uniform constant in \cref{L^1 compactness}.
\end{proof}
Next, we move on to establish the following important volume-capacity estimate. 

\begin{proposition}\label{volume-cap estimate for beta}
    Fix $\beta\in\Gamma_m(\omega)$ and $1<\tau<\frac{n}{n-m}$. There exists a uniform constant $C=C(\tau,X,\beta,\omega)>0$,  such that for any Borel set $E\subset X$,
    $$
V_{\omega}(E)\leq C[\operatorname{Cap}_{\beta,m,\omega}(E)]^{\tau}.
    $$
    The constant $C$ depends on $X,n,m,\beta,\omega$.
\end{proposition}
\begin{proof}
    When $\beta$ is a Hermitian metric, the result has been established by \cite[Proposition 3.6]{KN16}. When $\beta$ is merely $(\omega,m)$-positive, the result follows immediately from the equivalence of $\operatorname{Cap}_{\beta,m,\omega}(\cdot)$ and $\operatorname{Cap}_{\omega,m,\omega}(\cdot)$ (\cref{properties of Cap_m}).
\end{proof}
\begin{remark}
    For the convenience of the readers, we give a direct proof of \cref{volume-cap estimate for beta} by using the comparison of the local capacity and the global capacity \cref{comparison of capacity}.
    We first prove the local version: let $U\subset X$ be a coordinate ball and select a Borel set $E\subset\subset U$, we are going to prove that
    $$
V_{\omega}(E)\leq C\cdot \operatorname{Cap}_{m}(E,U)^{\tau},
    $$
    where $Cap_m$ is the local capacity defined in \cref{def: local cap}. The argument in \cite[Proposition 2.1]{DK14} can then be carried out word by word in this case.
    
We next choose coordinate balls $U_{j}\subset\subset U_{j}^{\prime},1\leq j\leq N$, such that $U_{1},...,U_{N}$ cover $X$. By the previous step, for any Borel set $K\subset X$, we can write 
\begin{align*}
    V_\omega(E)&=V_\omega(\cup_{j=1}^N(K\cap U_j))\leq\sum_{j=1}^NV_{\omega}(K\cap U_j)\leq\sum_{j=1}^NC[Cap_m(K\cap U_j,U_j^{\prime})]^{\tau}\\
    &\leq C_1\cdot \operatorname{Cap}_{BT,m}(K)^{\tau}\leq C_2\cdot \operatorname{Cap}_{\beta,m,\omega}(K)^{\tau}.
\end{align*}
Where the third inequality follows from the definition of $\operatorname{Cap}_{BT,m}(\cdot)$ and the last follows from \cref{comparison of capacity}.
\end{remark}

We now give an estimate of the capacity of sublevel sets by using the modified comparison principle \cref{modified comparison principle}:
\begin{corollary}\label{decay of cap}
    Fix $0<\epsilon<\frac{3}{4}$. Consider $u,v\in \operatorname{SH}_m(X,\beta,\omega)\cap L^{\infty} (X)$ with $u\leq0$ and $-1\leq v\leq0$. Set $S(\epsilon):=\inf_X[u-(1-\epsilon)v]$ and $U(\epsilon,s):=\{u<(1-\epsilon)v+S(\epsilon)+s\}$. Then, for any $0<s,t<\frac{(\lambda\epsilon)^3}{10C_1}$ (here and after, $C_1$ will always denote the uniform constant in \cref{modified comparison principle}), we have
    $$
t^m\operatorname{Cap}_{\beta,m,\omega}(U(\epsilon,s))\leq C\int_{U(\epsilon,s+t)}(\beta+dd^cu)^m\wedge\omega^{n-m}.
    $$
    Where $C>0$ is a uniform constant depending only on $m$. 
\end{corollary}
\begin{proof}
  The proof is a simple modification of \cite[Lemma 5.4, Remark 5.5]{KN15} in the Monge-Amp\`ere case. Pick  a function $w\in \operatorname{SH}_m(X,\beta,\omega)\cap L^{\infty}(X)$ such that $0\leq w\le1$. Since $w\geq0\geq v$, we have the following inclusion of sets:
    $$
U(\epsilon,s)\subset\{u<(1-\epsilon)[(1-t)v+tw]+S(\epsilon)+s\}.
    $$
    Set 
    $$
S(\epsilon,t):=\inf_X\left(u-(1-\epsilon)[(1-t)v+tw]\right),
    $$
    then we have $S(\epsilon,t)\leq S(\epsilon)\leq S(\epsilon,t)+2(1-\epsilon)t$. Consequently, we furthermore have the following relationship of sets:
    \begin{align*}
        U(\epsilon,s)\subset V&:=\{u<(1-\epsilon)[(1-t)v+tw]+S(\epsilon,t)+2(1-\epsilon)t+s\}\\
        &\subset U(\epsilon,s+4(1-\epsilon)t\}.
    \end{align*}
    We then use \cref{modified comparison principle} for the functions $u$ and $(1-t)v+tw$ on $V$ to get that
    \begin{align*}
        [(1-\epsilon)t]^m\int_{U(\epsilon,s)}\beta_w^m\wedge\omega^{n-m}&\leq\int_{U(\epsilon,s)}[\epsilon\beta+(1-\epsilon)(1-t)\beta_v+(1-\epsilon)t\beta_w]^m\wedge\omega^{n-m}\\
        &\leq\int_V[\beta+(1-\epsilon)dd^c((1-t)v+tw)]^m\wedge\omega^{n-m}\\
        &\leq 4^m\left(1+\frac{C_1(s+2(1-\epsilon)t)}{(\lambda\epsilon)^3}\right)\int_V\beta_u^m\wedge\omega^{n-m}\\
        &\leq C(m)\int_{U(\epsilon,s+4(1-\epsilon)t\}}\beta_u^m\wedge\omega^{n-m}.
    \end{align*}
  Note that since we have assumed that $0<s,t<\frac{(\lambda\epsilon)^3}{10C_1}$, $0<s+2(1-\epsilon)t<\frac{(\lambda\epsilon)^3}{2C
  _1}$, thus \cref{modified comparison principle} is applicable and we may take $C(m):=4^{m+1}$. After a rescaling $t^{\prime}:=4(1-\epsilon)t$, then the above inequality reads
    $$
t^m\int_{U(\epsilon,s)}\beta_w^m\wedge\omega^{n-m}\leq 4^{2m+1}\int_{U(\epsilon,s+t\}}\beta_u^m\wedge\omega^{n-m}.
    $$
\end{proof}

\begin{lemma}\label{comparison of volumes}
    We keep the notations as in \cref{decay of cap}. Assume furthermore that
    $$
(\beta+dd^c u)^m\wedge\omega^{n-m}\leq f\omega^n 
    $$
    for $0\leq f\in L^p(\omega^n)$, and $p>\frac{n}{m}$. Fix $0<\alpha<\frac{p-n/m}{p(n-m)}$. Then, there exists a constant $C$ such that for any $0<s,t<\frac{(\lambda\epsilon)^3}{10C_1}$,
    $$
t[V_\omega(U(\epsilon,s)]^{\frac{1}{m\tau}}\leq C\|f\|_p^{\frac{1}{m}}[V_\omega(U(\epsilon,s+t)]^{\frac{1+m\alpha}{m\tau}},
    $$
    where $\tau=\frac{(1+m\alpha)p}{p-1}<\frac{n}{n-m}$. Here $C$ is a constant depending only on $\alpha,n,m,B,X,\omega$. 
\end{lemma}
\begin{proof}
    Similar to \cite[Lemma 3.9]{KN16}, we can compute the indexes:
    $$
0<\alpha<\frac{p-n/m}{p(n-m)}\Leftrightarrow\frac{p}{p-1}<\tau=\frac{(1+m\alpha)p}{p-1}<\frac{n}{n-m}.
    $$
    By the volume-capacity estimate \cref{volume-cap estimate for beta} and \cref{decay of cap} we have
    $$
t^m[V_\omega(U(\epsilon,s)]^{\frac{1}{\tau}}\leq C^\prime t^m\operatorname{Cap}_{\beta,m,\omega}(U(\epsilon,s))\leq C^\prime\cdot C^{\prime\prime}\int_{U(\epsilon,s+t)}f\omega^n.
    $$
  Where $C^\prime$ is the constant in \cref{volume-cap estimate for beta} and $C^{\prime\prime}$ is the constant in \cref{decay of cap}.  The result follows from the using of H\"older's inequality and taking the $m$th root. The dependence of constants follows easily from \cref{volume-cap estimate for beta} and \cref{decay of cap}.
    
    \end{proof}

\begin{theorem}\label{hessian a priori estimate}
     Fix $0<\epsilon<\frac{3}{4}$. Consider $u,v\in \operatorname{SH}_m(X,\beta,\omega)\cap L^{\infty}(X)$ with $u\leq0$ and $-1\leq v\leq0$. Assume furthermore that
    $$
(\beta+dd^cu)^m\wedge\omega^{n-m}\leq f\omega^n,
    $$
    for some $0\leq f\in L^p(\omega^n)$ and $p>\frac{n}{m}$. Put $U(\epsilon,s):=\{u<(1-\epsilon)v+S(\epsilon)+s\}$ for $s>0$ and fix $0<\alpha<\frac{p-n/m}{p(n-m)}$. Then, there exists a constant $C_\alpha$ such that for any $0<s<\frac{(\lambda\epsilon)^3}{10C_1}:=\delta$,
    $$
s\leq4C\|f\|_p^{\frac{1}{m}}[V_\omega(U(\epsilon,s))]^{\frac{\alpha}{\tau}}.
    $$
    Where $\tau=\frac{(1+m\alpha)p}{p-1}<\frac{n}{n-m}$ and $C$ is a constant depending only on $\alpha,n,m,B,X,\omega$. 
\end{theorem}
\begin{proof}
   The proof is identical to that in \cite[Theorem 3.10]{KN16}, so we will be brief. Define 
   $$
a(s):=[V_\omega(U(\epsilon,s))]^{\frac{1}{m\tau}},
   $$
   then \cref{comparison of volumes} implies that for any $0<s,t<\frac{(\lambda\epsilon)^3}{10C_1}$, we have
   $$
ta(s)\leq C\|f\|_p^{\frac{1}{m}}[a(s+t)]^{1+m\alpha}.
   $$
   To finish the proof, it is enough to show that for any $0<s<\frac{(\lambda\epsilon)^3}{10C_1}$, 
   $$
s\leq\frac{2^{1+m\alpha}}{2^{m\alpha}-1}C[a(s)]^{m\alpha}.
   $$
   The argument is exactly identical to that of \cite[Theorem 3.10]{KN16}.
   
\end{proof}
We can now state our main theorem in this section:
\begin{theorem}\label{main a priori Hessian estimates}
Let $\beta\in\Gamma_m(\omega)$ be a smooth $(1,1)$-form on $X$. If $u\in \operatorname{SH}_m(X,\beta,\omega)\cap L^{\infty}(X)$ satisfies
\begin{equation}
    (\beta+dd^cu)^m\wedge\omega^{n-m}\leq f\omega^n,\quad \sup_Xu=-1
\end{equation}
for some $0\leq f\in L^p(\omega^n)$, $p>\frac{n}{m}$, then
$$
Osc_X(u)\leq C.
$$
    Where the bound $C$ depends on $\beta,n,m,B,X,\omega$ and the upper bound of  $\|f\|_p$.
\end{theorem}
\begin{proof}
    
Using the notations in \cref{hessian a priori estimate}, we take $\epsilon=\frac{1}{2}$ and $v=0$ in \cref{hessian a priori estimate}. We then have for any $0<s<\frac{(\lambda_\beta\epsilon)^3}{10C_1}:=\delta$,
$$
s\leq4C\|f\|_p^{\frac{1}{m}}[V_\omega(U(\epsilon,s))]^{\frac{\alpha}{\tau}}\leq\frac{4C\|f\|_p^{\frac{1}{m}}}{|-\inf_Xu-s|^{(p-1)\alpha/p(1+m\alpha)}},
$$
where the second inequality follows from the decay of volumes (\cref{decay of volume}) and $0<\alpha<\frac{p-n/m}{p(n-m)}$ is fixed. This gives
$$
|-\inf_Xu-s|\leq \left(\frac{4C\|f\|_p^{\frac{1}{m}}}{s}\right)^{\frac{p(1+m\alpha)}{(p-1)\alpha}},
$$
and hence 
$$
|-\inf_Xu|\leq C\|f\|_p^{\frac{p(1+m\alpha)}{m(p-1)\alpha}}
$$
by choosing $s$ sufficiently small and invoking that $\|f\|_p^{\frac{1}{m}}$ has a lower bound: take $s=\frac{\delta}{2}=\frac{(\lambda_\beta\epsilon)^3}{20C_1}$ in \cref{hessian a priori estimate}, we have
$$
\|f\|_p^{\frac{1}{m}}\geq\frac{\delta}{8C[V_\omega(X)]^{\frac{\alpha}{\tau}}}.
$$
The dependence of the constant $C$ is clearly inherited from \cref{hessian a priori estimate} and we have concluded our proof.
\end{proof}

\section{Stability of the solution}
Let $(X,\omega)$ be a compact Hermitian manifold of complex dimension $n$ equipped with a Hermitian metric $\omega$. Let $m$ be a positive integer such that $1\leq m\leq n$. 
Let $\beta\in\Gamma_m(\omega)$ again be an $(\omega,m)$-positive form.

The goal of this section is to establish a $L^1-L^\infty$ stability of complex Hessian equations, which will be useful later.
\begin{theorem}\label{L^1 L^infty stability}
    Let $u,v\in \operatorname{SH}_m(X,\beta,\omega)\cap L^{\infty}(X)$ be such that $\sup_Xu=0$ and $v\leq0$. Suppose that
    $$
(\beta+dd^cu)^m\wedge\omega^{n-m}\leq f\omega^n,
    $$
    for some $0\leq f\in L^p(\omega^n)$, $p>\frac{n}{m}$. Fix $0<\alpha<\frac{p-\frac{n}{m}}{p(n-m)}$. Then,
    $$
\sup_X(v-u)\leq C\|(v-u)_+\|_{L^1}^{\frac{1}{aq}}.
    $$
    Where $q:=\frac{1}{1-\frac{1}{p}}$ is the conjugate number of $p$, $a:=\frac{1}{q}+4m+\frac{4}{\alpha}$ and the constant $C$ depends on $\|v\|_\infty$, $n,m,B,X,\omega,\beta$ and the upper bound of  $\|f\|_p$.
\end{theorem}

\begin{proof}
The proof is very similar to that in \cite[Theorem 3.11]{KN16}. Firstly, by \cref{main a priori Hessian estimates}, $u$ is bounded by a constant $C$. Write $S:=\inf_X(u-v)$ so that $-S=\sup_X(v-u)$. If $-S\leq0$, the statement is trivial. So we may assume without loss of generality that $-S>0$. Suppose
    \begin{equation}\label{eq 5}
\|(v-u)_+\|_1\leq\epsilon^{aq}
    \end{equation}
    for some fixed small constant $\epsilon<\epsilon_0$ and $a>0$ to be defined. We only have to show that there exists a constant $C$ such that $\sup_X(v-u)\leq C\epsilon$. Indeed, once $\epsilon_0$ is fixed (we will see that we only need to choose $\epsilon_0<\frac{\lambda_\beta^3}{20(N+1)C_1}$, recall that $\lambda_\beta$ is a fixed constant such that $\beta-\lambda_\beta\omega$ lies in $\Gamma_m(\omega)$ again), if $\|(v-u)_+\|_1>\epsilon_0^{aq}$, we only have to choose $C\geq\frac{\|v\|_\infty+\|u\|_\infty}{\epsilon_0^{aq}}$ (note that $C$ is allowed to depend on $\|v\|_\infty$).

    Let $N:=\|v\|_\infty$ , $S(\epsilon):=\inf_X[u-(1-\epsilon)v]$ and $U(\epsilon,t):=\{u<(1-\epsilon)v+S(\epsilon)+t\}$ be as in \cref{hessian a priori estimate}. We then have 
    $$
S-\epsilon N\leq S(\epsilon)\leq S.
    $$
    Hence it is easy to see that $U(\epsilon,2t)\subset\{u<v+S+N\epsilon+2t\}$. On the latter set, we have $(v-u)_+\geq-S-N\epsilon-2t$. If $-S=|S|\leq(N+1)\epsilon$, then we are done. So we may assume without loss of generality that $0<\epsilon<\frac{|S|}{N+1}$. We choose $0<t<\delta:=\frac{(\lambda_\beta\epsilon)^3}{10(N+1)C_1}<<\epsilon$ and hence we have $(v-u)_+\geq-S-N\epsilon-2t>-S-(N+1)\epsilon>0$ on $\{u<v+S+N\epsilon+2t\}$.

    By \cref{decay of cap} and the H\"older's inequality, we have
    \begin{align*}
        \operatorname{Cap}_{\beta,m,\omega}(U(\epsilon,t))&\leq\frac{C}{t^m}\int_{U(\epsilon,2t)}f\omega^n\leq\frac{C}{t^m}\int_X\frac{(v-u)_+^{\frac{1}{q}}}{(|S|-N\epsilon-2t)^{\frac{1}{q}}}f\omega^n\\
        &\leq\frac{C\|f\|_p}{t^m(|S|-N\epsilon-2t)^{\frac{1}{q}}}\|(v-u)_+\|_{L^1}^{\frac{1}{q}}.
    \end{align*}

On the other hand, by \cref{hessian a priori estimate} and \cref{volume-cap estimate for beta} we deduce that
$$
h(t):=\left(\frac{t}{4C_\alpha\|f\|_p^{\frac{1}{m}}}\right)^{\frac{1}{\alpha}}\leq[V_\omega(U(\epsilon,t))]^{\frac{1}{\tau}}\leq C\cdot \operatorname{Cap}_{\beta,m,\omega}(U(\epsilon,t)).
$$
Here, $C_\alpha$ is the constant in \cref{hessian a priori estimate}. Combining these two inequalities, we obtain
$$
(|S|-N\epsilon-2t)^{\frac{1}{q}}\leq\frac{C\|f\|_p}{t^mh(t)}\|(v-u)_+\|_{L^1}^{\frac{1}{q}}.
$$
Using (\ref{eq 5}), we furthermore have 
\begin{equation}\label{eq 6}
|S|\leq N\epsilon+2t+\left(\frac{C\|f\|_p}{t^mh(t)}\right)^q\|(v-u)_+\|_{L^1}\leq(N+1)\epsilon+\left(\frac{C\|f\|_p}{t^mh(t)}\epsilon^a\right)^q.
\end{equation}
We now choose firstly $\epsilon<\epsilon_0<\delta_1:=\frac{\lambda_\beta^3}{20(N+1)C_1}$ and then $t:=\frac{(\lambda_\beta\epsilon)^3}{20(N+1)C_1}$, thus we have $t>\epsilon^4 $. This implies that
$$
t^mh(t)\geq\epsilon^{4m}\frac{C\epsilon^{\frac{4}{\alpha}}}{\|f\|_p^{\frac{1}{m\alpha}}},
$$
hence it follows from (\ref{eq 6}) that
$$
|S|\leq(N+1)\epsilon+C\epsilon^{q(a-4m-\frac{4}{\alpha})}\|f\|_p^{q(1+\frac{1}{m\alpha})}.
$$
If we take $a:=\frac{1}{q}+4m+\frac{4}{\alpha}$, we get
$$
|S|\leq C(\|v\|_\infty,\alpha,p,\omega,\beta,\|f\|_p)\epsilon.
$$
The proof is therefore concluded.
\end{proof}
\begin{corollary}\label{cor of stability}
    With the same notations as in \cref{L^1 L^infty stability}, suppose that $u,v\in \operatorname{SH}_m(X,\beta,\omega)\cap L^{\infty}(X)$ satisfying $\sup_Xu=\sup_Xv=0$ and solving 
    $$
(\beta+dd^cu)^m\wedge\omega^{n-m}\leq f\omega^n,\quad(\beta+dd^cv)^m\wedge\omega^{n-m}\leq g\omega^n,
    $$
    where $0\leq f,g\in L^{p}(\omega^n),p<\frac{n}{m}$. Then
    $$
\|u-v\|_\infty\leq C\|u-v\|_{L^1(\omega^n)}^{\frac{1}{aq}},
    $$
    where $C$ depends on $n,m,B,X,\omega,\beta$ and the upper bound of  $\|f\|_p,\|g\|_p$.
\end{corollary}
\begin{proof}
    Thanks to the a priori estimates, $u,v$ are uniformly bounded by the corresponding data. The result is then an immediate consequence of  \cref{L^1 L^infty stability}.
\end{proof}
\section{Solving complex Hessian equations in the m-positive cone and applications}
Let $(X,\omega)$ be a compact Hermitian manifold of complex dimension $n$ equipped with a Hermitian metric $\omega$. Let $m$ be a positive integer such that $1\leq m\leq n$. 

In this section, we solve the degenerate Hessian equation for $(\omega,m)$-positive forms and, as an application, derive an approximation result for $(\beta,\omega,m)$-subharmonic functions that extends \cite[Lemma 3.20]{KN16} and \cite[Theorem 1.7]{GL25}.

\begin{theorem}\label{main thm for m-positive beta}
   Let $\beta\in\Gamma_m(\omega)$ be an $(\omega,m)$-positive form. Set $0\leq f\in L^p(X,\omega^n)$ with $p>\frac{n}{m}$ and $\int_Xf\omega^n>0$. Then there exists a unique constant $c>0$ and a continuous function $\varphi\in \operatorname{SH}_m(X,\beta,\omega)\cap C^{0}(X)$ such that
        $$
(\beta+dd^c\varphi)^m\wedge\omega^{n-m}=cf\omega^n.
        $$
    Furthermore, we have
    $$
    \operatorname{osc} \varphi\leq C,
    $$ 
    where $C$ is the constant in \cref{main a priori Hessian estimates}.
\end{theorem}
\begin{proof}
    Thanks to \cite[Proposition 21]{Sze18}, we can obtain a smooth solution for $\beta\in\Gamma_m(\omega)$ when the data in the right hand side is smooth. Choose a sequence of smooth positive functions $0<f_j\in C^{\infty}(X)$ converges in $L^p(\omega^n)$ to $f$. By \cite[Proposition 21]{Sze18}, there exists a constant $c_j>0$ and $\varphi_j\in \operatorname{SH}_m(X,\beta,\omega)\cap C^\infty(X)$ such that
    \begin{equation}\label{eq 7}
        (\beta+dd^c\varphi_j)^m\wedge\omega^{n-m}=c_jf_j\omega^n,\quad \sup_X\varphi_j=0.
    \end{equation}
    We first claim that the sequence $c_j$ are uniformly bounded above and below away from zero. By the Garding's inequality (see \cref{prop:mixed type for chi} in general), we have
    $$
(\beta+dd^c\varphi_j)\wedge\omega^{n-1}\geq(c_jf_j)^{\frac{1}{m}}\omega^n.
    $$
    Integrating both sides we get
    $$
c_j^{\frac{1}{m}}\int_Xf_j^{\frac{1}{m}}\omega^n\leq\int_X\beta\wedge\omega^{n-1}+\int_X\varphi_jdd^c\omega^{n-1}\leq C,
    $$
    because $\sup_X\varphi_j=0$. Since $f_j\rightarrow f$ in $L^p(\omega^n)$, we have $f_j\rightarrow f$ in $L^1(\omega^n)$ by H\"older's inequality and also $f_j^{\frac{1}{m}}\rightarrow f^{\frac{1}{m}}$ in $L^1(\omega^n)$. This implies that $\int_Xf_j^{\frac{1}{m}}\omega^n$ has a lower bound away from zero since $\int_Xf\omega^n>0$. We therefore obtain the desired upper bound for $c_j$.

    For the lower bound, we invoke \cref{hessian a priori estimate}. Taking $\epsilon=\frac{1}{2},v\equiv0$ and $s=\frac{\delta}{2}=\frac{\lambda_\beta^3}{20C_1}$ in \cref{hessian a priori estimate} we deduce that
    $$
s\leq 4C_\alpha c_j^{\frac{1}{m}}\|f_j\|_p^{\frac{1}{m}}[V_\omega(\{\varphi_j<s+\inf_X\varphi_j\})]^{\frac{\alpha}{\tau}}\leq4C_\alpha c_j^{\frac{1}{m}}\|f_j\|_p^{\frac{1}{m}}[V_\omega(X)]^{\frac{\alpha}{\tau}},
    $$
    hence
    $$
c_j^{\frac{1}{m}}\geq \frac{s}{4C_\alpha\|f_j\|_p^{\frac{1}{m}}[V_\omega(X)]^{\frac{\alpha}{\tau}}}.
    $$
    Note that $f_j\rightarrow f$ in $L^p(\omega^n)$, it follows that $\|f_j\|_p$ is uniformly bounded away from zero and thus the claim follows.

    Using \cref{main a priori Hessian estimates} we see that the sequence $\varphi_j$ is uniformly bounded. Since $\sup_X\varphi_j=0$, we may assume that $\varphi_j\rightarrow\varphi\in \operatorname{SH}_m(X,\beta,\omega)\cap L^\infty(X)$ in $L^1(\omega^n)$ and almost everywhere. The stability estimate \cref{cor of stability} implies that
    $$
\|\varphi_j-\varphi_k\|_\infty\leq C\|\varphi_j-\varphi_k\|_1^{\frac{1}{aq}},
    $$
    where $q:=\frac{1}{1-\frac{1}{p}}$ and $a:=\frac{1}{q}+4m+\frac{\alpha}{4}$. It follows from the fact $\{\varphi_j\}_j$ is a Cauchy sequence in $L^1(\omega^n)$ that $\{\varphi_j\}_j$ is also a Cauchy sequence in $C(X)$, thus we obtain that $\varphi\in C(X)$ as the uniform limit of $\varphi_j$ and 
    $$
(\beta+dd^c\varphi)^m\wedge\omega^{n-m}=cf\omega^n.
    $$
  The uniqueness of the constant $c$ follows from the domination principle \cref{m-positive have domination}  and the proof is therefore concluded.
\end{proof}
We next turn to the equation with an exponential twist. Firstly, we have to establish a smooth version similar to \cite[Theorem 3.17]{KN16}:
\begin{theorem}\label{smooth Hessian for beta}
    Let $\beta\in\Gamma_m(\omega)$ be a $(\omega,m)$-positive $(1,1)$-form and let $H$ be a smooth function on $X$. Fix a positive constant $\lambda>0$. Then there exists a unique $u\in \operatorname{SH}_m(X,\beta,\omega)\cap C^\infty(X)$ solving the Hessian equation
    $$
(\beta+dd^cu)^m\wedge\omega^{n-m}=e^{\lambda u+H}\omega^n.
    $$
\end{theorem}
\begin{proof}
    The proof is almost unchanged as in \cite[Theorem 3.17]{KN16}, which is essentially based on the work in \cite{Sze18}. 
    
    \textbf{Step 1.} We first show the $C^0$-estimate. Suppose that $u$ attains maximum at a point $x\in X$. It then follows that $dd^cu(x)\leq 0$ and hence
    $$
e^{\lambda u(x)+H(x)}=\beta_u^m\wedge\omega^{n-m}/\omega^n\leq\beta^m\wedge\omega^n/\omega^n.
    $$
    The second inequality is because $u\in \operatorname{SH}_m(X,\beta,\omega)\cap C^\infty(X)$ and hence $\overline{\Gamma_m(\omega)}\ni\beta+dd^cu\leq\beta$. Note that $\beta^m\wedge\omega^n/\omega^n$ is a smooth positive function and hence we get the upper bound for $u$. The lower bound follows similarly. 

    \textbf{Step 2.} The crucial part is to establish the Hou-Ma-Wu estimate \cite{HMW10}:
    \begin{equation}
        \sup_X|\partial\overline{\partial}u|\leq C\left(1+\sup_X|\nabla u|^2\right),
    \end{equation}
    where the constant $C$ depends on $\|u\|_\infty,\omega,\beta,H$. We will follow the lines in \cite[Proposition 13]{Sze18}. Write
    $$
\omega=\sqrt{-1}\sum\omega_{j\overline{k}}dz_j\wedge d\overline{z}_k
    $$
    $$
\beta=\sqrt{-1}\sum\beta_{j\overline{k}}dz_j\wedge d\overline{z}_k.
    $$
    Let $(\omega^{j\overline{k}})$ be the inverse matrix of $(\omega_{j\overline{k}})$ and set
    $$
A^{ij}:=\omega^{j\overline{k}}(\beta_{i\overline{k}}+u_{i\overline{k}})=:\omega^{j\overline{k}}g_{i\overline{k}}.
    $$
    Then the equation can be written as 
    $$
F(A)=\lambda u+H,
    $$
    where 
$$
F(A):=\log\,S_m(\lambda[A^{ij}]).
$$
We fix a coordinate chart at a point $z_0$ of $X$ and do all the calculations at this point. As pointed out in the proof of \cite[Theorem 3.17]{KN16}, even when we take the reference form $\beta$ to be only $(\omega,m)$-positive here instead of to be Hermitian, the terms $u_{1\overline{1}}$ can still be dominated by the first eigenvalue $\lambda_1>1$ of $(A^{ij})$ and we can obtain the inequalities (81) and (95) in \cite{Sze18} similarly  as in \cite[Theorem 3.17]{KN16} (only these two inequalities depend on the new equation $F(A)=u+H$ during the proof of \cite[Proposition 13]{Sze18}). For example, we get
$$
F^{kk}u_{pk\overline{k}}u_{\overline{p}}\geq-C_0K\mathcal{F
}-\epsilon_1F^{kk}\lambda_k^2-C_{\epsilon_1}\mathcal{F}K.
$$
Here, we take all the derivatives with respect to the Chern connection of $\omega$ locally and $K:=\left(1+\sup_X|\nabla u|^2\right)$, $\mathcal{F}:=\sum_kF^{kk}>\tau>0$, as shown in \cite[Lemma 9]{Sze18}. The proof of \cite[Proposition 13]{Sze18} can be then copied along the lines.

\textbf{Step 3}. It is now standard to use the blow-up argument to get the $C^1$-estimate and then the Evans-Krylov method to get the $C^{2,\alpha}$-estimate. The existence of the solution follows by using the continuity method through the path
$$
\log(\beta_{u_t}^m\wedge\omega^{n-m}/\omega^n)=\lambda u_t+tH,
$$
as shown in \cite[Theorem 3.17]{KN16}.

Finally, the uniqueness follows directly from the domination principle \cref{m-positive have domination}.
\end{proof}

\begin{theorem}\label{main thm twist for m-positive beta}
     Let $\beta\in\Gamma_m(\omega)$ be an $(\omega,m)$-positive form and $\lambda>0$ be a positive constant. Set $0\leq f\in L^p(X,\omega^n)$ with $p>\frac{n}{m}$ and $\int_Xf\omega^n>0$. Then there exists a unique continuous function $\varphi\in \operatorname{SH}_m(X,\beta,\omega)\cap C^{0}(X)$ such that
        $$
(\beta+dd^c\varphi)^m\wedge\omega^{n-m}=e^{\lambda\varphi}f\omega^n.
        $$
    Furthermore, we have
    $$
    \operatorname{osc} \varphi\leq C,
    $$ 
    where $C$ is a constant depends on
\end{theorem}
\begin{proof}
  The uniqueness is a simple consequence of the domination principle \cref{m-positive have domination}, we next focus on the existence. Up to rescaling we may assume without loss of generality that $\lambda=1$. Choosing a sequence of smooth positive functions $f_j\rightarrow f$ in $L^p(\omega^n)$. Using \cref{smooth Hessian for beta} we can find $u_j\in\operatorname{SH}_m(X,\beta,\omega)\cap C^{\infty}(X)$ solving
    \begin{equation}\label{eq 8}
        (\beta+dd^cu_j)^m\wedge\omega^{n-m}=e^{u_j}f_j\omega^n.
    \end{equation}
    Set $M_j:=\sup_Xu_j$ and $v_j:=u_j-M_j$. We first claim that $M_j$ are uniformly bounded. The strategy is similar to that in \cite[Claim 2.6]{Ngu16}. Using the mixed type inequality \cref{prop:mixed type for chi} we get that
    $$
(\beta+dd^cu_j)\wedge\omega^{n-1}\geq(e^{u_j}f_j)^{\frac{1}{m}}\omega^n.
    $$
    Let $G\in C^{\infty}(X)$ be the Gauduchon function such that $e^G\omega^{n-1}$ is $dd^c$-closed. Then we have 
    \begin{equation}\label{eq 9}
\int_X(e^{u_j}f_j)^{\frac{1}{m}}e^G\omega^n\leq\int_X(\beta+dd^cu_j)\wedge e^G\omega^{n-1}=\int_Xe^G\beta\wedge\omega^{n-1}.
    \end{equation}
    Let $A_j:=\int_Xf_j^{\frac{1}{m}}e^G\omega^n$, by Jensen's inequality we have that
    $$
\frac{1}{A_j}\int_X(e^{u_j}f_j)^{\frac{1}{m}}e^G\omega^n\geq e^{\frac{1}{mA_j}\int_Xu_jf_j^{\frac{1}{m}}e^G\omega^n}.
    $$
    This combined with \eqref{eq 9} gives that
    $$
\frac{M_j}{m}+\frac{1}{mA_j}\int_Xv_jf_j^{\frac{1}{m}}e^G\omega^n=\frac{1}{mA_j}\int_Xu_jf_j^{\frac{1}{m}}e^G\omega^n\leq \log\left(\frac{1}{A_j}\int_Xe^G\beta\wedge\omega^{n-1}\right).
    $$
    Since $f_j\rightarrow f$ in $L^p(\omega^n)$, we have
    $$
A_j=\int_Xf_j^{\frac{1}{m}}e^G\omega^n\rightarrow\int_Xf^{\frac{1}{m}}e^G\omega^n>0,
    $$
    this implies that the sequence $A_j$ is uniformly bounded (away from zero). Moreover, by H\"older's inequality,
    $$
\int_X(-v_j)f_j^{\frac{1}{m}}e^G\omega^n\leq\left(\int_X(-v_j)^qe^G\omega^n\right)^{\frac{1}{q}}\left(\int_Xf_j^{p}e^G\omega^n\right)^{\frac{1}{mp}}\leq C.
    $$
    Since $f_j$ is uniformly bounded in $L^p$ and $\sup_Xv_j=0$ is a compact family in $\operatorname{SH}_m(X,\beta,\omega)$, the last inequality follows from \cref{integrablity 2}. We have therefore derived the uniform upper bound of $M_j$. It remains to get a lower bound for $M_j$. Using \cref{main thm for m-positive beta} we can solve for every $j\in\mathbb{N}$
    $$
(\beta+dd^cw_j)^m\wedge\omega^{n-m}=c_jf_j\omega^n,
    $$
    for some $c_j>0$ and $w_j\in \operatorname{SH}_m(X,\beta,\omega)\cap C^\infty(X)$. Moreover, from the proof of \cref{main thm for m-positive beta} we know that $c_j$ is uniformly bounded away from zero. Moreover,
    \begin{align*}
        \beta_{v_j}^m\wedge\omega^{n-m}=e^{v_j+M_j}f_j\omega^n\leq e^{M_j}f_j\omega^n=\frac{e^{M_j}}{c_j} \beta_{w_j}^m\wedge\omega^{n-m}.
    \end{align*}
    The domination principle \cref{m-positive have domination} therefore gives $e^{M_j}\geq c_j$ and hence the lower bound of $M_j$. The proof of the claim is finished.

    Return to the proof, let $M$ be an uniform upper bound of $M_j$. We can rewrite \eqref{eq 8} as
    \begin{equation}\label{eq 12}
 \beta_{v_j}^m\wedge\omega^{n-m}\leq e^Mf_j\omega^n.
    \end{equation}
    Since $\sup_Xv_j=0$, up to extracting a subsequence we may assume $\{v_j\}_j$ is a Cauchy sequence in $L^1(\omega^n)$. The right hand side of \eqref{eq 12} is uniformly bounded in $L^p(\omega^n)$, thus we can apply \cref{cor of stability} to conclude that $\{v_j\}_j$ is also a Cauchy sequence in $C(X)$. Therefore, $\{v_j\}_j$ converges uniformly to a function $v\in \operatorname{SH}_m(X,\beta,\omega)\cap C(X)$ solving
    $$
\beta_v^m\wedge\omega^{n-m}=e^{v+M^\prime}f\omega^n,
    $$
    where $M^\prime=\lim_jM_j$. Set $u:=v+M^\prime$, we obtain finally
     $$
\beta_u^m\wedge\omega^{n-m}=e^{u}f\omega^n.
    $$
\end{proof}
\subsection*{An approximation theorem}
As an application of the previous theorems, we prove the following approximation property of $(\beta,\omega,m)$-subharmonic functions, which is an extension of \cite[Lemma 3.20]{KN16} and \cite[Theorem 1.7]{GL25}.
\begin{theorem}\label{decreasing approximation}
  Let $\beta\in\Gamma_m(\omega)$ be an $(\omega,m)$-positive form. Then, for any $u\in \operatorname{SH}_m(X,\beta,\omega)$, there exists a decreasing sequence of smooth $(\beta,\omega,m)$-subharmonic functions on $X$ converging pointwise to $u$.
\end{theorem}
\begin{proof}
    The method is adapted from \cite[Lemma 3.20]{KN16}. Fix a function $u\in \operatorname{SH}_m(X,\beta,\omega)$. Since $\beta\in\Gamma_m(\omega)$, we have $\max(u,-j)\in \operatorname{SH}_m(X,\beta,\omega)$ and hence we may assume without loss of generality that $u$ is bounded. Pick a function $h\in C^\infty(X)$ such that $h\geq u$, it follows that $u\leq P_{\beta,m}(h)$. We are going to show that the envelope $P_{\beta,m}(h)$ can be approximated uniformly by a sequence of smooth $(\beta,\omega,m)$-subharmonic functions. Once this was done, we can let $h\searrow u$ and using a diagonal process to obtain a smooth decreasing approximation for $u$.

    Write $\beta_h^m\wedge\omega^{n-m}:=F\omega^n$ for some $F\in C^\infty(X)$, note that $F$ may not positive. Set $F_1:=\max(F,0)$ and choose a sequence of smooth non-negative functions $F_1^k$ decreasing to $F_1$. Using \cref{smooth Hessian for beta}, we can solve, for each $j\geq1$,
    \begin{equation}\label{eq 11}
\beta_{v_j^k}^m\wedge\omega^{n-m}=e^{j(v_j^k-h)}(F_1^k+\frac{1}{j})\omega^n,
    \end{equation}
    where $v_j^k\in \operatorname{SH}_m(X,\beta,\omega)\cap C^\infty(X)$. We claim that $v_j^k\leq h$. Indeed, assume $v_j^k-h$ attains its maximum at $x_0\in X$, then $dd^cv_j^k(x_0)\leq dd^ch(x_0)$. It follows that at $x_0$
    $$
 \beta_{v_j^k}^m\wedge\omega^{n-m}(x_0)\leq  \beta_{h}^m\wedge\omega^{n-m}(x_0).
    $$
    This combined with \eqref{eq 11} shows that
    \begin{align*}
   \beta_{h}^m\wedge\omega^{n-m}(x_0)\geq\beta_{v_j^k}^m\wedge\omega^{n-m}(x_0)\geq e^{j(v_j^k(x_0)-h(x_0))} \beta_{h}^m\wedge\omega^{n-m}(x_0)
    \end{align*}
    and hence $v_j^k\leq h$, which automatically implies that $v_j^k\leq P_{\beta,m}(h)$. Similar arguments using the maximum principle shows that $v_j^k$ is increasing in $j$. Now, by Dini's lemma we have $F_1^k\rightarrow F_1$ uniformly. Letting $k\rightarrow+\infty$, the proof of \cref{main thm twist for m-positive beta} tells us that $v_j^k$ tends uniformly to $v_j$ for every fixed $j$ and we have
    $$
\beta_{v_j}^m\wedge\omega^{n-m}=e^{j(v_j-h)}(F_1+\frac{1}{j})\omega^n.
    $$
Moreover, $v_j$ is still increasing with respect to $j$. Write $v_j\nearrow v\in \operatorname{SH}_m(X,\beta,\omega)$, it is clear that $v\leq P_{\beta,m}(h)\leq h$. Furthermore, the sequence $v_j$ is bounded above by $h$ and below by $v_1$, it is uniformly bounded. We also have 
$$
\beta_{v_j}^m\wedge\omega^{n-m}\leq(F_1+1)\omega^n.
$$
Using \cref{cor of stability} we infer that $v_j$ converges uniformly to $v$, this in turn implies that there is a subsequence of $v_j^k$ converges uniformly to $v$. We finally claim that $v=P_{\beta,m}(h)$.

Note that since $v\leq h$, $\underset{j\rightarrow\infty}{\lim}e^{j(v_j-h)}=\mathds{1}_{\{v=h\}}$. This combined with the monotone convergence theorem \cite[Corollary 4.11]{KN25a} gives 
$$
\beta_v^m\wedge\omega^{n-m}\leq\mathds{1}_{\{v=h\}}F_1\omega^n.
$$
This inequality gives that $\mathds{1}_{\{v<P_{\beta,m}(h)\}}\beta_v^m\wedge\omega^{n-m}=0$. The domination principle \cref{m-positive have domination} therefore yields that $v\geq P_{\beta,m}(h)$ and hence $v=P_{\beta,m}(h)$. The claim follows.
\end{proof}

\section{Sup-slopes for nef classes}\label{section:sup slopes}

Let $(X,\omega)$ be a compact Hermitian manifold of complex dimension $n$ equipped with a Hermitian metric $\omega$. Let $m$ be a positive integer such that $1\leq m\leq n$. 

In this section we introduce various slopes for nef Bott-Chern classes based on the recent remarkable work \cite{GS24}. Unless otherwise stated, we will always take the normalization $\int_X\omega^n=1$ in this section. First recall the following fundamental result in \cite{GS24}:

\begin{proposition}\label{prop:sup slope}
    Let $(X,\omega)$ be a compact Hermitian manifold of complex dimension $n$ and let $\chi$ be another Hermitian metric. Fix a function $\psi\in C^\infty(X)$ and let $\varphi\in \operatorname{PSH^+}(X,\chi)\cap C^\infty(X)$ be the unique solution to the equation
    $$
(\chi+dd^c\varphi)^n=ce^{\psi}\omega^n,\quad \sup_X\varphi=0.
    $$
    Then we have the following characterization of the constant $c>0$:
    \begin{equation}\label{eq 13}
        c=\underset{u\in C^\infty(X)\cap \operatorname{PSH^+}(X,\chi)}{\inf}\sup_X\frac{e^{-\psi}(\chi+dd^cu)^n}{\omega^n}.
    \end{equation}
    Here and after, we will always use the notation $\operatorname{PSH^+}(X,\chi)\cap C^\infty(X)$ to denote the set of smooth strictly $\chi$- plurisubharmonic functions.
\end{proposition}
Following \cite{GS24}, we introduce the following notion of sup-slopes:
\begin{definition}\label{def:sup slope}
    Let $\{\chi\}$ be a Hermitian class on $X$ (i.e., it contains a Hermitian metric). Let $\psi\in C^\infty(X)$ and $\omega$ be another Hermitian metric. We then define the sup-slope of the class $\{\chi\}$ with respect to $\psi$ and $\omega$ to be
    $$
SL_{\omega,\psi}(\{\chi\}):=\underset{u\in C^\infty(X)\cap \operatorname{PSH^+}(X,\chi)}{\inf}\sup_X\frac{e^{-\psi}(\chi+dd^cu)^n}{\omega^n}.
    $$
    It is clear from the definition that the sup-slope of $\{\chi\}$ is independent of the representative form $\chi$. Moreover, the classical maximum principle yields that $SL_{\omega,\psi}(\{\chi\})>0$ in any case, see \cite[Lemma 3.4]{GS24} for details.
\end{definition}
\begin{remark}\label{rmk:Kahler sup slope}
    If $\chi$ and $\omega$ are both K\"ahler forms, the characterization theorem \cref{prop:sup slope} implies that the sup-slope can be explicitly given by
    $$
SL_{\omega,\psi}(\{\chi\})=\frac{\int_X\chi^n}{\int_Xe^{\psi}\omega^n}.
    $$
\end{remark}
We show that the sup-slope is monotone with respect to the metric:
\begin{proposition}\label{monotone of slope}
    Let $\chi_1\leq\chi_2$ be two Hermitian metrics. Then for any $\psi\in C^\infty(X)$, 
    $$
SL_{\omega,\psi}(\{\chi_1\})\leq SL_{\omega,\psi}(\{\chi_2\}).
    $$
\end{proposition}
\begin{proof}
    By \cref{prop:sup slope}, there exists constants $c_1,c_2>0$ and functions $u_1\in \operatorname{PSH}(X,\chi_1),u_2\in\operatorname{PSH}(X,\chi_2)$ such that
    $$
(\chi_1+dd^cu_1)^n=c_1e^{\psi}\omega^n,\quad(\chi_2+dd^cu_2)^n=c_2e^{\psi}\omega^n,
    $$
    where $c_1=SL_{\omega,\psi}(\{\chi_1\})$ and $c_2=SL_{\omega,\psi}(\{\chi_2\})$. Then we have
    $$
(\chi_2+dd^cu_1)^n\geq(\chi_1+dd^cu_1)^n=c_1e^{\psi}\omega^n=\frac{c_1}{c_2}(\chi_2+dd^cu_2)^n.
    $$
    The classical maximum principle or the domination principle for Monge-Amp\`ere equations (see \cref{domination for beta}) immediately gives that $c_1\leq c_2$, whence our result.
\end{proof}
We next introduce a new notion called $p$-slope:
\begin{definition}\label{def of p-slope}
    Let $\chi,\omega$ be two Hermitian metrics on $X$ and $1\leq p<\infty$ be a fixed positive constant, we define the $p$-slope of $\chi$ with respect to $\omega$ to be 
    $$
SL_{\omega,p}(\chi):=\inf_{F\in C^\infty(X),\|e^F\|_{L^p(\omega^n)}\leq1}SL_{\omega,F}(\chi).
    $$
    We also define the $\infty$-slope as follows:
    $$
SL_{\omega,\infty}(\chi):=\inf_{F\in C^\infty(X),e^F\leq1}SL_{\omega,F}(\chi)=\underset{u\in C^\infty(X)\cap \operatorname{PSH^+}(X,\chi)}{\inf}\sup_X\frac{(\chi+dd^cu)^n}{\omega^n}.
    $$
\end{definition}
The $p$-slope is monotone with respect to $p$:
\begin{proposition}
     Let $\chi,\omega$ be two Hermitian metrics on $X$ and $1\leq p<q\leq\infty$. Then,
     $$
SL_{\omega,p}(\chi)\leq SL_{\omega,q}(\chi).
     $$
\end{proposition}
\begin{proof}
    Recall that we have made the assumption $\int_X\omega^n=1$. Jensen's inequality shows that for a fixed smooth function $F$, the norm $\|e^F\|_{L^p(\omega^n)}$ is increasing with respect to $p$ and hence $\|e^F\|_{L^q(\omega^n)}\leq1$ implies that $\|e^F\|_{L^p(\omega^n)}\leq1$. The result follows immediately from the definition.
\end{proof}

To illustrate the size of sup slopes, we first recall the lower volume introduced in \cite{GL22} and \cite{BGL25}:
\begin{definition}
    Let $\omega$ be a Hermitian metric on $X$, then we define the lower volume $\underline{\operatorname{Vol}}(\omega)$ of $\omega$ by
    $$
\underline{\operatorname{Vol}}(\omega):=\underset{u\in C^\infty(X)\cap \operatorname{PSH^+}(X,\omega)}{\inf}\int_X(\omega+dd^cu)^n=\underset{u\in L^\infty(X)\cap \operatorname{PSH}(X,\omega)}{\inf}\int_X(\omega+dd^cu)^n,
    $$
    where the second equality is a simple consequence of Demailly's regularization theorem.

    If $\{\beta\}$ is a nef Bott-Chern class, we define its lower volume as the following limits:
    $$
\underline{\operatorname{Vol}}(\{\beta\}):=\underset{\epsilon\rightarrow0}{\lim}\underline{\operatorname{Vol}}(\chi+\epsilon\omega).
    $$
\end{definition}

We prove that the $p$-slopes equal to the lower volume when $p=1$ and is always positive when $p>1$:
\begin{proposition}\label{slope=volume 1}
    Let $\omega,\chi$ be two Hermitian metrics on $X$. Then we have
    $$
SL_{\omega,1}(\chi)=\underline{\operatorname{Vol}}(\chi),
    $$
    and $SL_{\omega,p}(\chi)>0$ for each $p>1$.
\end{proposition}
\begin{proof}
    We first prove that $SL_{\omega,1}(\chi)=\underline{\operatorname{Vol}}(\chi)$. Fix a smooth function $F$ such that $\int_Xe^F\omega^n=1$. For any $u\in C^\infty(X)\cap\operatorname{PSH}(X,\chi)$,
    \begin{align*}
\int_X(\chi+dd^cu)^n&=\int_Xe^{F}\frac{e^{-F}(\chi+dd^cu)}{\omega^n}^n\omega^n\leq\sup_X\left\{\frac{e^{-F}(\chi+dd^cu)}{\omega^n}^n\right\}\int_Xe^F\omega^n\\
&=\sup_X\left\{\frac{e^{-F}(\chi+dd^cu)}{\omega^n}^n\right\}.
    \end{align*}
    Taking infimum on both sides first in $u$ and then in $F$ we derive that 
    $$
\underline{\operatorname{Vol}}(\{\chi\})\leq SL_{\omega,1}(\{\chi\}).
    $$
    Conversely, for each strictly $\chi$-psh function $u\in\operatorname{PSH^+}(X,\chi)\cap C^\infty(X)$, set $f:=\frac{\chi_u^n}{\omega^n}$ and $c:=\|f\|_{L^1(\omega^n)}$. Then we have by \cref{prop:sup slope} that
    $$
(\chi+dd^cu)^n=SL_{\omega,\log\frac{f}{c}}(\chi)\frac{f}{c}\omega^n.
    $$
Integrating both sides yields that
$$
SL_{\omega,\log\frac{f}{c}}(\chi)=\int_X(\chi+dd^cu)^n.
$$
Since $u$ was chosen arbitrarily and $\|\frac{f}{c}\|_{L^1(\omega^n)}=1$, the conclusion follows.

We next turn to prove the second statement. Fix $p>1$, it follows from the proof of \cite[Theorem 2.1]{GL23} or \cite[Lemma 3.3]{GL23} (see also \cref{subsolution 1}) that there is a constant $c=c(p,\omega,\chi)>0$ such that for each $0\leq f\in L^p(\omega^n)$ with $\|f\|_p\leq1$, there exists a subsolution $v\in\operatorname{PSH}(X,\chi)\cap L^\infty(X)$ satisfying
$$
(\chi+dd^cv)^n\geq cf\omega^n.
$$
For each $F\in C^\infty(X)$ and $\|e^F\|_{L^p(\omega^n)}=1$, \cref{prop:sup slope} and the above discussion yield that one can find $u\in\operatorname{PSH}(X,\chi)\cap C^\infty(X)$ and $v\in\operatorname{PSH}(X,\chi)\cap L^\infty(X)$ solving
$$
(\chi+dd^cu)^n=SL_{\omega,F}(\chi)\cdot e^F\omega^n,
$$
and
$$
(\chi+dd^cv)^n\geq c\cdot e^F\omega^n.
$$
Combing the above two inequalities we get
\begin{align*}
    (\chi+dd^cu)^n=SL_{\omega,F}(\chi)\cdot e^F\omega^n\leq\frac{SL_{\omega,F}(\chi)}{c}(\chi+dd^cv)^n.
\end{align*}
The domination principle for Monge-Amp\`ere equations \cite[Proposition 2.8]{GL22} (see also \cref{domination for beta 3}) give that $SL_{\omega,F}(\chi)\geq c=c(p,\omega,\chi)$. Taking infimum we finally obtain that
$$
SL_{\omega,p}(\chi)\geq c(p,\omega,\chi)>0.
$$
\end{proof}

We next introduce the notion of sup-slopes for nef classes:
\begin{definition}\label{def of slope in nef class}
    Let $\{\beta\}\in BC^{1,1}(X)$ be a nef Bott-Chern class and let $1\leq p<\infty$ be a fixed positive constant, we define the $p$-slope of $\{\beta\}$ with respect to $\omega$ to be
    $$
SL_{\omega,p}(\{\beta\}):=\lim_{\epsilon\rightarrow0}SL_{\omega,p}(\{\beta+\epsilon\omega\}).
    $$
  For $p=+\infty$, we also define the $\infty$-slope of $\{\beta\}$ with respect to $\omega$ to be
  $$
SL_{\omega,\infty}(\{\beta\})=\lim_{\epsilon\rightarrow0}SL_{\omega,\infty}(\{\beta+\epsilon\omega\})=\lim_{\epsilon\rightarrow0}\underset{u\in\operatorname{PSH^+}(X,\beta+\epsilon\omega)\cap C^\infty(X)}{\inf}\sup_X\frac{(\beta+\epsilon\omega+dd^cu)^n}{\omega^n}.
  $$
  Clearly, the definition of $p$-slopes in nef classes is independent of the approximating sequence $\beta+\epsilon\omega$ by the monotonicity \cref{monotone of slope}.
\end{definition}
The definition of slopes in nef classes is compatible with the one in Hermitian classes in the following sense:
\begin{proposition}
    Let $\{\beta\}$ be a Hermitian class, then $SL_{\omega,p}(\{\beta\})=\lim_{\epsilon\rightarrow0^+}SL_{\omega,p}(\{\beta+\epsilon\omega\})$ for each $1\leq p\leq\infty$.
\end{proposition}
\begin{proof}
    That $SL_{\omega,p}(\{\beta\})\leq\lim_{\epsilon\rightarrow0^+}SL_{\omega,p}(\{\beta+\epsilon\omega\})$ is clear. For the reverse direction,  \cref{prop:sup slope} easily implies that $SL_{\omega,p}(\{\lambda\beta\})=\lambda^nSL_{\omega,p}(\{\beta\})$ for each $\lambda>0$. Since $\{\beta\}$ is Hermitian, there is a sequence of positive numbers $\lambda_\epsilon\searrow0$ such that $\{\beta+\epsilon\omega\}\leq\{\lambda_\epsilon\beta\}$ and hence $SL_{\omega,p}(\{\beta+\epsilon\omega\})\leq\lambda_\epsilon^nSL_{\omega,p}(\{\beta\})$. The proof is concluded by letting $\epsilon\rightarrow0$.
\end{proof}
The following proposition is an immediate consequence of the definition and \cref{slope=volume 1}.

\begin{proposition}\label{Vol=slope 2}
    Let $\{\beta\}$ be a nef Bott-Chern class, then
    $$
SL_{\omega,1}(\{\beta\})=\underline{\operatorname{Vol}}(\beta).
    $$
\end{proposition}

\begin{proposition}\label{monotone of slope 2}
    Let $\{\beta_1\},\{\beta_2\}$ be two nef Bott-Chern classes such that $\beta_1\leq\beta_2$ and let $1\leq p\leq\infty$ be fixed. Then,
    $$
SL_{\omega,p}(\{\beta_1\})\leq SL_{\omega,p}(\{\beta_2\}).
    $$
    In particular, we obtain also that $\underline{\operatorname{Vol}}(\beta_1)\leq\underline{\operatorname{Vol}}(\beta_2)$ by \cref{Vol=slope 2}, which was proved in \cite{BGL25}.
\end{proposition}
\begin{proof}
    This is an immediate consequence of \cref{def of slope in nef class} and \cref{monotone of slope}. 
\end{proof}
We prove that under the bounded mass property introduced in \cite{BGL25}, all the $p$-slopes coincide:
\begin{lemma}\label{bounded mass property and slope}
    Let $(X,\omega)$ be a compact Hermitian manifold satisfying the bounded mass property and let $\{\beta\}\in H^{1,1}(X,\mathbb{R})$ be a nef cohomology class. Fix $1\leq p\leq\infty$. Then,
    $$
\int_X\beta^n=SL_{\omega,p}(\{\beta\})=\underline{\operatorname{Vol}}(\beta).
    $$
    Here we say that $(X,\omega)$ satisfies the bounded mass property if 
    $$\overline{\operatorname{Vol}}(\omega):=\sup_{u\in \operatorname{PSH}(X,\omega)\cap L^\infty(X)}\int_X(\omega+dd^cu)^n<+\infty.$$
\end{lemma}
\begin{proof}
   By \cite[Proposition 3.16]{BGL25} and \cref{Vol=slope 2} we have $SL_{\omega,1}(\{\beta\})=\underline{\operatorname{Vol}}(\beta)=\int_X\beta^n$. By the monotonicity \cref{monotone of slope 2} it is enough to show that $SL_{\omega,\infty}(\{\beta\})=\int_X\beta^n$. Up to rescaling we may assume that $\beta<\frac{1}{2}\omega$. For each $0<\epsilon<\frac{1}{2}$, there is a function $v\in\operatorname{PSH}(X,\beta+\epsilon\omega)\cap C^\infty(X)$ such that
   $$
(\beta+\epsilon\omega+dd^cv)^n=SL_{\omega,\infty}(\{\beta+\epsilon\omega\})\cdot \omega^n.
   $$
   Integrating both sides we have $SL_{\omega,\infty}(\{\beta+\epsilon\omega\})=\int_X(\beta+\epsilon\omega+dd^cv)^n$. Since $\{\beta\}$ is closed, we can furthermore write
   \begin{align*}
\int_X(\beta+\epsilon\omega+dd^cv)^n&=\int_X(\beta+dd^cv)^n+\sum_{k=1}^n\epsilon^k(\beta+dd^cv)^{n-k}\wedge\omega^k\\
&=\int_X\beta^n+\sum_{k=1}^n\epsilon^k((\beta-\omega)+(\omega+dd^cv))^{n-k}\wedge\omega^k\\
&=\int_X\beta^n+O(\epsilon).
   \end{align*}
   Where the last equality follows easily from the bounded mass property and \cite[Proposition 3.3]{GL22}. This yields that $SL_{\omega,\infty}(\{\beta+\epsilon\omega\})=\int_X\beta^n+O(\epsilon)$ and hence the result by letting $\epsilon\rightarrow0$.
\end{proof}

\begin{proposition}\label{right continuity}
    Let $\{\beta\}$ be a nef Bott-Chern class and $\omega$ be a Hermitian metric on $X$, then the function $[1,+\infty]\ni p\mapsto SL_{\omega,p}(\{\beta\})$ is right continuous.
\end{proposition}
\begin{proof}
   It is clear that we can assume without loss of generality that $\{\beta\}$ is a Hermitian class, the general case follows from a limiting process. Fix $p\geq1$ and set $a:=\lim_{q\searrow p}SL_{\omega,q}(\{\beta\})\geq SL_{\omega,p}(\{\beta\})$, our goal is to show that $SL_{\omega,p}(\{\beta\})\geq a$. Fix any $F\in C^\infty(X)$ such that $\|e^F\|_p\leq1$. Then for any $\epsilon>0$, we can find $\delta>0$ such that $\|e^F\|_{p+\delta}\leq1+\epsilon$ and hence $\|\frac{f}{1+\epsilon}\|_{p+\delta}\leq1$, where $f:=e^F$. It follows that
   $$
SL_{\omega,\log\frac{f}{1+\epsilon}}(\{\beta\})\geq SL_{\omega,p+\delta}(\{\beta\}).
   $$
   On the other hand, by definition we have $SL_{\omega,\log\frac{f}{1+\epsilon}}(\{\beta\})=(1+\epsilon)SL_{\omega,F}(\{\beta\})$. This implies that $(1+\epsilon)SL_{\omega,F}(\{\beta\})\geq a$ and hence $SL_{\omega,F}(\{\beta\})\geq a$. Take infimum in $F$ we arrive at the conclusion.
\end{proof}
\begin{remark}
  When $\{\beta\}$ is a Hermitian class, by \cref{right continuity}, we know that $SL_{\omega,p}(\{\beta\})\searrow\underline{\operatorname{Vol}}(\{\beta\})$ as $p\searrow1$ and $SL_{\omega,p}(\{\beta\})>0$ for each $p>1$. Therefore, in order to determine whether the manifold has the positive volume property, i.e. $\underline{\operatorname{Vol}}(\omega)>0$, one may try to get a uniform lower bound of $SL_{\omega,p}(\{\omega\})$. 

  Another interesting question is to determine whether the function $[1,+\infty]\ni p\mapsto SL_{\omega,p}(\{\beta\})$ is left continuous. In particular, whether $SL_{\omega,p}(\{\beta\})\nearrow SL_{\omega,\infty}(\{\beta\})$. We know from \cref{bounded mass property and slope} that this is the case when $\overline{\operatorname{Vol}}(\omega)<+\infty$.
\end{remark}

The Hessian analogue of sup slopes can also be introduced:
\begin{proposition}\cite[Theorem 1.2]{GS24} \label{prop:sup slope Hessian}
    Let $(X,\omega)$ be a compact Hermitian manifold of complex dimension $n$ and let $\eta\in\Gamma_m(\omega)$ be a strictly $(\omega,m)$-positive form on $X$. Fix a function $\psi\in C^\infty(X)$ and let $\varphi\in \operatorname{SH}_m(X,\eta,\omega)\cap C^\infty(X)$ be the unique solution to the equation
    $$
(\eta+dd^c\varphi)^m\wedge\omega^{n-m}=ce^{\psi}\omega^n,\quad \sup_X\varphi=0,
    $$
  the existence of the smooth solution was established in \cite[Proposition 21]{Sze18}. Then we have the following characterization of the constant $c>0$:
    \begin{equation}\label{eq 14}
        c=\underset{u\in C^\infty(X)\cap \operatorname{SH}^+_m(X,\eta,\omega)}{\inf}\sup_X\frac{e^{-\psi}(\eta+dd^cu)^m\wedge\omega^{n-m}}{\omega^n}.
    \end{equation}
\end{proposition}
\begin{definition}
    We say that a Bott-Chern class $\{\eta\}$ is $(\omega,m)$-nef if for each $\epsilon>0$ there is a smooth form $\gamma_\epsilon\in\Gamma_m(\omega)$ and a smooth function $u_\epsilon\in C^\infty(X)$ such that $\eta+\epsilon\omega+dd^cu_\epsilon\geq\gamma_\epsilon$.

    Clearly, an $(\omega,m)$ semi-positive class is $(\omega,m)$-nef. When $m=n$, this notion coincides with the usual notion of nef classes.
\end{definition}
\begin{definition}
    Similar to the Monge-Amp\`ere case, we use the notation $SL_{m,\omega,\psi}(\eta)$ to denote the Hessian slope
    $$
\underset{u\in C^\infty(X)\cap \operatorname{SH}^+_m(X,\eta,\omega)}{\inf}\sup_X\frac{e^{-\psi}(\eta+dd^cu)^m\wedge\omega^{n-m}}{\omega^n}.
    $$
    For any constant $p>1$, we can also define the Hessian $p$-slope by
    $$
SL_{m,\omega,p}(\eta):=\inf_{F\in C^\infty(X),\|e^F\|_{L^p(\omega^n)}\leq1}SL_{m,\omega,F}(\eta).
    $$
    Here $\operatorname{SH}^+_m(X,\eta,\omega)\cap C^\infty(X)$ denotes the set of smooth strictly $(\eta,\omega,m)$- subharmonic functions. 
\end{definition}
Similar to \cref{monotone of slope} and \cref{monotone of slope 2}, it is straightforward to check that the Hessian slope $SL_{m,\omega,p}(\eta)$ is monotone with respect to $\eta$ and $p$. Therefore, we can similarly introduce the sup-slopes for $(\omega,m)$-nef classes:

\begin{definition}\label{def of slope in m-nef class}
    Let $\{\eta\}\in BC^{1,1}(X)$ be an $(\omega,m)$-nef Bott-Chern class and let $1\leq p<\infty$ be a fixed positive constant, we define the Hessian $p$-slope of $\{\eta\}$ with respect to $\omega$ to be
    $$
SL_{m,\omega,p}(\{\eta\}):=\lim_{\epsilon\rightarrow0}SL_{m,\omega,p}(\{\eta+\epsilon\omega\}).
    $$
  Clearly, the definition of Hessian $p$-slopes in $(\omega,m)$-nef classes is independent of the approximating sequence $\eta+\epsilon\omega$ by the monotonicity of slopes.
\end{definition}

\section{Solving the Hessian equation for \texorpdfstring{$(\omega,m)$}{(omega,m)}-big classes}
Let $(X,\omega)$ be a compact Hermitian manifold of complex dimension $n$ equipped with a Hermitian metric $\omega$. Let $m$ be a positive integer such that $1\leq m\leq n$. 

In this subsection we will prove the following:
\begin{theorem}\label{twist equation for beta m-big}
     Let $\beta\in\overline{\Gamma_m(\omega)}$ be an $(\omega,m)$ semi-positive and $(\omega,m)$-big form. Set $0\leq f\in L^p(X,\omega^n)$ with $p>\frac{n}{m}$ and $\int_Xf\omega^n>0$. Fix a constant $\lambda>0$. Then there exists a unique bounded function $\varphi\in \operatorname{SH}_m(X,\beta,\omega)\cap L^\infty(X)$ such that
        $$
(\beta+dd^c\varphi)^m\wedge\omega^{n-m}=e^{\lambda\varphi}f\omega^n.
        $$
    Furthermore, we have
    $$
    \operatorname{osc} \varphi\leq C,
    $$ 
    where $C$ is a constant depends on $\beta,\omega,p,X,\|f\|_p$.
\end{theorem}
The strategy of the proof follows the idea in \cite{GL23}, we first construct a subsolution when $\beta\in\Gamma_m(\omega)$:
\begin{lemma}\label{subsolution 1}
    Let $\beta\in\Gamma_m(\omega)$ be an $(\omega,m)$-positive form and let $p>\frac{n}{m}$ be a fixed constant. Then there exists a uniform constant $l(p,\beta,X)>0$ such that for any $0\leq f\in L^p(\omega^n)$ with $\|f\|_p\leq1$, we can find a function $u\in \operatorname{SH}_m(X,\beta,\omega)\cap L^\infty(X)$ such that
    $$
(\beta+dd^cu)^m\wedge\omega^{n-m}\geq lf\omega^n,
    $$
    and
    $$
osc_Xu\leq C
    $$
    for some uniform constant $C$ depending on $\beta,\omega,p,X,\|f\|_p$.
\end{lemma}
\begin{proof}
    The idea is inspired by \cite[Theorem 2.1]{GL23}. Choose a finite double covering of X by small  coordinate balls $U_j\subset\subset U_j^{\prime}=\{\rho_j<0\}$ such that $\{U_j\}_j$ covers $X$, $1\leq j\leq N$. Where $\rho_j:X\rightarrow\mathbb{R}$ is a smooth function which is strictly plurisubharmonic in a neighborhood of $\overline{U_j^{\prime}}$. Solving the dirichlet problem
    \[
\begin{cases}
(dd^c u_j)^m\wedge\omega^{n-m} = f\omega^n, & \text{on } U_j^{\prime}, \\
u_j= -1, & \text{on } \partial{U_j^{\prime}}. \\
u_j\in C(\overline{U_j^{\prime}})\cap \mathrm{SH}_m(U_j^\prime,\omega).
\end{cases}
\]
The existence of solutions follows from \cref{local dirichlet for omega} or \cite[Theorem 8.2]{KN25a}. The stability result \cite[Proposition 8.1]{KN25a} ensures that $u_j$ is uniformly bounded in $U_j^\prime$.
Choose a sequence of positive numbers $\lambda_j>1$ such that $\lambda_j\rho_j<u_j$ on $U_j$. Then we define
$$
w_j:=\begin{cases}
\max(u_j,\lambda_j\rho_j), & \text{on } U_j^{\prime}, \\
\lambda_j\rho_j, & \text{on } X-U_j^\prime. \\
\end{cases}
$$
Then $w_j$ coincides with $u_j$ in $U_j$ and is smooth where it is not $(\omega,m)$-subharmonic, thus we may find a positive number $\delta>0$ independent of $j$ such that $\delta w_j$ lie in $\operatorname{SH}_m(X,\beta,\omega)$. Set
$$
u:=\frac{\delta}{N}\sum_{j=1}^Nw_j.
$$
We have in $U_j$
\begin{align*}
    (\beta+dd^cu)^m\wedge\omega^{n-m}&\geq(\frac{1}{N}\beta+\frac{1}{N}dd^c(\delta w_j))^m\wedge\omega^{n-m}\\
    &=(\frac{1}{N}\beta+\frac{1}{N}dd^c(\delta u_j))^m\wedge\omega^{n-m}\\
    &\geq\frac{\delta^m}{N^m}(dd^cu_j)^m\wedge\omega^{n-m}=\frac{\delta^m}{N^m}f\omega^n.
\end{align*}
The second equality is because $w_j=u_j$ on $U_j$ and hence the proof is concluded.
\end{proof}

\begin{lemma}\label{subsolution 2}
     Let $\beta\in\overline{\Gamma_m(\omega)}$ be an $(\omega,m)$ semi-positive and $(\omega,m)$-big form. Let $p>\frac{n}{m}$ be a fixed constant. Then there exists a uniform constant $c(p,\beta,X)>0$ such that for any $0\leq f\in L^p(\omega^n)$ with $\|f\|_p\leq1$, we can find a function $v\in \operatorname{SH}_m(X,\beta,\omega)\cap L^\infty(X)$ such that
    $$
(\beta+dd^cv)^m\wedge\omega^{n-m}\geq cf\omega^n,\quad0\leq v\leq1
    $$
\end{lemma}
\begin{proof}
   The proof is inspired by \cite[Lemma 3.3]{GL23}. Fix a potential $\rho\in \operatorname{SH}_m(X,\beta,\omega)$ with analytic singularities such that $\beta+dd^c\rho\geq\gamma\in\Gamma_m(\omega)$ (see \cref{def of m-big}) and $\sup_X\rho=-1$. Since $\rho$ is quasi-plurisubharmonic, $\rho$ belongs to $L^r(\omega^n)$ for any $r>1$. It follows from the H\"older's inequality that there is a number $q>\frac{n}{m}$ such that $|\rho|^{2m}f\in L^q(\omega^n)$. By \cref{subsolution 1} we can find a positive number $l=l(p,\gamma,X)>0$ and $u\in \operatorname{SH}_m(X,\beta,\omega)\cap L^\infty(X)$ such that
   $$
(\gamma+dd^cu)^m\wedge\omega^{n-m}\geq l|\rho|^{2m}f\omega^n,\quad\sup_Xu=-1.
   $$
   Since we have assumed that $\beta+dd^c\rho\geq\gamma\in\Gamma_m(\omega)$, the function $u+\rho\in \operatorname{SH}_m(X,\beta,\omega)$ and $\beta+dd^c(u+\rho)\geq\gamma+dd^cu$. Set $v:=-\frac{1}{u+\rho}=\chi\circ(u+\rho)$, where $\chi(t):=-\frac{1}{t}$ is a convex increasing function on $\mathbb{R}^-$. We can write
   \begin{align*}
       \beta+dd^cv=&\beta+dd^c\chi\circ({u+\rho})\\
       =&\beta+\chi^{\prime\prime}\circ(u+\rho)d(u+\rho)\wedge d^c(u+\rho)+\chi^{\prime}\circ(u+\rho)dd^c(u+\rho)\\
       \geq&(1-\chi^\prime\circ(u+\rho))\beta+\chi^\prime\circ(u+\rho)(\beta+dd^c(u+\rho))\\
       \geq&\left(1-\frac{1}{(u+\rho)^2}\right)\beta+\frac{1}{(u+\rho)^2}(\gamma+dd^cu).
   \end{align*}
   Note that by our normalization $u+\rho\leq-2$ and hence $1-\frac{1}{(u+\rho)^2}>0$. This combined with the fact that $\beta\in\overline{\Gamma_m(\omega)}$ gives that
   $$
(\beta+dd^cv)^m\wedge\omega^{n-m}\geq\frac{1}{(u+\rho)^{2m}}(\gamma+dd^cu)^m\wedge\omega^{n-m}\geq\frac{l|\rho|^{2m}f}{(u+\rho)^{2m}}\omega^n.
   $$
   Since $u\leq-1$ is uniformly bounded by \cref{subsolution 1} and $\rho\leq-1$, we get
   $$
(\beta+dd^cv)^m\wedge\omega^{n-m}\geq cf\omega^n,\quad0\leq v\leq1.
   $$
   Here it is clear from our construction that $0\leq v\leq1$ and we can take $c:=\frac{l}{(1+\|u\|_\infty)^{2m}}$, which depends on $\gamma,\omega,p,X$.
\end{proof}

\begin{remark}
    As a corollary, we infer that an $(\omega,m)$-big form has positive Hessian $p$-slopes for each $p>\frac{n}{m}$.
\end{remark}

We now return to the proof of \cref{twist equation for beta m-big}.
\begin{proof}[Proof of \cref{twist equation for beta m-big}]
    We may assume without loss of generality that $\lambda=1$. Since $\beta+\epsilon\omega$ lies in $\Gamma_m(\omega)$ for any $\epsilon>0$, we can use \cref{main thm twist for m-positive beta} to find $u_\epsilon\in \operatorname{SH}_m(X,\beta+\epsilon\omega,\omega)\cap C(X)$ such that
    $$
(\beta+\epsilon\omega+dd^cu_\epsilon)^m\wedge\omega^{n-m}=e^{u_\epsilon}f\omega^n.
    $$
    Choose $\epsilon_1>\epsilon_2>0$, we have
\begin{align*}
    e^{-u_{\epsilon_1}}(\beta+\epsilon_1\omega+dd^cu_{\epsilon_1})^m\wedge\omega^{n-m}&=f\omega^n=e^{-u_{\epsilon_2}}(\beta+\epsilon_2\omega+dd^cu_{\epsilon_2})^m\wedge\omega^{n-m}\\
    &\leq e^{-u_{\epsilon_2}}(\beta+\epsilon_1\omega+dd^cu_{\epsilon_2})^m\wedge\omega^{n-m}.
\end{align*}
Since $\beta+\epsilon_1\omega\in\Gamma_m(\omega)$, it is non-collapsing by \cref{cor:m-positive is non-collapsing} and hence the domination principle \cref{m-positive have domination} gives that $\{u_\epsilon\}_\epsilon$ is decreasing as $\epsilon$ decreases to $0$. On the other hand, by \cref{subsolution 2} we can find $c>0$ $u\in \operatorname{SH}_m(X,\beta,\omega)$ such that 
$$
(\beta+dd^cu)^m\wedge\omega^{n-m}\geq c\frac{f}{\|f\|_p}\omega^n,\quad-1\leq u\leq0.
$$
This implies that
$$
(\beta+\epsilon\omega+dd^cu)^m\wedge\omega^{n-m}\geq c\frac{f}{\|f\|_p}\omega^n\geq e^{u+\log\frac{c}{\|f\|_p}}f\omega^n
$$
since $u\leq0$. Therefore, we have $u+\log\frac{c}{\|f\|_p}\leq u_\epsilon$ by the domination principle \cref{m-positive have domination} again for any $\epsilon>0$. It follows that $u_\epsilon$ decreases to some function $\varphi\in \operatorname{SH}_m(X,\beta,\omega)\cap L^\infty(X)$ and we deduce easily from the decreasing convergence theorem \cite[Corollary 4.11]{KN25a} that
$$
(\beta+dd^c\varphi)^m\wedge\omega^{n-m}=e^{\varphi}f\omega^n.    
$$
The uniqueness is an easy consequence of the domination principle \cref{m-positive have domination} again.
\end{proof}

\section{A priori estimates for complex Hessian equations in nef classes}
In this section, let $\{\beta\}$ be a (possibly non closed) nef Bott-chern class, i.e., for each $\epsilon>0$, $\{\beta+\epsilon\omega\}$ contains a Hermitian metric. Assume moreover that there is a bounded $(\beta,\omega,m)$-potential $\rho\in \operatorname{SH}_m(X,\beta,\omega)\cap L^\infty(X)$. Since $\{\beta+t\omega\}$ are Hermitian classes for each $t>0$, we can apply \cite[Proposition 21]{Sze18} to solve the following family of complex Hessian equations:
\begin{equation}\label{eq 15}
    (\beta_t+dd^c\varphi_t)^m\wedge\omega^{n-m}=c_te^{F_t}\omega^n,\quad\sup_X\varphi_t=0,
\end{equation}
where $F_t\in C^\infty(X)$ is a family of smooth functions on $X$, $c_t>0$, $\beta_t:=\beta+t\omega$ and $\varphi_t\in C^\infty(X)\cap \operatorname{SH}_m(X,\beta+t\omega,\omega)$. Let $V_t:=\sup\{v\,|\,v\in \operatorname{SH}_m(X,\beta_t,\omega),v\leq0\}$ be the largest $(\beta_t,\omega,m)$-subharmonic function lying below $0$.
Our first goal of this section is to establish the following $L^\infty$ estimates, extending the main result in \cite{GPTW24}:
\begin{theorem}\label{a priori in nef class}
    Let $\{\beta\}\in BC^{1,1}(X)$ be a nef Bott-Chern class satisfying $SL_{\omega,a}(\{\beta\})>0$ for some constant $1\leq a<\frac{n}{n-m}$. Assume also $\varphi_t\in \operatorname{SH}_m(X,\beta_t,\omega)\cap C^\infty(X)$ satisfying
\begin{equation}\label{eq hessian}
(\beta_t+dd^c\varphi_t)^m\wedge\omega^{n-m}= c_te^{F_t}\omega^n,\quad\sup_X\varphi_t=0.
\end{equation}
Fix a constant $p>na$. Then there exists a uniform constant $C$ depending on $\beta,\omega,m,n,p$, the upper bound of $\|e^{\frac{an}{m}F_t}\|_{L^1(\log L)^p}:=\int_{X}e^{\frac{an}{m}F_t(z)}(1+\frac{an}{m}|F_t(z)|)^p\omega^n,E_t:=\int_X(-\varphi_t+V_t)^ae^{\frac{an}{m}F_t}\omega^n$ and the lower bound of  $SL_{\omega,a}(\{\beta\})$ such that
$$
0\leq-\varphi_t+V_t\leq C.
$$
\end{theorem}
Our method depends largely on that in \cite{GPT23}, \cite{GPTW24}, where the authors proved the analogue of \cref{a priori in nef class} on compact K\"ahler manifolds. We will need the following approximation theorem:

\begin{lemma}\label{uniform approximation of envelope}
    Let $\omega,\chi$ be two Hermitian forms on $X$ and set $P_{\chi,m}(h):=\sup\{v\,|\,v\in \operatorname{SH}_m(X,\chi,\omega),v\leq h\}$, where $h$ is an arbitrary smooth function on $X$. Then there exists a sequence of smooth $(\chi,\omega,m)$-subharmonic functions $v_j$ converging uniformly to $P_{\chi,m}(h)$.
\end{lemma}
\begin{proof}
    The result is already contained in the proof of \cref{decreasing approximation} .
\end{proof}

\begin{lemma}\label{exp lemma}
    Consider \eqref{eq 15}. Let $\{\beta\}\in BC^{1,1}(X)$ be a nef Bott-Chern class satisfying $SL_{\omega,a}(\{\beta\})>0$ for some constant $1\leq a<\frac{n}{n-m}$. For any $s>0$, set $\Omega_s:=\{\varphi_t-V_t\leq -s\}$ and assume that $\mathring{\Omega}_s$ is non empty, then there are uniform constants $C$ and $\alpha$ such that
\begin{equation}
    \int_{\Omega_s}\exp\left(\alpha\frac{ (-\varphi_t+V_t-s)^{\frac{n+1}{n}}}{A_{s,t}^{\frac{1}{n}}}\right)\omega^n\leq Ce^{CE_t}.
\end{equation}
Where $A_{s,t}:=\left(\int_{\Omega_s}(-\varphi_t+V_t-s)^ae^{\frac{an}{m}F_t}\omega^n\right)^{\frac{1}{a}}>0$ and $E_t:=\left(\int_X(-\varphi_t+V_t)^ae^{\frac{an}{m}F_t}\omega^n\right)^{\frac{1}{a}}$. The constant $C$ depends on $\beta,\omega,m,n$, the lower bound of $\int_Xe^{\frac{F_t}{m}}\omega^n$ and the lower bound of $SL_{\omega,a}(\{\beta\})$, while the constant $\alpha$ depends on $\beta,\omega,n$.
\end{lemma}
\begin{proof}
    The proof is inspired by \cite[Theorem 1]{GPTW24}. Choosing a sequence of smooth functions $\tau_k:\mathbb{R}\rightarrow\mathbb{R}^+$ such that $\tau_k(x)$ decreases to the function $x\cdot\mathds{1}_{\mathbb{R}^+}(x)$ as $k\rightarrow+\infty$. For example, we may take
    $$
    \tau_k(x):=\frac{1}{2}(\sqrt{x^2+\frac{1}{k}}+x).
    $$
    For each fixed $t>0$, using \cref{uniform approximation of envelope} we can find a sequence $u_{l,t}\in \operatorname{SH}_m(X,\beta_t,\omega)\cap C^\infty(X)$ converging uniformly to $V_t:=\sup\{v\,|\,v\in \operatorname{SH}_m(X,\beta_t,\omega),v\leq0\}$ as $l\rightarrow+\infty$. Using the main theorem of \cite{TW10}, we can then solve the auxiliary Monge-Amp\`ere equation
    \begin{equation}\label{eq 16}
        (\beta_t+dd^c\psi_{t,l,k,s})^n=c_{t,l,k,s}^\prime\frac{\tau_k(-\varphi_t+u_{l,t}-s)}{A_{s,k,l,t}}e^{\frac{n}{m}F_t}\omega^n,\quad\sup_X\psi_{t,l,k,s}=0,
    \end{equation}
    for some $\psi_{t,l,k,s}\in PSH(X,\beta_t)\cap C^\infty(X)$ and $c_{t,l,k,s}^\prime>0$. Where $s>0$ is a positive number and 
    $$
A_{s,k,l,t}:=\|\tau_k(-\varphi_t+u_{l,t}-s)e^{\frac{n}{m}F}\|_{L^a(\omega^n)}=\left(\int_X\tau_k(-\varphi_t+u_{l,t}-s)^ae^{\frac{an}{m}F}\omega^n\right)^{\frac{1}{a}}>0.
    $$
    We first fix $t$ and $k$. Since $u_{l,t}$ converges uniformly to $V_t$ and $\psi_{t,l,k,s}\leq V_t$ by definition, we may assume without loss of generality that $\psi_{t,l,k,s}<u_{l,t}+1$ by choosing $l$ large enough. Set
    $$
\Phi:=-\epsilon(-\psi_{t,l,k,s}+u_{l,t}+1+\Lambda)^{\frac{n}{n+1}}-(\varphi_t-u_{l,t}+s),
    $$
    where
    $$
\epsilon^{n+1}=\epsilon_{s,k,l,t}^{n+1}:=\frac{c_t^{\frac{n}{m}}}{c_{t,l,k,s}^\prime}A_{s,k,l,t}\left(\frac{n+1}{n}\right)^n,\quad\Lambda=\Lambda_{s,k,l,t}:=\left(\frac{n}{n+1}\right)^{n+1}\epsilon^{n+1}.
    $$
Since $\Phi$ is a smooth function on the compact manifold $X$, it must achieves maximum at some point $x_0\in X$. Recall that $\Omega_s:=\{\varphi_t+s\leq V_t\}$. If $x_0\in X-\Omega_s$, then   
$$
\Phi(x_0)\leq-(\varphi_t-u_{l,t}+s)\leq u_{l,t}-V_t\leq\epsilon_l.
$$
Where the first inequality follows from our assumption $\psi_{t,l,k,s}<u_{l,t}+1$ and $\epsilon_l\rightarrow0$ as $l\rightarrow+\infty$. 

We next assume that $x_0\in\Omega_s$. The computation is similar but slightly different from that in \cite[Lemma 1]{GPT23}.
Write $\hat{\beta_t}:=\beta_t+dd^c\varphi_t\in\Gamma_m(\omega)$ such that $\hat{\beta_t}^m\wedge\omega^{n-m}=c_te^{F_t}\omega^n$. At $x_0$ we have $dd^c\Phi(x_0)\leq0$ and hence $-dd^c\Phi(x_0)\in\overline{\Gamma_m(\omega)}$. Therefore, we can write
\begin{align*}
    0\geq& n\frac{dd^c\Phi(x_0)\wedge\hat{\beta_t}^{m-1}\wedge\omega^{n-m}}{\hat{\beta_t}^{m}\wedge\omega^{n-m}}\\
    =&\frac{n\hat{\beta_t}^{m-1}\wedge\omega^{n-m}}{\hat{\beta_t}^{m}\wedge\omega^{n-m}}\wedge\left\{(-\epsilon)\frac{n}{n+1}(-\psi_{t,l,k,s}+u_{l,t}+1+\Lambda)^{-\frac{1}{n+1}}[-(\beta_t+dd^c\psi_{t,l,k,s})+(\beta_t+dd^cu_{l,t})]\right\}\\
    &+\frac{n\hat{\beta_t}^{m-1}\wedge\omega^{n-m}}{\hat{\beta_t}^{m}\wedge\omega^{n-m}}\wedge\left\{-(\beta_t+dd^c\varphi_{t})+(\beta_t+dd^cu_{l,t})\right\}\\
    &+\frac{n\hat{\beta_t}^{m-1}\wedge\omega^{n-m}}{\hat{\beta_t}^{m}\wedge\omega^{n-m}}\wedge\left\{\frac{n\epsilon}{(n+1)^2}(-\psi_{t,l,k,s}+u_{l,t}+1+\Lambda)^{-\frac{n+2}{n+1}}\cdot d(-\psi_{t,l,k,s}+u_{l,t})\wedge d^c(-\psi_{t,l,k,s}+u_{l,t})\right\}\\    
    \geq&\frac{n\hat{\beta_t}^{m-1}\wedge\omega^{n-m}}{\hat{\beta_t}^{m}\wedge\omega^{n-m}}\wedge\left\{\frac{n\epsilon}{n+1}(-\psi_{t,l,k,s}+u_{l,t}+1+\Lambda)^{-\frac{1}{n+1}}[(\beta_t+dd^c\psi_{t,l,k,s})-(\beta_t+dd^cu_{l,t})]\right\}\\
    &+\frac{n\hat{\beta_t}^{m-1}\wedge\omega^{n-m}}{\hat{\beta_t}^{m}\wedge\omega^{n-m}}\wedge\left\{[-(\beta_t+dd^c\varphi_{t})+(\beta_t+dd^cu_{l,t})]\right\}\\
    =&\left\{\frac{n^2\epsilon}{n+1}(-\psi_{t,l,k,s}+u_{l,t}+1+\Lambda)^{-\frac{1}{n+1}}\right\}\frac{(\beta_t+dd^c\psi_{t,l,k,s})\wedge\hat{\beta_t}^{m-1}\wedge\omega^{n-m}}{\hat{\beta_t}^{m}\wedge\omega^{n-m}}-n\\
    &+\left(1-\frac{n\epsilon}{n+1}(-\psi_{t,l,k,s}+u_{l,t}+1+\Lambda)^{-\frac{1}{n+1}}\right)\frac{n(\beta_t+dd^cu_{l,t})\wedge\hat{\beta_t}^{m-1}\wedge\omega^{n-m}}{\hat{\beta_t}^{m}\wedge\omega^{n-m}}\\
    \geq&\frac{n^2\epsilon}{n+1}(-\psi_{t,l,k,s}+u_{l,t}+1+\Lambda)^{-\frac{1}{n+1}}\frac{(c_{t,l,k,s}^\prime)^{\frac{1}{n}}\tau_k(-\varphi_t+u_{l,t}-s)^{\frac{1}{n}}}{c_t^{\frac{1}{m}}A_{s,k,l,t}^{\frac{1}{n}}}\\
    &-n+\left(1-\frac{n\epsilon}{n+1}\Lambda^{-\frac{1}{n+1}}\right)\frac{n(\beta_t+dd^cu_{l,t})\wedge\hat{\beta_t}^{m-1}\wedge\omega^{n-m}}{\hat{\beta_t}^{m}\wedge\omega^{n-m}}\\
    =&\frac{n^2\epsilon}{n+1}(-\psi_{t,l,k,s}+u_{l,t}+1+\Lambda)^{-\frac{1}{n+1}}\frac{(c_{t,l,k,s}^\prime)^{\frac{1}{n}}\tau_k(-\varphi_t+u_{l,t}-s)^{\frac{1}{n}}}{c_t^{\frac{1}{m}}A_{s,k,l,t}^{\frac{1}{n}}}-n\\   
    \geq&\frac{n^2\epsilon}{n+1}(-\psi_{t,l,k,s}+u_{l,t}+1+\Lambda)^{-\frac{1}{n+1}}\frac{(c_{t,l,k,s}^\prime)^{\frac{1}{n}}(-\varphi_t+u_{l,t}-s)^{\frac{1}{n}}}{c_t^{\frac{1}{m}}A_{s,k,l,t}^{\frac{1}{n}}}-n&.
\end{align*}
Where in the fifth inequality we have used the mixed type inequalities and the fact that $\psi_{t,l,k,s}+u_{t,l}+1\geq0$, in the sixth equality we have used the definition of $\Lambda$, in the seventh inequality we have used the assumption of $\tau_k$. Let us explain the fifth inequality in detail: indeed, it follows from \eqref{eq 16} and the mixed type inequality for Monge-Amp\`ere equations that
$$
(\beta_t+dd^c\psi_{t,l,k,s})^m\wedge\omega^{n-m}\geq \left(c_{t,l,k,s}^\prime\frac{\tau_k(-\varphi_t+u_{l,t}-s)}{A_{s,k,l,t}}\right)^{\frac{m}{n}}e^{F_t}\omega^n.
$$
This combined with \eqref{eq hessian} and the Garding's inequality (see also \cref{prop:mixed type for chi}) yields that 
\begin{align*}
    (\beta_t+dd^c\psi_{t,l,k,s})\wedge\hat{\beta_t}^{m-1}\wedge\omega^{n-m}&\geq\left(c_{t,l,k,s}^\prime\frac{\tau_k(-\varphi_t+u_{l,t}-s)}{A_{s,k,l,t}}\right)^{\frac{1}{n}}e^{\frac{1}{m}F_t}\cdot c_t^{\frac{m-1}{m}}\cdot e^{\frac{m-1}{m}F_t}\omega^n\\
    &=\left(c_{t,l,k,s}^\prime\frac{\tau_k(-\varphi_t+u_{l,t}-s)}{A_{s,k,l,t}}\right)^{\frac{1}{n}}\cdot c_t^{\frac{m-1}{m}}e^{F_t}\omega^n,
\end{align*}
and hence the fifth inequality follows.

It follows that at $x_0\in\Omega_s$,
\begin{align*}
    (-\varphi_t+u_{l,t}-s)&\leq A_{s,k,l,t}\left(\frac{n+1}{n\epsilon}\right)^n(-\psi_{t,l,k,s}+u_{l,t}+1+\Lambda)^{\frac{n}{n+1}}\frac{c_t^{\frac{n}{m}}}{c_{t,l,k,s}^\prime}\\
    &=\epsilon(-\psi_{t,l,k,s}+u_{l,t}+1+\Lambda)^{\frac{n}{n+1}}.
\end{align*}
This implies that $\Phi(x_0)\leq0$. 

Combining the above two cases we obtain that $\sup_X\Phi\leq\epsilon_l\rightarrow0$ as $l\rightarrow\infty$. We need to let $l\rightarrow\infty$ and study the limit behavior. 

Consider \eqref{eq 16}, since $u_{t,l}\rightarrow V_t$ uniformly as $l\rightarrow\infty$, we have $\tau_k(-\varphi_t+u_{t,l}-s)\rightarrow\tau_k(-\varphi_t+V_t-s)$ uniformly and $A_{s,k,l,t}\rightarrow A_{s,k,t}:=\left(\int_X\tau_k(-\varphi_t+V_t-s)^ae^{\frac{an}{m}F_t}\omega^n\right)^{\frac{1}{a}}>0$. Observe that $\frac{\tau_k(-\varphi_t+u_{l,t}-s)}{A_{s,k,l,t}}e^{\frac{n}{m}F_t}$ are smooth positive functions for all $l$ and converges uniformly to $\frac{\tau_k(-\varphi_t+V_t-s)}{A_{s,k,t}}e^{\frac{n}{m}F_t}>0$, which in turn implies that $c_{t,l,k,s}^\prime\rightarrow c_{t,k,s}^\prime>0$ by the characterization \cref{prop:sup slope}. Consequently, we obtain that $c_{t,l,k,s}^\prime\frac{\tau_k(-\varphi_t+u_{l,t}-s)}{A_{s,k,l,t}}e^{\frac{n}{m}F_t}$ is a Cauchy sequence in $C(X)$ with respect to $l$ and uniformly bounded away from $0$. The famous stability result \cite[Theorem A]{KN19} yields that $\{\psi_{t,l,k,s}\}_l$ is also a Cauchy sequence in $C(X)$ and hence converges uniformly to some $\psi_{t,k,s}\in PSH(X,\beta_t)\cap C(X)$ with $\sup_X\psi_{t,k,s}=0$ again.

Return to the inequality $\Phi\leq\epsilon_l\rightarrow0$. Letting $l\rightarrow\infty$ on both sides yield that
$$
(-\varphi_t+V_t-s)\leq \epsilon_{t,k,s}(-\psi_{t,k,s}+V_t+1+\Lambda_{t,k,s})^{\frac{n}{n+1}},
$$
here
$$
\epsilon_{t,k,s}^{n+1}:=\lim_l\epsilon_{s,k,l,t}^{n+1}=\frac{c_t^{\frac{n}{m}}}{c_{t,k,s}^\prime}A_{s,k,t}\left(\frac{n+1}{n}\right)^n,
$$
and
$$
\Lambda_{t,k,s}:=\lim
_l\Lambda_{s,k,l,t}=\left(\frac{n}{n+1}\right)^{n+1}\epsilon_{t,k,s}^{n+1}.
$$
Which in turn implies that
\begin{align}\label{eq 20}
    (-\varphi_t+V_t-s)^{\frac{n+1}{n}}\leq&\epsilon_{t,k,s}^{\frac{n+1}{n}}(-\psi_{t,k,s}+V_t+1+\Lambda_{t,k,s})\\
    =&\frac{c_t^{\frac{1}{m}}}{(c_{t,k,s}^\prime)^{\frac{1}{n}}}A_{s,k,t}^{\frac{1}{n}}\left(\frac{n+1}{n}\right)(-\psi_{t,k,s}+V_t+1+\Lambda_{t,k,s}).
\end{align}
We next claim that $c_t$ is uniformly bounded above and $c_{t,l,k,s}^\prime$ is uniformly bounded below and away from $0$. For $c_t$, combining \eqref{eq 15} and Garding's inequality, we deduce that
$$
c_t^{\frac{1}{m}}\int_Xe^{\frac{F_t}{m}}\omega^n\leq\int_X(\beta_t+dd^c\varphi_t)\wedge\omega^{n-1}=\int_X(\beta+t\omega)\wedge\omega^{n-1}+\int_X\varphi_tdd^c(\omega^{n-1})\leq C_1,
$$
because $\sup_X\varphi_t=0$. Where $C_1$ is a constant depending on $n, X,\omega, \beta$. 
It remains to derive the lower bound of $c_{t,l,k,s}^\prime$. 

Consider \eqref{eq 16}, since $\int_X\left(\frac{\tau_k(-\varphi_t+u_{l,t}-s)}{A_{s,k,l,t}}e^{\frac{n}{m}F_t}\right)^a\omega^n=1$, our assumption and \cref{prop:sup slope} yields that
$$
c_{t,l,k,s}^\prime\geq SL_{\omega,a}(\{\beta_t\})\geq SL_{\omega,a}(\{\beta\})>0.
$$
Similarly, $\Lambda_{t,k,s}\leq C_2A_{t,k,s}$ by the very definition, where $C_2$ is a constant depending on $m,n,\beta,\omega$, the lower bound of $\int_Xe^{\frac{F_t}{m}}\omega^n$ and the lower bound of $SL_{\omega,a}(\{\beta\})$. Therefore, \eqref{eq 20} can be simply written as
\begin{align*}
    (-\varphi_t+V_t-s)^{\frac{n+1}{n}}\leq C_3 A_{s,k,t}^{\frac{1}{n}}(-\psi_{t,k,s}+V_t+1+C_2 A_{t,k,s})
\end{align*}
where $C_3$ is a uniform constant depending on $m,n,\beta,\omega$, the lower bound of $\int_Xe^{\frac{F_t}{m}}\omega^n$ and the lower bound of $SL_{\omega,a}(\{\beta\})$. Note that $V_t\leq0$ and $\sup_X\psi_{t,k,s}=0$, hence we can use the $\alpha$-invariant estimate (or Skoda's uniform integrability theorem) as in \cite{GPT23} to derive a constant $\alpha=\alpha(m,n,\beta,\omega)$ such that
\begin{equation}
    \int_{\Omega_s}\exp\left(\alpha\frac{ (-\varphi_t+V_t-s)^{\frac{n+1}{n}}}{A_{s,k,t}^{\frac{1}{n}}}\right)\omega^n\leq\int_{\Omega_s}\exp\left(\alpha C_3(-\psi_{t,k,s}+1+C_2 A_{t,k,s})\right)\leq Ce^{CA_{t,k,s}},
\end{equation}
where $C$ is a uniform constant depending on $m,n,\beta,\omega$, the lower bound of $\int_Xe^{\frac{F_t}{m}}\omega^n$ and the lower bound of $SL_{\omega,a}(\{\beta\})$
Letting $k\rightarrow\infty$ we obtain that
\begin{equation}
     \int_{\Omega_s}\exp\left(\alpha\frac{ (-\varphi_t+V_t-s)^{\frac{n+1}{n}}}{A_{s,t}^{\frac{1}{n}}}\right)\omega^n\leq Ce^{CA_{t,s}}.
\end{equation}
Observing that $A_{t,s}\leq E_t$ and hence the conclusion follows.
\end{proof}

\begin{proof}[Proof of \cref{a priori in nef class}]
    Having \cref{exp lemma}, the argument is now standard and can almost be copied line by line from \cite{GPT23} or \cite{GPTW24}, for convenience of the reader we briefly repeat it here. Fix $p>na$ and define $\eta:\mathbb{R}^+\rightarrow\mathbb{R}^+$ by $\eta(x):=[\log(1+x)]^p$. Observe that $\eta$ is an increasing function with $\eta(0)=0$ and let $\eta^{-1}(y)=\exp(y^{\frac{1}{p}})-1$ be its inverse function. Recall that Young's inequality yields for any numbers $u,v\geq0$,
    \begin{align}\label{eq 21}
        uv&\leq\int_0^u\eta(x)dx+\int_0^v\eta^{-1}(y)dy=\int_0^u\eta(x)dx+\int_0^{\eta^{-1}(v)}x\eta^{\prime}(x)dx\\
        &\leq u\cdot\eta(u)+v\cdot\eta^{-1}(v)=u[\log(1+u)]^p+v(\exp(v^{\frac{1}{p}})-1).
    \end{align}
Set
$$
v:=\frac{\alpha}{2}\frac{ (-\varphi_t+V_t-s)^{\frac{n+1}{n}}}{A_{s,t}^{\frac{1}{n}}}
$$
We apply \eqref{eq 21} with $u=e^{\frac{n}{m}F}$ and $v^p$ to derive
\begin{align}
    v(z)^pe^{\frac{an}{m}F_t(z)}\leq e^{\frac{an}{m}F_t(z)}[\log(1+e^{\frac{an}{m}F_t(z)})]^p+v(z)^p(e^{v(z)}-1)\leq e^{\frac{an}{m}F_t(z)}(1+\frac{an}{m}|F_t(z)|)^p+C_pe^{2v(z)}.
\end{align}
Plugging the value of $v$ and integrating both sides over $\Omega_s$ we obtain
\begin{align}
   & \left(\frac{\alpha}{2}\right)^p\int_{\Omega_s}\frac{ (-\varphi_t+V_t-s)^{\frac{p(n+1)}{n}}}{A_{s,t}^{\frac{p}{n}}}e^{\frac{an}{m}F_t}\omega^n\\
    \leq&\int_{\Omega_s}e^{\frac{an}{m}F_t(z)}(1+\frac{an}{m}|F_t(z)|)^p\omega^n+C_p\int_{\Omega_s} \exp\left({\alpha\frac{ (-\varphi_t+V_t-s)^{\frac{n+1}{n}}}{A_{s,t}^{\frac{1}{n}}} }\right)\omega^n\leq C.\\
\end{align}
    Where the last inequality follows from \cref{exp lemma} and $C$ is a uniform constant depending on $n,m,p,\beta,\omega$, the lower bound of $\int_Xe^{\frac{F_t}{m}}\omega^n$, the lower bound of $SL_{\omega,a}(\{\beta\})$, the upper bound of $\|e^{\frac{an}{m}F_t}\|_{L^1(\log L)^p}:=\int_{X}e^{\frac{an}{m}F_t(z)}(1+\frac{an}{m}|F_t(z)|)^p\omega^n$ and the upper bound of $E_t$. As a consequence, we can apply H\"older's inequality and invoking the definition of $A_{s,t}$ to deduce that
    \begin{align}
        A_{s,t}=&\left(\int_{\Omega_s}(-\varphi_t+V_t-s)^ae^{\frac{an}{m}F_t}\omega^n\right)^{\frac{1}{a}}\leq\left(\int_{\Omega_s}(-\varphi_t+V_t-s)^{\frac{p(n+1)}{n}}e^{\frac{an}{m}F_t}\right)^{\frac{n}{p(n+1)}}\cdot\left(\int_{\Omega_s}e^{\frac{an}{m}F_t}\omega^n\right)^{\frac{1}{q}}\\
        \leq&CA_{s,t}^{\frac{1}{n+1}}\cdot\left(\int_{\Omega_s}e^{\frac{an}{m}F_t}\omega^n\right)^{\frac{1}{q}}.
    \end{align}
    Where $q:=\frac{p(n+1)}{p(n+1)-na}$. The above inequality can be written by
    $$
A_{s,t}\leq C\left(\int_{\Omega_s}e^{\frac{an}{m}F_t}\omega^n\right)^{\frac{n+1}{nq}}.
    $$
    Observe that $\frac{n+1}{nq}=\frac{p(n+1)-na}{pn}:=1+\delta_0>1$. Define $\phi:\mathbb{R}\rightarrow\mathbb{R}$ by $\phi(s):=\left(\int_{\Omega_s}e^{\frac{an}{m}F_t}\omega^n\right)^{\frac{1}{a}}$. By the definition of $A_{s,t}$ it is easy to check that
    $$
r\phi(s+r)=\left(\int_{\Omega_{s+r}}r^ae^{\frac{an}{m}F_t}\right)^{\frac{1}{a}}\leq\left(\int_{\Omega_{s}}(-\varphi_t+V_t-s)^ae^{\frac{an}{m}F_t}\right)^{\frac{1}{a}}=A_{s,t}\leq C\phi(s)^{a(1+\delta_0)},
    $$
    for any $0\leq r\leq1$. The result then follows from a classic lemma due to De Giorgi, see \cite[Lemma 2]{GPT23}. The proof is therefore concluded.
\end{proof}

\section{Solving Complex Hessian equations in nef classes}

Let $(X,\omega)$ be a compact Hermitian manifold of complex dimension $n$ equipped with a Hermitian metric $\omega$. Let $m$ be a positive integer such that $1\leq m\leq n$. 

In this section, we   apply \cref{a priori in nef class} to solve the degenerate complex Hessian equations in nef classes with some more assumptions.

\begin{theorem}\label{solve nef 1}
    Let $\{\beta\}\in BC^{1,1}(X)$ be a nef Bott-Chern class satisfying $SL_{\omega,a}(\{\beta\})>0$ for some constant $1\leq a<\frac{n}{n-m}$. Assume moreover that there is a function $\rho\in \operatorname{SH}_m(X,\beta,\omega)\cap L^\infty(X)$ and that $\{\beta\}$ is $(\omega,m)$-non-collapsing. Fix $0\leq f\in L^p(\omega^n)$ with $p>\frac{n^2}{m^2-(1-\frac{1}{a})mn}$ satisfying $\int_Xf\omega^n>0$. Then, there exists $\varphi\in \operatorname{SH}_m(X,\beta,\omega)\cap L^\infty(X)$ and $c>0$ solving
$$
(\beta+dd^c\varphi)^m\wedge\omega^{n-m}=cf\omega^n,\quad\sup_X\varphi=0.
$$
Moreover, $osc_X\varphi$ is bounded by a constant $C$ depending on $\beta,\omega,\|f\|_p,n,m$, the bound of $\rho$ and the lower bound of $SL_{\omega,a}(\{\beta\})$.
\end{theorem}
\begin{proof}
    The idea is to proceed the approximation process. Let $\beta_j:=\beta+2^{-j}\omega$ and choose a sequence of smooth functions $f_j\rightarrow f$ in $L^p(\omega^n)$. Thanks to \cite[Proposition 21]{Sze18} we can find $\varphi_j\in \operatorname{SH}_m(X,\beta_j,\omega)\cap C^\infty(X)$ and $c_j>0$ solving
    \begin{equation}
        (\beta_j+dd^c\varphi_j)^m\wedge\omega^{n-m}=c_jf_j\omega^n,\quad\sup_X\varphi_j=0.
    \end{equation}
    We first claim that $\{c_j\}_j$ is uniformly bounded away from zero. The upper bound is easily obtained by using Garding's inequality as in \cref{main thm for m-positive beta}. It remains to derive a lower bound for $\{c_j\}_j$. Since $\{\beta_j\}$ is a Hermitian class, the main theorem of \cite{TW10} yields that there is $u_j\in PSH(X,\beta_j)\cap C^\infty(X)$ and $b_j>0$ such that
    $$
(\beta_j+dd^cu_j)^n=b_jf_j^{\frac{n}{m}}\omega^n,\quad\sup_Xu_j=0.
    $$
    Write $f_j^{\frac{n}{m}}=e^{F_j}$ for some $F_j\in C^\infty(X)$. \cref{prop:sup slope} tells us
    $b_j=SL_{\omega,F_j}(\beta_j)$. Since $p>\frac{n^2}{m^2-(1-\frac{1}{a})mn}$, it is easy to check that $p>\frac{an}{m}$ and $f_j\rightarrow f$ in $L^p(\omega^n)$, it follows from H\"older's inequality that $f_j^{\frac{n}{m}}\rightarrow f^{\frac{n}{m}}$ in $L^a(\omega^n)$ and hence $\int_Xe^{aF_j}\omega^n=\int_Xf_j^{\frac{an}{m}}\omega^n\geq \epsilon_0$ is uniformly bounded. This implies that 
    $$
b_j\geq{\epsilon_0^\prime}SL_{\omega,a}(\{\beta\}):=b>0
    $$
    is uniformly bounded below by our condition. On the other hand, the mixed type inequality for Monge-Amp\`ere equations yields that
    $$
(\beta_j+dd^cu_j)^m\wedge\omega^{n-m}\geq b_j^{\frac{m}{n}}f_j\omega^n\geq b^{\frac{m}{n}}f_j\omega^n=\frac{b^{\frac{m}{n}}}{c_j}(\beta_j+dd^c\varphi_j)^m\wedge\omega^{n-m}.
    $$
    Since $\beta_j$ is $(\omega,m)$-non-collapsing by \cref{cor:m-positive is non-collapsing}, the domination principle \cref{domination for beta} yields that $c_j\geq b^{\frac{m}{n}}>0$ and hence the claim follows.

  We next claim that $\{\varphi_j\}_j$ is uniformly bounded. Writing $f_j=e^{G_j}$. Since $\rho\in \operatorname{SH}_m(X,\beta,\omega)\cap L^\infty(X)$, we may assume that $\rho\leq0$ and hence $\rho\leq V_{\beta_j}$ for all $j$, which implies that $\{V_{\beta_j}\}_j$ is uniformly bounded. Invoking \cref{a priori in nef class}, it is enough to show that $\|e^{\frac{an}{m}G_j}\|_{L^1(\log L)^p}:=\int_{X}e^{\frac{an}{m}G_j(z)}(1+\frac{an}{m}|G_j(z)|)^p\omega^n$ and $E_j:=\left(\int_X(-\varphi_j+V_{\beta_j})^af_j^{\frac{an}{m}}\omega^n\right)^{\frac{1}{a}}$ are both uniformly bounded above with respect to $j$. The first statement follows easily from the fact that $f_j\rightarrow f$ in $L^p(\omega^n)$ and $p>\frac{an}{m}$. Set $p^\prime:=\frac{pm}{an}$ and $q:=\frac{1}{1-\frac{1}{p^\prime}}$. It is clear that $1<aq<\frac{n}{n-m}$ and $f_j^{\frac{an}{m}}\rightarrow f^{\frac{an}{m}}$ in $L^{p^\prime}(\omega^n)$. H\"older's inequality gives that
  $$
E_j\leq\left(\int_X(-\varphi_j+V_{\beta_j})^{aq}\omega^n\right)^{\frac{1}{aq}}\cdot\left(\int_Xf_j^p\omega^n\right)^{\frac{1}{ap^\prime}},
  $$
  which is uniformly bounded above by \cref{integrablity 2} since $1<aq<\frac{n}{n-m}$ and $\sup_X\varphi_j=\sup_XV_j=0$. The second claim follows.
  
  Up to extracting a subsequence we can assume that $c_j\rightarrow c>0$ and $\varphi_j\rightarrow\varphi\in \operatorname{SH}_m(X,\beta,\omega)\cap L^\infty(X)$ in $L^1(\omega^n)$ and almost everywhere. Set
\[
\Phi_j := P_{\beta,m} \left( \inf_{k \geq j} \varphi_k \right), \quad \psi_j := \left( \sup_{k \geq j} \varphi_k \right)^*.
\]
From the preceding discussion, $\Phi_j$ and $\psi_j$ are uniformly bounded thanks to the uniform boundedness of $\varphi_j$ and the existence of $\rho\in \operatorname{SH}_m(X,\beta,\omega)\cap L^\infty(X)$. The Hartog's lemma \cref{Hartogs' lemma} yields $\psi_j \searrow \varphi$. For each fixed $j$ and $l\geq j$,
$$
(\beta_j+dd^c\varphi_l)^m\wedge\omega^{n-m}\geq(\beta_l+dd^c\varphi_l)^m\wedge\omega^{n-m}=c_lf_l\omega^n.
$$
According to the maximum principle,
\begin{align*}
   (\beta_j+dd^c\max(\varphi_l,\varphi_{l+1}))^m\wedge\omega^{n-m}\geq \min{(c_l,c_{l+1})}\min(f_l,f_{l+1})\omega^n.
\end{align*}
An induction argument yields that
\begin{equation}\label{eq 22}
    (\beta_j+dd^c\psi_j)^m\wedge\omega^{n-m}\geq(\inf_{l\geq j}c_l)(\inf_{l\geq j}f_l)\omega^n.
\end{equation}
Since $c_j\rightarrow c>0$ and $f_j\rightarrow f$ almost everywhere, we have $\inf_{l\geq j}c_l$ increases to $c$ and $\inf_{l\geq j}f_l$ increases to $f$ almost everywhere, which in turn implies that the measure $(\inf_{l\geq j}c_l)(\inf_{l\geq j}f_l)\omega^n\rightarrow cf\omega^n$ weakly. Taking limits in \eqref{eq 22},  the standard monotone convergence theorem yields that
\begin{equation}\label{eq 24}
(\beta + dd^c \varphi)^n \geq cf\omega^n.
\end{equation}
On the other hand, since $\beta\leq\beta_j$ and $(\beta_j+dd^c\varphi_j)^m\wedge\omega^{n-m}=e^{\varphi_j-\varphi_j}c_jf_j\omega^n$, we can use \cref{contact3} to write
$$
(\beta+dd^cP_{\beta,m}(\varphi_j,\varphi_{j+1}))^m\wedge\omega^{n-m}\leq e^{P_{\beta,m}(\varphi_j,\varphi_{j+1})}\max(c_je^{-\varphi_j}f_j,c_{j+1}e^{-\varphi_{j+1}}f_{j+1})\omega^n.
$$
An easy induction argument gives that
\begin{equation}\label{eq 23}
    (\beta+dd^c\Phi_j)^m\wedge\omega^{n-m}\leq e^{\psi_j-\inf_{l\geq j}\varphi_l}(\sup_{l\geq j}c_l)(\sup_{l\geq j}f_l)\cdot\omega^n.
\end{equation}
It is clear from the definitions that $\Phi_j$ increases almost everywhere to $\psi$ for some $\psi\in \operatorname{SH}_m(X,\beta,\omega)\cap L^\infty(X)$ and $\sup_{l\geq j}c_l\rightarrow c$, $\sup_{l\geq j}f_l$ decreases to $f$ almost everywhere hence weakly. Letting $j\rightarrow\infty$ in \eqref{eq 23} we get
$$
  (\beta+dd^c\psi)^m\wedge\omega^{n-m}\leq e^{\psi-\varphi}cf\omega^n.
$$
This combined with \eqref{eq 24} yields that
$$
e^{-\psi}\beta_\psi^m\wedge\omega^{n-m}\leq e^{-\varphi}cf\omega^n\leq e^{-\varphi}\beta_\varphi^m\wedge\omega^{n-m}.
$$
Since $\beta$ is assumed to be $(\omega,m)$-non-collapsing, the domination principle therefore gives that $\varphi\leq\psi$. The reverse inequality $\varphi\geq\psi$ is clear from the definition, thus we obtain $\varphi=\psi$ and 
$$
\beta_\varphi^m\wedge\omega^{n-m}=cf\omega^n.
$$
The proof is therefore concluded.
\end{proof}
\begin{remark}\label{rmk of nef 1}
    In \cref{solve nef 1}, we have to assume $p>\frac{n^2}{m^2}$ when $a=1$ to get the uniform control of the energy $E_j$. When the potential $\rho\in\operatorname{PSH}(X,\beta)\cap L^\infty(X)$, we can actually improve the index to $p>\frac{n}{m}$.
\end{remark}
When $(X,\omega)$ is compact K\"ahler and $\beta$ is closed, we immediately get the following corollary:
\begin{corollary}\label{Corollary Kahler}
    Let $(X,\omega)$ be a compact Hermitian manifold satisfying the bounded mass property (in particular this is true if $(X,\omega)$ is K\"ahler) and let $\{\beta\}\in H^{1,1}(X,\mathbb{R})$ be a nef cohomology class such that $\int_X\beta^n>0$. Assume moreover that $\beta$ admits a bounded Hessian potential $\rho\in \operatorname{SH}_m(X,\beta,\omega)\cap L^\infty(X)$ and that $\beta$ is $(\omega,m)$-non-collapsing. Fix $0\leq f\in L^p(\omega^n)$ with $p>\frac{n^2}{m^2}$ satisfying $\int_Xf\omega^n>0$. Then, there exists $\varphi\in \operatorname{SH}_m(X,\beta,\omega)\cap L^\infty(X)$ and $c>0$ solving
$$
(\beta+dd^c\varphi)^m\wedge\omega^{n-m}=cf\omega^n,\quad\sup_X\varphi=0.
$$
Moreover, $osc_X\varphi$ is bounded by a constant $C$ depending on $\beta,\omega,n,m,\|f\|_p$ and the bound of $\rho$.
\end{corollary}
\begin{proof}
    This is an immediate consequence of \cref{bounded mass property and slope} and \cref{solve nef 1}.
\end{proof}

As an application, we can also derive partially the main result of \cite{GL25}, when the exponent $p$ is slightly larger:
\begin{corollary}[{\cite[Theorem A]{GL25}}]\label{GL25 main result}
    Assume $\beta$ is semi-positive and big. Fix $0\leq f\in L^p(\omega^n)$ with $p>\frac{n^2}{m^2}$ satisfying $\int_Xf\omega^n>0$. Then, there exists $\varphi\in \operatorname{SH}_m(X,\beta,\omega)\cap L^\infty(X)$ and $c>0$ solving
$$
(\beta+dd^c\varphi)^m\wedge\omega^{n-m}=cf\omega^n,\quad\sup_X\varphi=0.
$$
Moreover, $osc_X\varphi$ is bounded by a constant $C$ depending on $\beta,\omega,\|f\|_p,n,m$.
\end{corollary}
\begin{proof}
    Since $\beta$ is semi-positive and big, we have that $\beta$ is nef and that $\beta$ is $(\omega,m)$-non-collapsing by \cref{cor:m-positive is non-collapsing 2} (see also \cite[Theorem 2.1]{GL25}). 

    We claim that $SL_{\omega,a}(\beta)>0$ for any $a>1$. Fix such $a>1$, it follows from \cite[Lemma 3.3]{GL23} that there exists a uniform positive constant $c=c(a,\beta,\omega)$ such that for each $0\leq f\in L^a(\omega^n)$ with $\|f\|_a\leq1$, we can find a subsolution $v\in\operatorname{PSH}(X,\beta)\cap L^\infty(X)$ such that
    $$
(\beta+dd^cv)^n\geq cf\omega^n.
    $$
    For each $t>0$ and $F\in C^\infty(X)$ with $\|e^F\|_a\leq1$, we can find $u\in\operatorname{PSH}(X,\beta+t\omega)\cap C^\infty(X)$ such that
    $$
(\beta+t\omega+dd^cu)^n=SL_{\omega,F}(\beta+t\omega)e^F\omega^n.
    $$
    By the above discussion, we can also find $v\in\operatorname{PSH}(X,\beta)\cap L^\infty(X)$ such that
    $$
(\beta+dd^cv)^n\geq c\cdot e^F\omega^n.
    $$
    Hence we can write
    \begin{align*}
(\beta+t\omega+dd^cu)^n&=SL_{\omega,F}(\beta+t\omega)e^F\omega^n\leq\frac{SL_{\omega,F}(\beta+t\omega)}{c}(\beta+dd^cv)^n\\
&\leq\frac{SL_{\omega,F}(\beta+t\omega)}{c}(\beta+t\omega+dd^cv)^n.
    \end{align*}
    The domination principle \cite[Proposition 2.8]{GL22} yields that $SL_{\omega,F}(\beta+t\omega)\geq c(a,\beta,\chi)>0$. Take infimum in $F$ and $t$ therefore gives that $SL_{\omega,a}(\{\beta\})\geq c>0$. The claim follows.

    Since $p>\frac{n^2}{m^2}$, there exists a large constant $a>1$ such that $p>\frac{n^2}{m^2-(1-\frac{1}{a})mn}$. This implies that $\beta$ meets all the assumptions in \cref{solve nef 1} and hence the result follows.
\end{proof}
We next turn to solve degenerate Hessian equations with an exponential twist on the right-hand side. 

\begin{theorem}\label{solve nef 2}
    Let $\{\beta\}\in BC^{1,1}(X)$ be a nef Bott-Chern class satisfying $SL_{\omega,a}(\{\beta\})>0$ for some positive constant $1\leq a<\frac{n}{n-m}$. Assume moreover that there is a function $\rho\in \operatorname{SH}_m(X,\beta,\omega)\cap L^\infty(X)$. Fix $0\leq f\in L^p(\omega^n)$ with $p>\frac{n^2}{m^2-(1-\frac{1}{a})mn}$ satisfying $\int_Xf\omega^n>0$. Then, there exists a unique $\varphi\in \operatorname{SH}_m(X,\beta,\omega)\cap L^\infty(X)$ solving
$$
(\beta+dd^c\varphi)^m\wedge\omega^{n-m}=e^{\varphi}f\omega^n.
$$
Moreover, the lower bound of $\varphi$ is bounded by a constant $C$ depending on $\beta,\omega,\|f\|_p,n,m$, the bound of $\rho$ and the lower bound of $SL_{\omega}(\{\beta\})$. Furthermore, if $\beta$ is $(\omega,m)$-non-collapsing, the solution is unique.
\end{theorem}
\begin{proof}
Uniqueness follows immediately from the domination principle \cref{domination for beta}. Let $\beta_j:=\beta+2^{-j}\omega$ and choose a sequence of smooth positive functions $f_k=e^{F_k}$ converging to $f$ in $L^p(\omega^n)$. Using \cref{smooth Hessian for beta} we can find $\varphi_{j,k}\in \operatorname{SH}_m(X,\beta_j,\omega)\cap C^\infty(X)$ solving
    \begin{equation}\label{eq 25}
        (\beta_j+dd^c\varphi_{j,k})^m\wedge\omega^{n-m}=e^{\varphi_{j,k}}f_k\omega^n.
    \end{equation}
    We first claim that $\varphi_{j,k}$ are uniformly bounded from below. By \cite[Proposition 21]{Sze18}, we can find $\psi_{j,k}\in  \operatorname{SH}_m(X,\beta_j,\omega)\cap C^\infty(X)$ and $c_{j,k}>0$ solving
    \begin{equation}
          (\beta_j+dd^c\psi_{j,k})^m\wedge\omega^{n-m}=c_{j,k}f_k\omega^n,\quad\sup_X\psi_{j,k}=0.
    \end{equation}
    Exactly the same argument as in \cref{solve nef 1} shows that $c_{j,k}$ are uniformly bounded away from zero and that $\psi_{j.k}$ are uniformly bounded. This implies that
    \begin{align*}
        e^{-\varphi_{j,k}}(\beta_j+dd^c\varphi_{j,k})^m\wedge\omega^{n-m}&=f_k\omega^n=\frac{1}{c_{j,k}}   (\beta_j+dd^c\psi_{j,k})^m\wedge\omega^{n-m}\\
        &\leq e^{-(\psi_{j,k}+\log c_{j,k})}   (\beta_j+dd^c(\psi_{j,k}+\log c_{j,k}))^m\wedge\omega^{n-m},
    \end{align*}
    where the last inequality follows from the fact that $\psi_{j,k}\leq0$. Since $\{\beta_j\}$ is a Hermitian class, the domination principle yields that $\varphi_{j,k}\geq\psi_{j,k}+\log c_{j,k}$ and hence the claim follows.

    Now fix $j$, \eqref{eq 25} fits in the proof in \cref{main thm twist for m-positive beta}, so it is immediate that $\varphi_{j,k}\rightarrow\varphi_j\in \operatorname{SH}_m(X,\beta_j,\omega)\cap C(X)$ uniformly such that 
    $$
(\beta_j+dd^c\varphi_{j})^m\wedge\omega^{n-m}=e^{\varphi_{j}}f\omega^n.
    $$
    By the domination principle \cref{domination for beta 2} again, $\varphi_j$ is decreasing with respect to $j$. Since $\varphi_{j,k}$ are uniformly bounded, so do $\varphi_j$. This implies that $\varphi_j\searrow\varphi\in \operatorname{SH}_m(X,\beta_j,\omega)\cap L^\infty(X)$ such that
    $$
(\beta+dd^c\varphi)^m\wedge\omega^{n-m}=e^{\varphi}f\omega^n.
    $$
    The proof is therefore concluded.
\end{proof}
The following corollary is immediate:
\begin{corollary}\label{Corollary Kahler twist}
    Let $(X,\omega)$ be a compact Hermitian manifold satisfying the bounded mass property (in particular this is true if $(X,\omega)$ is K\"ahler) and let $\{\beta\}\in H^{1,1}(X,\mathbb{R})$ be a nef cohomology class such that $\int_X\beta^n>0$. Moreover, assume that $\beta$ admits a bounded Hessian potential $\rho\in \operatorname{SH}_m(X,\beta,\omega)\cap L^\infty(X)$. Fix $0\leq f\in L^p(\omega^n)$ with $p>\frac{n^2}{m^2}$ satisfying $\int_Xf\omega^n>0$. Then, there exists $\varphi\in \operatorname{SH}_m(X,\beta,\omega)\cap L^\infty(X)$ solving
$$
(\beta+dd^c\varphi)^m\wedge\omega^{n-m}=e^{\varphi}f\omega^n.
$$
Moreover, the lower bound of $\varphi$ is bounded by a constant $C$ depending on $\beta,\omega,n,m,\|f\|_p$ and the bound of $\rho$.
\end{corollary}
Parallel to \cref{GL25 main result}, the following corollary is also immediate:
\begin{corollary}\label{GL25 main result 2}
    Assume $\beta$ is semi-positive and big. Fix $0\leq f\in L^p(\omega^n)$ with $p>\frac{n^2}{m^2}$ satisfying $\int_Xf\omega^n>0$. Then, there exists a unique $\varphi\in \operatorname{SH}_m(X,\beta,\omega)\cap L^\infty(X)$ solving
$$
(\beta+dd^c\varphi)^m\wedge\omega^{n-m}=e^{\varphi}f\omega^n.
$$
Moreover, $osc_X\varphi$ is bounded below by a constant $C$ depending on $\beta,\omega,\|f\|_p,n,m$.
\end{corollary}
\begin{proof}
    As explained in \cref{GL25 main result}, $\beta$ satisfies all the assumptions in \cref{solve nef 2} and hence the equation is solvable. The uniqueness follows from the domination principle since $\beta$ is big and hence $(\omega,m)$-non-collapsing.
\end{proof}
\section{Stability of Hessian equations}
Let $(X,\omega)$ be a compact Hermitian manifold of complex dimension $n$ equipped with a Hermitian metric $\omega$. Let $m$ be a positive integer such that $1\leq m\leq n$. 

In this section, we study the stability of complex Hessian equations in $m$-positive cone and in nef classes.

\subsection{  Stability of complex Hessian equations in the m-positive cone}
In this section, let $\beta\in\Gamma_m(\omega)$ be a strictly $(\omega,m)$- positive form on $X$. As before, let $B=B(\beta,\omega)>0$ satisfies, on $X$,
$$
-B\omega^2\leq dd^c\omega\leq B\omega^2,\quad-B\omega^3\leq d\omega\wedge d^c\omega\leq B\omega^3,
$$
$$
-B\omega^2\leq dd^c\beta\leq B\omega^2,\quad-B\omega^3\leq d\beta\wedge d^c\beta\leq B\omega^3.
$$
and
$$
-B\omega^3\leq d\beta\wedge d^c\omega\leq B\omega^3,\quad-B\omega^3\leq d\omega\wedge d^c\beta\leq B\omega^3.
$$
We prove the following stability and uniqueness result: \begin{theorem}\label{stability for m-positive hessian}
     Suppose that $\varphi_1, \varphi_2 \in \mathrm{SH}_m(X, \beta,\omega) \cap L^{\infty}(X)$, satisfy
$$
\beta_{\varphi_1}^m \wedge \omega^{n-m}=e^{\lambda \varphi_1}f_1 \omega^n, \quad \beta_{\varphi_2}^m \wedge \omega^{n-m}=e^{\lambda \varphi_2}f_2 \omega^n,
$$
where $f_1, f_2 \in L^p\left(\omega^n\right), p>n / m$ (When $\lambda=0$,  we need to assume that
$f _1\geqslant c_0>0$ for some constant $c_0$). Then
$$
\|\varphi_1-\varphi_2\|_{\infty} \leqslant C\|f_1-f_2\|_p^{\frac{1}{m}}
$$
where the constant $C$ depends on $c_0, p,\|f_1\|_p,\|f_2\|_p, \omega, X$, $B$, $n$, $m$, $\beta$.
 \end{theorem}
\begin{proof}
Here we follow the proof of \cite[Proposition 3.16]{KN16}, \cite[Theorem 3.1]{KN19} and \cite[Theorem 3.5]{GL23}.

We first consider the simpler case where $\lambda>0$. Rescaling, we can assume without loss of generality that $\lambda=1$. We first observe that, by \cref{main thm twist for m-positive beta}, $\varphi_1, \varphi_2$ are uniformly bounded. Since $\varphi_1, \varphi_2$ plays the same role, it suffices to show that
$$
\varphi_1-\varphi_2 \leq C\left\|f_1-f_2\right\|_p^{\frac{1}{m}}.
$$
By the uniqueness we may also assume without loss of generality that $\|f_1-f_2\|_p>0$. Using \cref{main thm for m-positive beta}, we solve
$$
\left(\beta+d d^c u\right)^m \wedge \omega^{n-m}=c\left(\frac{\left|f_1-f_2\right|}{\left\|f_1-f_2\right\|_p}+1\right) \omega^n
$$
with $\sup _X (u-\varphi_i)\leq 0$ for $i=1,2$. Moreover, $u, c, c^{-1}$ are uniformly bounded from the proof of \cref{main thm for m-positive beta}. We can assume that
$$
\varepsilon:=e^{\frac{\sup _X\varphi_1-\log c}{m}}\left\|f_1-f_2\right\|_p^{\frac{1}{m}}
$$
is small enough. We next consider $\phi:=(1-\varepsilon) \varphi_1+\varepsilon u+ m\log(1-\varepsilon)$,
\begin{align*}
   & \left(\beta+d d^c \phi\right)^m \wedge \omega^{n-m}\\ \geq &e^{\varphi_1+m \log(1-\varepsilon)}f_1 \omega^n+\varepsilon^m \cdot c\cdot\left( \frac{|f_1-f_2|}{\|f_1-f_2 \|_p}+1\right)\omega^n\\
   \geq &  e^{\varphi_1+m \log(1-\varepsilon)}f_1 \omega^n+e^{\varphi_1} |f_1-f_2| \omega^n  \\
   \geq &e^{\varphi_1+m \log(1-\varepsilon)} f_2 \omega^n\\
   \geq & e^\phi f_2 \omega^n,
   \end{align*}
   where the last inequality follows from our assumption that $\sup _X (u-\varphi_i)\leq 0$ for $i=1,2$. The domination principle gives $\phi \leq \varphi_2$ , hence the desired estimate. 

Now we treat the case $\lambda=0$.

Throughout the proof we use $C$ to denote various uniform constants. By \cref{main a priori Hessian estimates} we have that $\sup _X\left(\left|\varphi_1\right|+\left|\varphi_2\right|\right) \leq C$.

\textbf{Step 1.} We first assume that $\varphi_1, \varphi_2,f_1,f_2$ are all smooth and that $f_1, f_2>0$. Assume moreover
\begin{equation}\label{equ: 3}
    2^{-\varepsilon} f_1 \leq f_2 \leq 2^{\varepsilon} f_1
\end{equation}
for some small constant $\varepsilon>0$. By elementary calculus it suffices to show that  $\left|\varphi_1-\varphi_2\right| \leq C \varepsilon$.

Set
$$
m_0:=\inf _X\left(\varphi_1-\varphi_2\right) \quad \text { and } \quad M_0:=\sup _X\left(\varphi_1-\varphi_2\right) .
$$
For notational convenience, we set $\mu=f_1^p\omega^n$. By \cref{subsolution 1}, there exists a constant $c(\beta,\omega,p)>0$ such that for any $0 \leq f \in L^p\left(X, \omega^n\right)$ with $\|f\|_p \leq 1$, we can find $u \in \operatorname{SH}_{m}(X, \beta,\omega)\cap L^\infty(X)$ such that $\left(\beta+d d^c u\right)^m \wedge \omega^{n-m} \geq 4 c f \omega^n$ with $-C \leq u \leq 0$ ($C$ may depends on $\|f\|_p$). 
Set
$$t_0  :=\sup \left\{t \in[m_0, M_0]: \mu\left(\varphi_1<\varphi_2+t\right) \leq c^p\right\}$$ and 
$$T_0 :=\inf \left\{t \in[m_0, M_0]: \mu\left(\varphi_1>\varphi_2+t\right) \leq c^p\right\}.$$

\textbf{Step 1.1.} We claim that $t_0 \leq m_0+C \varepsilon$ and $T_0 \geq M_0-C \varepsilon$.

Fix a number $t\leq t_0$. By definition we have $\mu(\varphi_1<\varphi_2+t)\leq c^p$. Set
$$
\hat{f}:=\frac{1}{2c}\mathds{1}_{\{\varphi_1<\varphi_2+t\}}f_1+\frac{1}{2\|f_1\|_p}\mathds{1}_{\{\varphi_1\geq\varphi_2+t\}}f_1\in L^p(X,\omega^n).
$$
Observe that
$$
\int_X\hat{f}^pdV_X\leq\frac{1}{2c^p}\mu(\varphi_1<\varphi_2+t)+2^{-p}\leq1,
$$
then by \cref{subsolution 1}, there is $u\in \operatorname{SH}_{m}(X,\beta,\omega)$ such that $\left(\beta+d d^c u\right)^m \wedge \omega^{n-m} \geq 4 c \hat{f} \omega^n$; using  \cref{main thm for m-positive beta} we can find $\phi\in \operatorname{SH}_{m}(X,\beta,\omega)\cap L^{\infty}(X)$ such that $\sup_X(\phi-\varphi_2)=0$ and
$$
(\beta+dd^c\phi)^m \wedge \omega^{n-m}=c(\hat{f})\hat{f}dV_X.
$$
From the domination principle \cref{m-positive have domination} we get that $0<4c\leq c(\hat{f})$; for an upper bound, observe that $\hat{f}\geq C^{-1}f_1$ and hence $(\beta+dd^c\phi)^m \wedge \omega^{n-m}\geq \frac{c(\hat{f})}{C}f_1dV_X=\frac{c(\hat{f})}{C}(\beta+dd^c\varphi_1)^m \wedge \omega^{n-m}$. By the domination principle \cref{m-positive have domination} again we get that $c(\hat{f})\leq C$. It follows from \cref{main a priori Hessian estimates} that $\|\phi\|_{\infty}\leq C$.

Now consider $\psi:=(1-\epsilon)\varphi_2+\epsilon\phi+t$. Then clearly $\{\varphi_1<\psi\}\subset\{\varphi_1<\varphi_2+t\}$ and on the latter set we have $\hat{f}=\frac{1}{2c}f_1$ and hence
$$
(\beta+dd^c\psi)^m \wedge \omega^{n-m}\geq((1-\epsilon)2^{-\frac{\epsilon}{m}}+2^{\frac{1}{m}}\epsilon)^mf_1dV_X>(1+\gamma)(\beta+dd^c\varphi_1)^m \wedge \omega^{n-m}
$$
for some constant $\gamma>0$, where we have used the mixed type inequality \cref{prop:mixed type for beta}. 

From the domination principle \cref{m-positive have domination} again we deduce that $\varphi_1\geq\psi\geq\varphi_2+t-C\epsilon$. It follows that $t\leq\varphi_1-\varphi_2+C\epsilon$ and hence $t\leq m_0+C\epsilon$. Since $t\leq t_0$ was selected arbitrarily, we obtain that $t_0\leq m_0+C\epsilon$. A similar argument yields that $T_0\geq M_0-C\epsilon$.

\textbf{Step 1.2.} Fixing $t_0<a<b<T_0$, we claim that
$$
I_{a, b}:=\int_{E_{a, b}} d\left(\varphi_1-\varphi_2\right) \wedge d^c\left(\varphi_1-\varphi_2\right) \wedge \beta_{\varphi_1}^{m-1} \wedge \omega^{n-m}\leq C(b-a)\left(\varepsilon+b-t_0\right)
$$
and
$$
I_{a, b} \int_{E_{a, b}} \beta_{\varphi_1} \wedge \omega^{n-1} \geq C^{-1}\left(\int_{E_{a, b}}\left|\partial\left(\varphi_1-\varphi_2\right)\right| \omega^n\right)^2,
$$
where $E_{a, b}:=\left\{\varphi_2+a<\varphi_1<\varphi_2+b\right\}$.

Set $\varphi:=\max(\varphi_1,\varphi_2+a),\psi:=\max(\varphi_1,\varphi_2+b)$ and consider
\begin{align*}
    J:=&\int_X(\varphi-\psi)(\beta_{\psi}^m \wedge \omega^{n-m}-\beta_{\varphi}^m \wedge \omega^{n-m})=
    \int_X(\varphi-\psi)dd^c(\psi-\varphi)\wedge T\\
    =&\int_Xd(\varphi-\psi)\wedge d^c(\varphi-\psi)\wedge T-\frac{1}{2}\int_X(\varphi-\psi)^2dd^cT\\
    \geq&\int_{E_{a,b}}d(\varphi-\psi)\wedge d^c(\varphi-\psi)\wedge(\beta+dd^c\varphi_1)^{m-1} \wedge \omega^{n-m}-\frac{1}{2}\int_X(\varphi-\psi)^2dd^cT,
\end{align*}

where $T:=\sum_{j=0}^{m-1}(\beta+dd^c\varphi)^j\wedge(\beta+dd^c\psi)^{m-1-j} \wedge \omega^{n-m}$ ,  and we have used that  $$ d(\varphi-\psi)\wedge d^c(\varphi-\psi) \wedge T\geq d(\varphi-\psi)\wedge d^c(\varphi-\psi) \wedge (\beta+dd^c\varphi)^{m-1} \wedge \omega^{n-m}$$ on $E_{a,b}$.
and $$
(\beta+dd^c\varphi_1)^{m-1} \wedge \omega^{n-m}=(\beta+dd^c\varphi)^{m-1} \wedge \omega^{n-m}
$$ on $\left\{\varphi_2+a<\varphi_1\right\}$.

Note that $|\varphi-\psi|\leq b-a$ in any case, by the computation of  $dd^c (\beta_{\varphi}^{j}\wedge \beta_{\psi}^{m-1-j}\wedge \omega^{n-m})$  in the proof of \cref{modified comparison principle} (note that \cite[Lemma 2.4]{KN16} plays a key role there) and a CLN type inequality (see \cite{KN25a}), we can write
\begin{equation} \label{eq estimate ddcT}
    \int_{X}\left|dd^c\left( (\beta+dd^c\varphi)^j\wedge(\beta+dd^c\psi)^{m-1-j} \wedge \omega^{n-m}   \right)\right|\leq C(B, \|\varphi_1\|_{\infty}, \|\varphi_2\|_{\infty})\end{equation}
provided the condition for $B$ at \cref{modified comparison principle}.
This implies that $-\frac{1}{2}\int_X(\varphi-\psi)^2dd^cT\geq C(b-a)^2$ and thus
$$
J\geq\int_{E_{a,b}}d(\varphi-\psi)\wedge d^c(\varphi-\psi)\wedge(\beta+dd^c\varphi)^{n-1}-C(b-a)^2.
$$
By the plurifine locality we have $(\varphi-\psi)(\beta_{\varphi}^m \wedge \omega^{n-m}-\beta_{\psi}^m \wedge \omega^{n-m})=0$ on $\{\varphi_1<\varphi_2+a\}\cup\{\varphi_1\geq\varphi_2+b\}$ and
$$
\left|\int_{E_{a,b}}(\varphi-\psi)(\beta_{\varphi}^m \wedge \omega^{n-m}-\beta_{\psi}^m \wedge \omega^{n-m})\right|\leq(b-a)\int_X|f_1-f_2|dV_X\leq C(b-a)\epsilon.
$$
On the contact set $\{\varphi_1=\varphi_2+a\}$, we can write
\begin{align*}
   &\left|\int_{\{\varphi_1=\varphi_2+a\}} (\varphi-\psi)(\beta_{\varphi}^m \wedge \omega^{n-m}-\beta_{\psi}^m \wedge \omega^{n-m})\right|\\
    =&(b-a)\left|\int_{\{\varphi_1=\varphi_2+a\}}(\beta_{\varphi}^m \wedge \omega^{n-m}-\beta_{\varphi_2}^m \wedge \omega^{n-m})\right|\\
    =&(b-a)\left|\int_{\{\varphi_1\leq\varphi_2+a\}}(\beta_{\varphi}^m \wedge \omega^{n-m}-\beta_{\varphi_2}^m \wedge \omega^{n-m})\right|\\
    \leq & C(b-a)\epsilon+(b-a)\left|\int_{\{\varphi_1\leq\varphi_2+a\}}(\beta_{\varphi}^m \wedge \omega^{n-m}-\beta_{\varphi_1}^m \wedge \omega^{n-m})\right|\\
      =&C(b-a)\epsilon+(b-a)\left|\int_X(\beta_{\varphi}^m \wedge \omega^{n-m}-\beta_{\varphi_1}^m \wedge \omega^{n-m})\right|\\
      \leq & C(b-a)\epsilon+(b-a)\cdot  C \cdot \sup_X(\varphi-\varphi_1)
\end{align*}
where in the sixth line, we have to do integration by parts and using the CLN inequality as above in \cref{eq estimate ddcT}(see also argued in \cite[Theorem 3.5]{GL23}).

We next turn to prove the second claim. We may use \cite[Lemma 2.5]{KN16} to show that there exists $\theta$ depending only on $n,m$ such that
\begin{align*}
   & \frac{d(\varphi-\psi)\wedge d^c(\varphi-\psi)\wedge(\beta+dd^c\varphi_1)^{m-1} \wedge \omega^{n-m}}{\omega^{n}}\frac{\beta_{\varphi_1}\wedge\omega^{n-1}}{\omega^{n}}\\
    \geq& \theta \cdot \frac{\beta_{\varphi_1}^m \wedge \omega^{n-m}}{\omega^n}\frac{d(\varphi-\psi)\wedge d^c(\varphi-\psi) \wedge\omega^{n-1}}{\omega^{n}}.
\end{align*}
The rest of the proof is unchanged as argued in \cite[Theorem 3.5]{GL23}, note that the condition $f_1\geq c_0>0$ is crucial here.
\textbf{Step 1} can now be finished by following the lines  of \cite[Theorem 3.5, Step 1.3]{GL23}. 

\textbf{Step 2.}
In this step we remove the assumption on $f_1,f_2$ made in \textbf{Step 1}. 

Performing the proof on page 10 of \cite[Theorem 3.5]{LPT20}, assumption (\ref{equ: 3}) can be removed. 

Next, we remove the smoothness assumption. Let $f_{i,j}\rightarrow f_i$ be positive smooth approximation, $i=1,2$. Let  $\varphi_{i,j}\downarrow \varphi_i$ be a smooth approximation, $i=1,2$.

Using \cref{smooth Hessian for beta}, we can solve the following equation for $i=1,2$ with $u_{i,j}\in \operatorname{SH}(X,\beta, m)\cap C^{\infty}(X)$:
$$(\beta+dd^c u_{i,j})^m\wedge \omega^{n-m} = e^{u_{i,j}-\varphi_{i,j}} f_{i,j} dV_{X}.
$$

By \textbf{Step 1}, we have 
$$|u_{1,j}-u_{2,j}|\leq T_j\|g_{1,j}-g_{2,j}\|_{p}^{\frac{1}{n}},$$
 where $g_{i,j}=e^{u_{i,j}-\varphi_{i,j}} f_{i,j}$, and from the proof of \textbf{Step 1}, we have $T_j$ is uniform with respect to $j$.
 As $j\rightarrow +\infty$, argue as in the proof of \cref{main thm twist for m-positive beta}, we can show that $u_{i,j}$ converging to some $u_i$ solving 
 $$(\beta+dd^c u_i)^m \wedge \omega^{n-m} = e^{u_i-\varphi_i} f_i dV_{X}.$$
 
Using the domination principle we obtain $u_i=\varphi_i$; hence $\|g_{1,j}-g_{2,j}\|_p\rightarrow \|f_1-f_2\|_p$, finishing the proof.
\end{proof}

\subsection{Stability of complex Hessian equations in nef classes}
Let $\{\beta\}\in BC^{1,1}(X)$ be a nef Bott-Chern class satisfying $SL_{\omega,a}(\{\beta\})>0$ for some constant $1\leq a<\frac{n}{n-m}$. Assume moreover that there is a function $\rho\in \operatorname{SH}_m(X,\beta,\omega)\cap L^\infty(X)$ and that $\{\beta\}$ is $(\omega,m)$-non-collapsing. 

 \begin{theorem}
     Fix $\lambda>0$, suppose that $\varphi_1, \varphi_2 \in \mathrm{SH}_m(X, \beta,\omega) \cap L^{\infty}(X)$, satisfy
$$
\beta_{\varphi_1}^m \wedge \omega^{n-m}=e^{\lambda \varphi_1}f_1 \omega^n, \quad \beta_{\varphi_2}^m \wedge \omega^{n-m}=e^{\lambda \varphi_2}f_2 \omega^n,
$$
where $f_1, f_2 \in L^p\left(\omega^n\right), p>\frac{n^2}{m^2-(1-\frac{1}{a})mn}$. Then
$$
\|\varphi_1-\varphi_2\|_{\infty} \leqslant C\|f_1-f_2\|_p
$$
where the constant $C$ depends on $c_0, p,\|f_1\|_p,\|f_2\|_p, \omega, X$, $B$, $n$, $m$, $a$, $\beta$.
 \end{theorem}
\begin{proof}
The proof is very similar to that in \cref{stability for m-positive hessian}.
\end{proof}

\section{The Monge-Amp\`ere equations in nef classes with mild singularities}\label{sigularities}

Let $(X,\omega)$ be a compact Hermitian manifold of complex dimension $n$ equipped with a Hermitian metric $\omega$.

In this section, we establish a $L^\infty$- estimate analogous to \cref{a priori in nef class} for the Monge-Amp\`ere equation and pursue more general resolution of Complex Monge-Amp\`ere equations in nef classes after \cite{LWZ24}, \cite{SW25}, \cite{GL23}. 

Let $\beta$ be a nef Bott-Chern class, write
$$
V_\beta:=\sup\{u\in\operatorname{PSH}(X,\beta),u\leq0\}.
$$
A function $u\in\operatorname{PSH}(X,\beta)$ is said to be of minimal singularity type if there is a constant $C$ such that $|u-V_\beta|\leq C$. Recall that the unbounded locus $L(V_\beta)$ of $V_\beta$ is defined as the set of all the points $x\in X$ such that $V_\beta$ is unbounded in each neighborhood of $x$, which is a closed subset of $X$. \textbf{Unless otherwise stated, we shall assume that $L(V_\beta)$ is a pluripolar set} in some cases for technical reasons (as was stated in the Introduction). Here, $P(V_\beta):=\{V_\beta=-\infty\}$ denotes the polar locus of $V_\beta$.

In \cite{Błocki06}, Błocki has successfully characterized the domain of definition of the complex Monge-Amp\`ere operators on an arbitrary domain $\Omega\subset\mathbb{C}^n$, i.e., we say a function $u\in\operatorname{PSH}(\Omega)$ belongs to $D(\Omega)$ if there exists a measure $\mu$ such that for every open subset $U\subset\Omega$ and every sequence $u_j\in\operatorname{PSH}(U)\cap C^\infty(U)$ satisfying $u_j\searrow u$, then we have $(dd^cu_j)^n$ converges weakly to $\mu$. The Monge-Amp\`ere measure $(dd^cu)^n$ of $u$ is then defined to be $\mu$.  As illustrated in \cite{Błocki06}, when $\Omega$ is a hyperconvex domain, $D(\Omega)$ coincides with the Cegrell class $\mathcal{E}(\Omega)$ introduced in \cite{Ceg04}.
\begin{definition}\label{DMA}
    We say a $\beta$-plurisubharmonic function $u$ belongs to the class $D(X,\beta)$ if for any local coordinate ball $\Omega\subset X$, there is a smooth function $\rho\in C^\infty(\Omega)$ (we usually choose $dd^c\rho\geq\beta$) such that $u+\rho\in D(\Omega)$.
\end{definition}
The following property established in \cite[Theorem 1.2]{Błocki06} will be crucial:
\begin{theorem}\label{Blocki thm 1.2}
    Let $u\in D(X,\beta)$ and $v\in\operatorname{PSH}(X,\beta)$ be such that $u-v$ is bounded. Then we have $v\in D(X,\beta)$.
\end{theorem}
\begin{proof}
    Select an arbitrary coordinate ball $B$ and a smooth potential $\rho$ such that $dd^c\rho\geq\beta$ on $B$, then we have $\rho+u\in\mathcal{E}(B)$ and hence $\rho+v\in\mathcal{E}(B)$ by \cite[Theorem 1.2]{Błocki06}.
\end{proof}
In the sequel of this section, we will always take the assumption $D(X,\beta)\neq\emptyset$, that is to say, $V_\beta\in D(X,\beta)$ by \cref{Blocki thm 1.2}. Again by \cref{Blocki thm 1.2}, any $\beta$- psh function with minimal singularity type belongs to $D(X,\beta)$.
\begin{remark}
   When the class $\{\beta\}$ admits some potentials with mild singularities, we indeed have $D(X,\beta)\neq\emptyset$, we now give two typical examples. 
   
   Thanks to \cite[Theorem 2.5, Proposition 2.9]{Dem93-1}, \cite[Corollary 3.6]{FS95}, \cite[Theorem 3.1.2, Theorem 3.1.3]{Ra01}, when $u\in\operatorname{PSH}(X,\beta)$ and $H_2(L(u))=0$, we have $u\in D(X,\beta)$. Here $H_2$ denotes the 2- dimensional Hausdorff measure.

   When $X$ is a compact surface, \cite{Błocki04} characterized exactly the class $D(X,\beta)$. Namely, $u\in D(X,\beta)$ if and only if $u\in W^{1,2}(X)$, where $W^{1,2}(X)$ denotes the sobolev space on $X$. This example illustrates that $D(X,\beta)$ is much more large than the class of bounded potentials considered in \cite{LWZ24}, \cite{SW25}, \cite{GL23}.
\end{remark}
As in \cite{Sal25}, the measure $(\beta+dd^cu)^n$ for $u\in D(X,\beta)$ can be defined as follows: fix a coordinate ball $B\subset X$ and select a potential $\rho\in C^\infty(B)$ such that $dd^c\rho\geq\beta$ in $B$, then
$$
(\beta+dd^cu)^n:=\sum_{j=0}^nC_n^j(dd^c(u+\rho))^j\wedge(\beta-dd^c\rho)^{n-j}.
$$
Since $u+\rho\in\mathcal{E}(B)$, we get a well-defined Radon measure on $B$ thanks to \cite[Theorem 4.2]{Ceg04}. Thanks to \cite{Ceg12}, we also have the decreasing (and increasing almost everywhere) convergence property for the measure $(\beta+dd^cu)^n$.

Taking $m=n$ in \cref{a priori in nef class}, we derive the following $L^\infty$ estimates analogous to \cref{main a priori Hessian estimates} for the Monge-Amp\`ere equations, which is a generalization of \cite[Theorem 3.1]{SW25} and \cite[Theorem 1]{GPTW24}.
\begin{theorem}\label{a priori in nef class for MA}
    Let $\{\beta\}\in BC^{1,1}(X)$ be a nef Bott-Chern class satisfying $SL_{\omega,a}(\{\beta\})>0$ for some constant $a\geq1$. Assume also $\varphi_t\in \operatorname{PSH}(X,\beta+t\omega)\cap C^\infty(X)$ satisfying
\begin{equation}
(\beta+t\omega+dd^c\varphi_t)^n= c_te^{F_t}\omega^n,\quad\sup_X\varphi_t=0.
\end{equation}
Here $F_t\in C^\infty(X)$. We also fix a constant $p>na$. Then there exists a uniform constant $C$ depending on $\beta,\omega,n,p$, the upper bound of $\int_{X}e^{aF_t(z)}[\log(1+e^{aF_t(z)})]^p\omega^n$, the lower bound of $\int_Xe^{\frac{F_t}{n}}\omega^n$ and the lower bound of  $SL_{\omega,a}(\{\beta\})$ such that
$$
0\leq-\varphi_t+V_t\leq C.
$$
Here $V_t:=\sup\{u\;|\;u\in \operatorname{PSH}(X,\beta+t\omega),u\leq0\}$ is the largest non-positive $(\beta+t\omega)$-psh function.
\end{theorem}

\subsection{Solutions of Complex Monge-Amp\`ere equations with an exponential twist}
In \cite{LWZ24} and \cite{SW25}, the authors studied complex Monge-Amp\`ere equations in a nef Bott-Chern class admitting a bounded potential $\rho\in\operatorname{PSH}(X,\beta)\cap L^\infty(X)$ and obtained bounded solutions. It is natural to relax the boundedness assumption of the potential $\rho$ and consider more singular potentials. We first provide the following resolution of CMA equations with an exponential twist:
\begin{theorem}\label{small unbounded locus twist equation}
    Let $\{\beta\}\in BC^{1,1}(X)$ be a nef Bott-Chern class satisfying $SL_{\omega,a}(\{\beta\})>0$ and $ D(X,\beta)\neq\emptyset$. Then, for any $f\in L^p(X), p>a$ satisfying $\int_Xf\omega^n>0$, there exists a function $\varphi$ of minimal singularity type solving the following equation:
    $$
    (\beta+dd^c \varphi)^n = e^ {\lambda \varphi} f\omega^n.
    $$
    Moreover, $\varphi$ is unique if we assume furthermore that $SL_{\omega,1}(\{\beta\})=\underline{\operatorname{Vol}}(\{\beta\})>0$ and $L(V_\beta)$ is pluripolar.
\end{theorem}
\begin{proof}
The uniqueness follows from the domination principle, see \cref{domination of MA} below.
Without of loss of generality, we will asssume that $\lambda=1$.
  Choose a sequence of smooth positive functions $f_k$ converging in $L^p(X)$ to $f$. It is then easy to check that $\int_Xf_k^a[\log(1+f_k^a)]^p\omega^n$ is uniformly bounded. The main theorem of \cite{TW10} yields smooth solutions of the following system of Monge-Amp\`ere equations:

    \begin{equation}\label{approximation equations}
            (\beta+\frac{1}{j}\omega +dd^c u_{j,k})^n= e^{u_{j,k}}f_k\omega^n
    \end{equation}    
    for any $j,k\geq1$. We claim that there is a uniform constant $C$ independent of $j,k$ such that 
    $$
-C\leq-u_{j,k}+V_{\beta+\frac{1}{j}\omega}\leq C.
    $$
    Indeed, set $v_{j,k}:=u_{j,k}-\sup_Xu_{j,k}$, rewrite \eqref{approximation equations} as
    \begin{equation}\label{approximation equations 1}
            (\beta+\frac{1}{j}\omega +dd^c v_{j,k})^n= c_{j,k}e^{v_{j,k}}f_k \omega^n,
    \end{equation}    
    where $c_{j,k}=e^{\sup_Xu_{j,k}}$. Now since $e^{v_{j,k}}f_k\leq f_k$, we have that $\int_X(e^{v_{j,k}}f_k)^a[\log(1+(e^{v_{j,k}}f_k)^a)]^q\omega^n$ is uniformly bounded above for each $q>na$. Moreover, an argument using Jenson's inequality similar to that in \cref{main thm twist for m-positive beta} yields a lower bound for $\int_X(e^{v_{j,k}}f_k)^{\frac{1}{n}}\omega^n$, the $L^\infty$ estimate \cref{a priori in nef class for MA} implies that
    $$
-C\leq-v_{j,k}+V_{\beta+\frac{1}{j}\omega}\leq C.
    $$
    To finish the proof of the claim, it suffices to show that $\sup_Xu_{j,k}$ is uniformly bounded, this in turn amounts to show that $c_{j,k}=e^{\sup_Xu_{j,k}}$ is uniformly bounded away from zero. As in the proof of \cref{main thm twist for m-positive beta}, the mixed type inequality gives an upper bound of $c_{j,k}=e^{\sup_X{u_{j,k}}}$. For a lower bound, we use our assumption $SL_{\omega,a}(\{\beta\})>0$. Write $f_k=e^{F_k}$, by the definition of sup slopes, we have
    $$
c_{j,k}=SL_{\omega,v_{j,k}+F_k}(\{\beta+\frac{1}{j}\omega\})\geq\delta SL_{\omega,a}(\{\beta+\frac{1}{j}\omega\})\geq \delta SL_{\omega,a}(\{\beta\}),
    $$
    where the second inequality is due to the monotonicity of sup slopes \cref{monotone of slope}. For the first inequality, the $L^a$-norm of $e^{v_{j,k+F_k}}=e^{v_{j,k}}f_k$ is bounded above since $v_{j,k}\leq0$, $f_k\rightarrow f$ in $L^p$ and $p>a$. It follows from the definition of sup slopes in Hermitian classes that there exists a uniform $\delta>0$ such that $SL_{\omega,v_{j,k}+F_k}(\{\beta+\frac{1}{j}\omega\})\geq\delta SL_{\omega,a}(\{\beta+\frac{1}{j}\omega\})>0$. Therefore, the claim follows.
    
   Next, we fix $j$ and letting $k\rightarrow\infty$. The same argument as in \cref{main thm twist for m-positive beta} gives that $u_{j,k}\rightarrow u_j$ uniformly for some $u_j\in\operatorname{PSH}(X,\beta+\frac{1}{j}\omega)\cap C^0(X)$. \eqref{approximation equations} yields that
    \begin{equation}\label{approximation equations 2}
            (\beta+\frac{1}{j}\omega +dd^c u_{j})^n= e^{u_{j}}f\omega^n,
    \end{equation}    
    and we still have the estimate
    $$
-C\leq-u_{j}+V_{\beta+\frac{1}{j}\omega}\leq C.
    $$
    By the domination principle \cref{domination for beta 2}, $\{ u_j\}_j$ is a decreasing sequence. Moreover, all the $u_j$ belongs to $D(X,\beta)$ because $u_j\geq V_\beta-C$, hence we can take the limit at both sides of \cref{approximation equations} and finally arrive at
$$
(\beta+dd^c u)^n=e^{u} f\omega^n.
$$

\end{proof}
\begin{remark}
   \cref{small unbounded locus twist equation} provides a unified generalization of many existing results for degenerate complex Monge-Amp\`ere equations, including \cite[Theorem 3.4]{GL23}, \cite[Theorem 1.2]{LWZ24} and \cite[Theorem 1.1]{SW25}.
\end{remark}

\subsection{Solutions of non-twist Complex Monge-Amp\`ere equations}

To deal with non-twist Complex Monge-Amp\`ere equations, we need to establish the following technical lemmas for nef class with mild singularities, such as maximum principle, properties of quasi-psh envelopes, mass comparison theorem, weak convergence theorem and  domination principle.

\subsubsection{{Maximum principle}}
We recall the following maximum principle established in \cite{KH09}:
\begin{lemma}\label{max_local} \cite[Theorem 4.1]{KH09}
   Let $\Omega$ be a hyperconvex domain and let $u, u_1, \ldots, u_{n-k} \in \mathcal{E}(\Omega), v \in P S H^{-}(\Omega)$ for $k=1, \cdots, n$. Set $T:=d d^c u_1 \wedge \ldots \wedge d d^c u_{n-k}$. Then
$$
\left(d d^c \max (u, v)\right)^{k} \wedge T|_{\{u>v\}}=\left(d d^c u \right)^k\wedge T|_{\{u>v\}}.
$$
\end{lemma}
\begin{lemma}\label{max2_local}\cite[Proposition 4.3]{KH09}
    a) Let $u, v \in \mathcal{E}(\Omega)$ be such that $\left(d d^c u\right)^n(\{u=v=-\infty\})=0$. Then
$$
\left(d d^c \max (u, v)\right)^n \geq \mathds{1}_{\{u \geq v\}}\left(d d^c u\right)^n+\mathds{1}_{\{u<v\}}\left(d d^c v\right)^n,
$$
where $\mathds{1}_E$ denotes the characteristic function of $E$.

b) Let $\mu$ be a positive measure vanishing on all pluripolar subsets of $\Omega$. Suppose $u, v \in \mathcal{E}(\Omega)$ are such that $\left(d d^c u\right)^n \geq \mu,\left(d d^c v\right)^n \geq \mu$. Then $\left(d d^c \max (u, v)\right)^n \geq \mu$.
\end{lemma}
\begin{corollary}\label{max_global}
    Let $\{\beta\}\in BC^{1,1}(X)$ be a nef Bott-Chern class such that  $ D(X,\beta)\neq\emptyset$. Then for $u,v\in PSH(X,\beta)$ with minimal singularities, we have 
    $$  
\left(\beta+d d^c \max (u, v)\right)^{n}|_{\{u>v\}}=\left(\beta+d d^c u \right)^n|_{\{u>v\}}.
    $$
\end{corollary}
\begin{proof}
 The problem is local, so we only need to check on each small coordinate ball $B\subset X$, where $D(\Omega)=\mathcal{E}(\Omega)$ and hence \cref{max_local} can be applied. Now, we just need to proceed the proof in \cite[Theorem 5.3]{Sal25}. 
\end{proof}

Following the lines of the proofs in \cite[Theorem 3.29]{GZ17} and \cite[Corollary 3.30]{GZ17}, we have the following two comparison principles:
\begin{lemma}
    Assume $u, v \in \mathcal{E}(\Omega)$ are such that $u(z)\geq v(z)$ near $\partial\Omega$. Moreover, assume that $\int_{\Omega} (dd^c u)^n, \int_{\Omega}(dd^c v)^n < +\infty$.
    Then
    \begin{align*}
        \int_{\{u<v\}}(d d^c v)^n \leq \int_{\{u<v\}}\left(d d^{\mathrm{c}} u\right)^n.
    \end{align*}
\end{lemma}

\begin{corollary}\label{comparison_unbounded2}
    Assume $u, v \in \mathcal{E}(\Omega)$ are such that $u(z)\geq v(z)$ near $\partial\Omega$. Moreover, assume that $\int_{\Omega} (dd^c u)^n, \int_{\Omega}(dd^c v)^n < +\infty$.
    If  $$ (dd^cu)^n\leq (dd^c v)^n, $$
    then $v\leq u$ in $\Omega$.
\end{corollary}

\subsubsection{{Mass comparison theorem}}

We establish a technical lemma analogous to Demailly's second comparison principle \cite{Dem93-1}, which will play a key role in proving $(\beta+dd^cu)^n$ is non-pluripolar for $u\in\operatorname{PSH}(X,\beta)$ with minimal singularity.
\begin{theorem}\label{thm: Demailly Second comparison}
Let $u,v\in\mathcal{E}(U)$ be plurisubharmonic functions on an open subset $U$ of  $\mathbb{C}^n$. Assume that $(dd^c v)^n$ is a non-pluripolar measure and that $u-v$ is bounded. Then $(dd^cv)^n$ is non-pluripolar.
\end{theorem}
\begin{proof}
    By our assumption for each fixed $c\in \mathbb{R}$, we have $u-c>2v$ when $z\rightarrow P(v)=\{v=-\infty\}$. So  for each $c$,  there exist $r_{c}$ such that $w_c:=\max\{u-c,2v \}$ coincides with $u-c$ on $\{v<r_c\}$, which is an open neighborhood of $P(v)$. Hence we can write 
    $$
    \mathds{1}_\{{v<r_c}\}(dd^c u)^n = \mathds{1}_\{{v<r_c}\}(dd^c w_c)^n,
    $$
    and furthermore 
    \begin{equation}\label{eq compare}
    \mathds{1}_{P(v)}(dd^c u)^n\leq (dd^cw_c)^n.
    \end{equation}
    
We have from \cite{Ceg04} the weak convergence
$$
(dd^c w_c)^n\rightarrow (dd^c 2v)^n
$$
as $c\rightarrow +\infty$.
Now taking the limit on both sides of \eqref{eq compare} we obtain that
$$
 \mathds{1}_{P(v)}(dd^c u)^n\leq(2dd^cv)^n=2^n(dd^cv)^n,
$$
thus $(dd^cv)^n$ puts no mass on $P(v)$. Cegrell's decomposition theorem \cite[Theorem 5.11]{Ceg04} then yields that $(dd^cu)^n$ can be decomposed into a sum of a non-pluripolar measure and a measure supported on $P(v)$, the proof is therefore concluded since $(dd^c u)^n$ puts no mass on $P(v)$.
\end{proof}

The arguments in \cref{thm: Demailly Second comparison} can be generalized to the globally setting:
\begin{theorem}\label{thm: Demailly Second comparison_global}
    Let $\beta$ be a smooth $(1,1)$- form on $X$ satisfying $D(X,\beta)\neq\emptyset$ and $u,v\in\operatorname{PSH}(X,\beta)$ be $\beta$- psh functions with minimal singularities. Assume moreover that $(\beta+dd^cv)^n$ is non-pluripolar and that $u-v$ is bounded. Then we have that $(\beta+dd^cu)^n$ is non-pluripolar.
\end{theorem}
\begin{proof}
    \textbf{Step 1.} We first show that $(\omega+\beta+dd^cv)^n$ is non-pluripolar for an arbitrary Hermitian metric $\omega$ on $X$. The problem is local, so we select a coordinate ball $B$ and a smooth plurisubharmonic function $\rho$ on $B$ such that $dd^c\rho\geq\omega-\beta$. For each $c>0$, it is clear that $\frac{v+\rho}{2}-c>v$ as $z\rightarrow P(v)$ and hence $w_c:=\max(\frac{v+\rho}{2}-c,v)\equiv\frac{v+\rho}{2}-c$ on the set $\{v<r_c\}$ for some constant $r_c$ depending on $c$. Note that $\beta+dd^c\frac{u+\rho}{2}=\frac{1}{2}\beta_u+\frac{1}{2}\beta_\rho\geq\frac{1}{2}\omega>0$ by the choice of $\rho$, so we still have $w_c\in\operatorname{PSH}(X,\beta)$ and $w_c\searrow v$ as $c\rightarrow+\infty$. Now we can write
    $$
    \mathds{1}_{P(v)}\left(\beta+dd^c\frac{v+\rho}{2}\right)^n\leq (\beta+dd^cw_c)^n.
    $$
    
    On the other hand, observe that $\beta+dd^c\frac{v+\rho}{2}\geq\frac{1}{2}(\omega+\beta+dd^cv)$, so we can further write
    $$
    \mathds{1}_{P(v)}(\omega+\beta+dd^cv)\leq2^n(\beta+dd^cw_c)^n.
    $$
    Letting $c\to+\infty$ on both sides we get
    \begin{equation}\label{eq:compare 2}
     \mathds{1}_{P(v)}(\omega+\beta+dd^cv)\leq2^n(\beta+dd^cv)^n.
    \end{equation}
    Now we can use Cegrell's decomposition as in \cref{thm: Demailly Second comparison} to conclude the proof of this step. Indeed, choose $\eta\in C^\infty(B)$ such that $dd^c\eta\geq\omega+\beta$ on $B$. Then we have
    $$
    (\omega+\beta+dd^cv)^n\leq(dd^c(\eta+v))^n=\mu_1+\mu_2,
    $$
    on $B$. Where $\mu_1$ is a non-pluripolar measure on $B$ and $\mu_2$ is supported on $P(v)$. It follows that $(\omega+\beta+dd^cv)^n$ is non-pluripolar since it puts no mass on $P(v)$ by \eqref{eq:compare 2}.

    \textbf{Step 2.} For the general case, select a Hermitian metric $\omega$ on $X$ such that $\omega\geq\beta$. It is enough to show that $(\omega+\beta+dd^cu)^n$ is non-pluripolar. As in the first step, fix an arbitrary $c>0$, there is a corresponding constant $q_c$ such that $u-c>2v$ on the set $\{v<q_c\}$. Since we have assumed that $\omega\geq\beta$, $2v\in\operatorname{PSH}(X,\omega+\beta)$ and hence $\varphi_c:=\max(u-c,2v)\in\operatorname{PSH}(X,\omega+\beta)$, which satisfies $\varphi_c\equiv u-c$ on $\{v<q_c\}$. Now we can write
    \begin{align*}
        \mathds{1}_{P(u)}(\omega+\beta+dd^cu)^n\leq\lim_{c\to+\infty}(\omega+\beta+dd^c\varphi_c)^n=  (\omega+\beta+dd^c2v)^n.
    \end{align*}
  From the first step we conclude that $(\omega+\beta+dd^cu)^n$ puts no mass on $P(u)$ and hence is non-pluripolar by a simple argument using Cegrell's decomposition theorem again.
\end{proof}

\subsubsection{Monge-Amp\`ere envelopes}
\begin{lemma} \label{lem: contact for beta lsc MA}
     Let $\{\beta\}\in BC^{1,1}(X)$ be a nef Bott-Chern class such that  $ D(X,\beta)\neq\emptyset$.  Let $h$ be a bounded  lower semi-continuous Lebesgue measurable function, then the Monge-Amp\`ere  measure $\left(\beta+d d^c P_{\beta}(h)\right)^n$ puts no mass on the open set $\left\{P_{\beta}(h)<h\right\}$. 
\end{lemma}

\begin{proof}
    The idea of proof is the same as \cref{lem: contact for beta lsc}, the difference is that the envelopes may not be bounded functions, the application of  Bedford-Taylor theory is limited, so we need to apply pluripotential theory for unbounded psh functions. 
    
Since $\{\beta\}$ is nef, we may find for each $j$ a function $\rho_j\in \operatorname{PSH}(X,\beta+\frac{1}{j}\omega)\cap C^\infty(X)$ such that $\beta+\frac{1}{j}\omega+dd^c\rho_j$ is a Hermitian metric.
 Observe that 
 $$
 P_{\beta+\frac{1}{j}\omega}(h)= P_{\beta+\frac{1}{j}\omega+dd^c \rho_j}(h-\rho_j)+\rho_j
 $$
 and 
 $$
 P_{\beta+\frac{1}{j}\omega}(h)\searrow P_{\beta}(h).
 $$
Since $h$ is bounded, $P_\beta(h)$ shares the same singularity type as $V_\beta$, \cref{Blocki thm 1.2} yields that $\left(\beta+d d^c P_{\beta}(h)\right)^n$  is well defined and  
$$
\left(\beta+ \frac{1}{j} \omega+ d d^c P_{\beta+\frac{1}{j}\omega}(h)\right)^{n} \rightarrow (\beta+dd^c P_{\beta}(h))^{n}
$$
in the sense of currents.
Indeed, locally choose a smooth potential $v$ such that $\beta+\omega\leq dd^c v$, then by \cite{Ceg04},
$\left[dd^c \left(v+P_{\beta}(h)\right)\right]^l $($l=1,\cdots,n$) is well defined, and moreover,  for $l=1,\cdots,n$,
$$
\left[dd^c \left(v+P_{\beta+\frac{1}{j}\omega} (h)\right)\right]^l \rightarrow\left[dd^c \left(v+P_{\beta}(h)\right)\right]^l 
$$
in the sense of currents as $j\rightarrow+\infty$.
 By \cref{lem: contact for herrmitian} we have 
 \begin{align*}
       &\mathds{1}_{\{P_{\beta+ \frac{1}{j}\omega}(h)<h\}} \left(\beta+ \frac{1}{j} \omega+ d d^c P_{\beta+\frac{1}{j}\omega}(h)\right)^n  \\
       =&\mathds{1}_{\{P_{\beta+\frac{1}{j}\omega+dd^c \rho_j}(h-\rho_j)<h-\rho_j\}} \left(\beta+ \frac{1}{j} \omega+dd^c \rho_j+   d d^c P_{\beta+\frac{1}{j}\omega}(h-\rho_j)\right)^n\\
    =&0.
 \end{align*}
 Since $\{P_{\beta+\frac{1}{k} \omega}(h)< h  \}$ is open, by upper semi-continuity of the weak convergence of Radon measures, we have for $\forall k\in \mathbb{N}^*$,
\begin{align*}
   & \int_{ \{P_{\beta+\frac{1}{k} \omega}(h)< h  \}}  (\beta+dd^c P_{\beta}(h))^{n} \\
 &\leq \liminf_{j\rightarrow +\infty}\int_{\{P_{\beta+\frac{1}{k} \omega}(h)< h  \}} \left(\beta+ \frac{1}{j} \omega+ d d^c P_{\beta+\frac{1}{j}\omega}(h)\right)^n \\
 =& 0.
\end{align*}
Letting $k\rightarrow +\infty$, we get the desired result.
\end{proof}

The following corollary says that any $\beta$-plurisubharmonic function with minimal singularity has non-pluripolar Monge-Amp\`ere measure, which is interesting on its own:
\begin{corollary}\label{non-pluripolar} 
 Let $\{\beta\}$ be a nef Bott-Chern class on $X$ satisfying $D(X,\beta)\neq\emptyset$. Then, for each $u\in\operatorname{PSH}(X,\beta)$ with minimal singularity, the Monge-Amp\`ere measure $(\beta+dd^cu)^n$ does not charge pluripolar sets. In particular, $u$ has zero Lelong number everywhere. 
\end{corollary}
\begin{proof}
We first show that $(\beta+dd^cV_\beta)^n$ is non-pluripolar. Take $h=0$ in \cref{lem: contact for beta lsc MA}, we see that $(\beta+dd^cV_\beta)^n$ puts no mass on $\{V_\beta<0\}$, whence on $\{V_\beta=-\infty\}$. Similar to the proof in \cref{thm: Demailly Second comparison_global}, we can apply Cegrell's decomposition theorem to conclude that $(\beta+dd^cV_\beta)^n$ is non-pluripolar.  Since $u$ has the same singularity type as $V_\beta$, it follows immediately from \cref{thm: Demailly Second comparison_global} that  $(\beta+dd^cu)^n$ is non-pluripolar. 

To prove that $u$ has vanishing Lelong number, choose an arbitrary $x\in X$, a coordinate ball $\Omega$ centered at $x$ and  a smooth local potential $\rho$ such that $\rho+u$ is plurisubharmonic on $\Omega$. We have showed above that $(\beta+dd^cu)^n$ is non-pluripolar and so does $(\omega+\beta+dd^cu)^n$ for arbitrary Hermitian metric $\omega$ by the proof of the first step in \cref{thm: Demailly Second comparison_global}. This in turn implies that $(dd^c \rho +dd^c u)^n(\{x\})=0$. Consequently, the comparison theorem between Lelong numbers (\cite[Corollary 5.7]{Ceg04}) yields that $2\pi \cdot \nu(\rho+u,x)\leq \{(dd^c \rho +dd^c u)^n(\{x\})\}^{\frac{1}{n}}=0$ and we have concluded our proof.
\end{proof}

According to \cref{non-pluripolar}, for $u,v\in \operatorname{PSH}(X,\beta)$ with minimal singularity types, we have that $(\beta+dd^cu)^n(u=-\infty)=0$. The arguments in \cite[Proposition 4.3]{KH09} can now be carried over, giving the following properties:
\begin{corollary}\label{max_global2}
Let $\{\beta\}\in BC^{1,1}(X)$ be a nef Bott-Chern class such that $V_{\beta}\in D(X,\beta)$. Then for $u,v\in \operatorname{PSH}(X,\beta)$ with minimal singularity types, we have 
$$
\left(\beta+d d^c \max (u, v)\right)^n \geq 1_{\{u \geq v\}}\left(\beta+d d^c u\right)^n+1_{\{u<v\}}\left(\beta+d d^c v\right)^n.
$$
\end{corollary}

Here and after, by the notion "converge in capacity", we always mean converge with respect to $\operatorname{Cap}_\omega$, which locally is equivalent to the Bedford-Taylor capacity.
The following weak convergence lemmas will be crucial in the resolution of Complex Monge-Amp\`ere equations:

\begin{lemma} \label{continuity}
      Let $\{\beta\}\in BC^{1,1}(X)$ be a nef Bott-Chern class such that  $ D(X,\beta)\neq\emptyset$. Let  $u, u_k \in \operatorname{PSH}\left(X, \beta\right)$ be such that $u_k\rightarrow u $  in capacity as $k \rightarrow \infty$ . Moreover, assume that there exists a uniform constant $C$ such that 
    $|u-V_{\beta}|\leq C, |u_k-V_{\beta}|\leq C$. Assume also that $L(V_\beta)$ is a pluripolar set. Let $\chi_k, \chi \geq 0$ be quasi-continuous and uniformly bounded functions such that $\chi_k \rightarrow \chi$ in capacity. Then
$$
\liminf _{k \rightarrow \infty} \int_X \chi_k (\beta+dd^c u_k)^n \geq \int_X \chi (\beta+dd^c u)^n.
$$
If additionally,
$$
\int_X (\beta+dd^cu)^n \geq \limsup _{k \rightarrow \infty} \int_X (\beta+dd^c u_k)^n,
$$
then $\chi_k (\beta+dd^c u_k)^n \rightarrow \chi (\beta+dd^c u)^n$ in the weak sense of measures on $X$. 

Note that when $\beta$ is a closed form, the second condition is always achieved because $u_k,u$ have the same singularity type as $V_{\beta}$, whose Lelong number vanishes everywhere by \cref{non-pluripolar}.
\end{lemma}
\begin{proof}
we may assume that $u_k,u>V_{\beta}$ after shifting a constant.
Fix $\varepsilon>0$ and consider
$$
f^{k, \varepsilon}:=\frac{u_k-V_{\beta}}{u_k-V_{\beta}+\varepsilon},  k \in \mathbb{N}^*.
$$

Since we assume that $L(V_{\beta})$ is pluripolar, \cref{thm: Demailly Second comparison_global} yields that $(\beta+dd^c {u})^n$ puts no mass on $L(V_{\beta})$.

Now we will work on a relatively compact subset $U$ of $X-L(V_{\beta})$, where $V_\beta$ is bounded.
Observe that the functions $u_k \geq V_{\beta}$ are uniformly bounded in $U$  and converges in capacity to $u$ as $k \rightarrow \infty$.   
For each $ \varepsilon$ fixed the functions $f^{k, \varepsilon}$ are quasi-continuous, uniformly bounded (with values in $[0,1]$ ) and converge in capacity to $f^{ \varepsilon}$, where $f^{ \varepsilon}$ is defined by
$$
f^{\varepsilon}:=\frac{u-V_{\beta}}{u-V_{\beta}+\varepsilon}.
$$
With the information above we can apply  \cref{thm: weak convergence} to deduce that
$$
f^{k,  \varepsilon} \chi_k \beta_{{u}_{k}}^n\rightarrow f^{ \varepsilon} \chi \beta_{{u}}^n \text { as } k \rightarrow \infty
$$
in the weak sense of measures on $U$. In particular since $0 \leq f^{k, \varepsilon} \leq 1$, we can use the lower semi-continuity of weak convergence on open sets to write
$$
\begin{aligned}
\liminf _{k \rightarrow \infty} \int_X \chi_k \beta_{{u}_k}^n & \geq \liminf _{k \rightarrow \infty} \int_U f^{k, \varepsilon} \chi_k \beta_{{u}_k}^n \\
& \geq \int_U f^{\varepsilon} \chi \beta_{{u}}^n.
\end{aligned}
$$
Letting $\varepsilon\rightarrow0^+$, we derive
$$
\liminf _{k \rightarrow \infty} \int_X \chi_k \beta_{{u}_k}^n  \geq \int_U \chi\beta_{{u}}^n.
$$
Now, letting $U \rightarrow X-L(V_{\beta})$ and taking into account the fact that $(\beta+dd^c {u})^n$ puts no mass on $L(V_{\beta})$,  we arrive at the result.

Now we turn to prove the second statement.
Take $\chi_k=\chi=1$ in the first statement, we get:
\begin{equation}\label{equ_lim}
    \lim_{k\rightarrow+\infty}\int_{X}(\beta+dd^c u_k)^n=\int_{X}(\beta+dd^cu)^n.
\end{equation}
Now, choose a test function $g\in C^0(X)$ and let $B \in \mathbb{R}$ be such that $g\chi, g\chi_k \leq B$. By the first statement of this theorem, we get that
$$
\liminf _{k \rightarrow \infty} \int_X\left(B-g\chi_k\right) (\beta+dd^c u_k)^n\geq \int_X(B-g\chi) (\beta+dd^c u)^n.
$$
Flipping the signs and using \eqref{equ_lim} , we obtain the following inequality:
$$
\limsup _{k \rightarrow \infty} \int_X g\chi_k (\beta+dd^c u_k)^n \leq \int_X g\chi (\beta+dd^c u)^n.
$$
For the reverse direction, when $g\geq0$, we may use the first statement to get
$$
\liminf _{k \rightarrow \infty} \int_X g\chi_k (\beta+dd^c u_k)^n \geq \int_X g\chi (\beta+dd^c u)^n,
$$
and hence
$$
\lim_{k \rightarrow \infty} \int_X g\chi_k (\beta+dd^c u_k)^n = \int_X g\chi (\beta+dd^c u)^n.
$$
When $g\in C^0(X)$ is arbitrary, we may write $g=g^+-g^-$ to conclude the proof.
\end{proof}

We next remove the lower semi-continuous condition of $h$ in \cref{lem: contact for beta lsc MA}. We need the following continuity property, which will be useful when taking limit.
\begin{theorem}\label{lem: contact for beta bounded MA}
     Let $\{\beta\}\in BC^{1,1}(X)$ be a nef Bott-Chern class such that  $ D(X,\beta)\neq\emptyset$ and that $L(V_\beta)$ is a pluripolar set.  If $h$ is a bounded Lebesgue measurable quasi-continuous function, then the Monge-Amp\`ere  measure $\left(\beta+d d^c P_{\beta}(h)\right)^n$ puts no mass on the quasi-open set $\left\{P_{\beta}(h)<h\right\}$. 
\end{theorem}
\begin{proof}
    We can just follow the lines of \cref{contact}. In the last step, we need to show that 
    \begin{align*}
            & \int_{X} \frac{\max ( -\hat{h}+h,0)}{ \max (- \hat{h}+h,0) +\varepsilon   } (\beta+dd^c \hat{h})^n\\
            \leq & \liminf_{k\rightarrow +\infty} 
     \int_{X} \frac{\max ( -\hat{h}_k+h_k,0)}{ \max (- \hat{h_k}+h_k,0) +\varepsilon   } (\beta+dd^c \hat{h}_k)^n,
    \end{align*}
    which follows from \cref{continuity}.
\end{proof}
\begin{remark}\label{bounded above reduction and decreasing}
1).  If $h$ is a quasi-continuous Lebesgue measurable function which is bounded from below and $P_\beta(h) \in \operatorname{PSH}(X, \beta)$. Then \cref{lem: contact for beta bounded MA} still holds. Actually, one can replace $h$ by $\min (h, C)$ for some constant $C>\sup _X P_\beta(h)$. Then we have
$$
P_\beta(h)=P_\beta(\min (h, C)),
$$
and the claim follows from \cref{lem: contact for beta bounded MA}. 

2). If  $h_j\downarrow h$ and $h_1$ is bounded from above, we have $P_{\beta}(h_j)\downarrow P_{\beta}(h)$.

We give a short proof. When $h$ is quasi-continuous (no lower bound condition) and $P_{\beta}(h)\in PSH(X,\beta)$, write $G:=\sup\{\varphi|\varphi\in\operatorname{PSH}(X,\beta),\;\varphi\leq h \quad \text{q.e.} \}$. We know that $G\leq G^*=P_{\beta}(h)\leq C$, hence $G$ is bounded above; From Choquet's lemma, there exists a countable sequence $\varphi_j\leq h \quad \text{q.e.}$ such that $G^*=\left( \sup \varphi_j\right)^*$. By \cite[Lemma 2.1]{SW25}, we have $P_{\beta}(h)=G^*\leq h \quad \text{q.e.}$ (We get $P_{\beta}(h)=P_{\beta}(\min(h,C))$).   As a result, we get that $G^*\leq G$, which implies that $G^*=G$. Thus if  $h_j\downarrow h$ and $h_1$ is bounded from above, we have $P_{\beta}(h_j)\downarrow P_{\beta}(h)$.

\end{remark}

\begin{theorem}\label{envelope for minimal sing type}
      Let $\{\beta\}\in BC^{1,1}(X)$ be a nef Bott-Chern class such that  $ D(X,\beta)\neq\emptyset$ and that $L(V_\beta)$ is a pluripolar set. Assume that $h$ is quasi-continuous on $X$ and that $P_\beta(h) \in \operatorname{PSH}(X, \beta)$ is a $\beta$-psh function with minimal singularity. Then $P_\beta(h) \leq h$ outside a pluripolar set, and moreover we have
$$
\int_{\left\{P_\beta(h)<h\right\}}\left(\beta+d d^c P_\beta(h)\right)^n=0 .
$$
\end{theorem}    
\begin{proof}
The idea of proof is adapted from \cite{ALS24}.
By \cref{bounded above reduction and decreasing}, we may assume that $h$ is bounded above. Set
$$
h_j:=\max (h,-j),
$$
which is a sequence of quasi-continuous bounded functions, hence $P_\beta\left(h_j\right) \in \operatorname{PSH}(X, \beta)$ and $P_\beta\left(h_j\right) \leq h_j$  quasi-everywhere. Since $P_\beta(h) \in \operatorname{PSH}(X, \beta)$, we conclude that $P_\beta(h) \leq h$ quasi-everywhere by \cref{bounded above reduction and decreasing} again.

Next, we prove the second result.
By \cref{lem: contact for beta bounded MA}
$$
\int_{\left\{P_\beta\left(h_j\right)<h_j\right\}}\left(\beta+d d^c P_\beta\left(h_j\right)\right)^n=0, \quad \forall j \geq 1.
$$

Fix $k \in \mathbb{N}$. For every $j \geq k$, we have $\left\{P_\beta\left(h_k\right)<h\right\} \subseteq\left\{P_\beta\left(h_j\right)<h_j\right\}$, which implies that
$$
\int_{\left\{P_\beta\left(h_k\right)<h\right\}}\left(\beta+d d^c P_\beta\left(h_j\right)\right)^n=0.
$$

Now fix $C>0$. From \cref{max_global} we can write

$$
\int_{\left\{P_\beta\left(h_k\right)<h\right\} \cap\left\{P_\beta(h)>V_{\beta}-C\right\}}\left(\beta+d d^c \max \left(P_\beta\left(h_j\right), V_{\beta}-C\right)\right)^n=0.
$$

Set
$$
f_k:=\left[\max \left(h, P_\beta\left(h_k\right)\right)-P_\beta\left(h_k\right)\right] \times\left[\max \left(P_\beta(h), V_{\beta}-C\right)-(V_{\beta}-C)\right],
$$
which is a non-negative, quasi-continuous bounded function. By \cite[Lemma 4.2]{KH09}, we have
$$
\int_X f_k\left(\beta+d d^c \max \left(P_\beta\left(h_j\right), V_{\beta}-C\right)\right)^n=0, \quad \forall j \geq k.
$$

By weak convergence theorem of Bedford and Taylor,
$$
f_k\left(\beta+d d^c \max \left(P_\beta\left(h_j\right), V_{\beta}-C\right)\right)^n \rightarrow f_k\left(\beta+d d^c \max \left(P_\beta(h), V_{\beta}-C\right)\right)^n,
$$
on $U:=X-L(V_\beta)$, which implies that
$$
\begin{aligned}
\int_X f_k\left(\beta+d d^c \max \left(P_\beta(h), V_{\beta}-C\right)\right)^n & \leq \liminf _{j \rightarrow+\infty} \int_X f_k\left(\beta+d d^c \max \left(P_\beta\left(h_j\right), V_{\beta}-C\right)\right)^n,\\
& =0.
\end{aligned}
$$
Here the first inequality holds true on any compact subset of $U$ and we conclude by invoking that both measures put no mass on $L(V_\beta)$. The above is equivalent to
$$
\int_{\left\{P_\beta\left(h_k\right)<h\right\} \cap\left\{P_\beta(h)>V_{\beta}-C\right\}}\left(\beta+d d^c \max \left(P_\beta(h), V_{\beta}-C\right)\right)^n=0 
$$
by \cite[Lemma 4.2]{KH09} again. By \cref{max_global}, 
$$
\int_{\left\{P_\beta\left(h_k\right)<h\right\} \cap\left\{P_\beta(h)>V_{\beta}-C\right\}}\left(\beta+d d^c P_\beta(h)\right)^n=0 .
$$
Firstly let $k\rightarrow +\infty$, then let $C\rightarrow+\infty$, because $(\beta+dd^cP_{\beta}(h))^n$ puts no mass on $\{V_{\beta}=-\infty\}$, we arrive at the final conclusion.  
\end{proof}

\subsubsection{Domination principle}
In this subsection, we establish the domination principle analogous to \cref{domination for beta}, which is an extension of the corresponding result in \cite{GL22}, \cite{GL23}, \cite{SW25}, \cite{PSW25}.
\begin{theorem}\label{domination of MA}
    Let $\{\beta\}\in BC^{1,1}(X)$ be a nef Bott-Chern class such that  $ D(X,\beta)\neq\emptyset$. Moreover, we assume that $L(V_\beta)$ is a pluripolar set and that $\underline{\operatorname{Vol}}(\beta)>0$. Fix $c\in[0,1)$, then for $u,v\in \operatorname{PSH}(X,\beta)$ with minimal singularity types satisfying
    $\mathds 1_{\{u<v\}}\left(\beta+d d^c u\right)^n \leq c \mathds 1_{\{u<v\}}\left(\beta+d d^c v\right)^n$, we have $u \geq v$.
\end{theorem}
\begin{proof}
    The proof is inspired by the recent work \cite[Theorem 2.11]{ALS24} (see also \cite{PSW25}).
By \cref{max_global}, we have
$$c\left(\beta +d d^c \max (u,v)\right)^n=c\left(\beta+d d^c v\right)^n$$
on $\{u<v\}$. Therefore, we can replace $v$ by $\max (u, v)$ and assume without loss of generality that $u\leq v$. Our goal is to prove that $u=v$.

    Fix $b>1$, we set $u_b:=P_{\beta}\left(bu-(b-1)v\right)$ on $X$, which is clearly $\beta$- psh with minimal singularity type.
Using \cref{envelope for minimal sing type}, we know that $(\beta+dd^c u_b)^n$ is concentrated on the contact set $D:=\{u_b= bu-(b-1)v\}$.
Since $b^{-1} u_b+\left(1-b^{-1}\right) v\leq u$ on $X$ with equality on $D$, by the maximum principle ( \cref{max_global2}),  we get
$$
\begin{aligned}
& \mathds{1}_D\left(\beta + d d^c u\right)^n \geq \mathds{1}_D\left(\beta+  d d^c\left(b^{-1} u_b+\left(1-b^{-1}\right) v\right)\right)^n \\
\geq & \mathds{1}_D b^{-n}\left(\beta + d d^c u_b\right)^n+\mathds{1}_D\left(1-b^{-1}\right)^n\left(\beta+  d d^c v\right)^n \\
\geq &\mathds{1}_D b^{-n}\left(\beta+  d d^c u_b\right)^n+\mathds{1}_D \cdot c\cdot \left(\beta+ d d^c v\right)^n,
\end{aligned}
$$
if we take $b$ sufficiently large. By our assumption, we have $\mathds 1_{D \cap\{u<v\}}\left(\beta+ d d^c u_b\right)^n=0$.

Recall that  $V_{\beta}$ has zero Lelong number everywhere by \cref{non-pluripolar}, so does $u_b$ since $u_b$ shares the same singularity type with $V_\beta$. By Demailly's Approximation Theorem, there exists smooth $\psi_j\in \operatorname{PSH}(X,\beta+\frac{1}{j}\omega)$ such that 
\begin{align*}
&\int_{X}(\beta+dd^c u_{b})^n\\
=&\lim_{j} \int_{X}(\beta+\frac{1}{j}\omega+dd^c \psi_j)^n\\
\geq &\lim_{j} \underline{\operatorname{Vol}}(\beta+\frac{1}{j}\omega)=\underline{\operatorname{Vol}}(\beta)>0.
\end{align*}

Since $\left(\beta+d d^c v\right)^n$ does not charge $\{v=-\infty\}$ by \cref{non-pluripolar}, one can construct a concave increasing function $h: \mathbb{R}^{+} \rightarrow \mathbb{R}^{+}$such that $h(+\infty)=+\infty$ and $h(|v|) \in L^1\left(\left(\beta+d d^c v\right)^n\right)$. Thus, by \cref{max_global2}, 
$$
\begin{aligned}
\int_X h\left(\left|u_b\right|\right) 
\left(\beta+d d^c u_b\right)^n & =\int_{D \cap\{u=v\}} h(|v|)\left(\beta+d d^c u_b\right)^n \\
& \leq \int_X h(|v|)\left(\beta+d d^c v\right)^n<+\infty.
\end{aligned}
$$

If $\left(\sup _X u_b\right)_{b>1}$ has a subsequence converging to $-\infty$, then for every positive scalar $\alpha$, one can find $b$ such that $\sup _X u_b \leq-\alpha$. This implies
\begin{equation} \label{eq nonvanising}
    \begin{aligned}
\int_X h\left(\left|u_b\right|\right)\left(\beta +d d^c u_b\right)^n\ & \geq h(\alpha) \int_X\left(\beta+d d^c u_b\right)^n \\
& \geq h(\alpha) \underline{\operatorname{Vol}}(\beta),
\end{aligned}
\end{equation}
which is impossible because $\sup _{b \geq 1} \int_X h\left(\left|u_b\right|\right)\left(\beta+d d^c u_b\right)^n<+\infty$. Therefore, $\left(\sup _Xu_b\right)_{b>1}$ is uniformly bounded.
 The sequence $u_b$ converges in $L^1$ and almost everywhere to a function $w \in\operatorname{PSH}(X, \beta)$.

 Fix $a>0$. On $\{u<v-a\}$, we have $u_b \leq v-a b$, hence
$$
\begin{aligned}
\int_{\{u<v-a\}} w \cdot \omega^n & =\lim _{j \rightarrow+\infty} \int_{\{u<v-a\}} u_{b_j} \cdot \omega^n \\
& \leq \lim _{j \rightarrow+\infty}\left(-a b_j\right) \int_{\{u<v-a\}} \omega^n+\int_{\{u<v-a\}} v \cdot \omega^n.
\end{aligned}
$$
Since $w \in L^1(X)$, we infer that
$$
\int_{\{u<v-a\}} \omega^n=0. 
$$
Thus we have $u\geq v-a$, letting $a\rightarrow 0$, we finish the proof.

\end{proof}
\begin{remark}
    when $\beta$ is closed, the condition $\underline{\operatorname{Vol}}(\beta)>0$ can be weakened as $\int_{X}\beta^n>0$. In fact, arguements in \eqref{eq nonvanising} will also work due to the integration by part formula  
    $$\int_{X}(\beta+dd^c u_b)^n = \int_{X} \beta^n.$$
\end{remark}

The same arguments as in \cref{domination for beta 2} and \cref{domination for beta 3} give the following corollary.
\begin{corollary}
    Let $\{\beta\}\in BC^{1,1}(X)$ be a nef Bott-Chern class such that $D(X,\beta)\neq\emptyset$ and $L(V_\beta)$ is a pluripolar set. Assume moreover that $\underline{\operatorname{Vol}}(\beta)>0$. Then for $u,v\in \operatorname{PSH}(X,\beta)$ with minimal singularities satisfying
    $e^{-u}\left(\beta+d d^c u\right)^n \leq e^{-v}\left(\beta+d d^c v\right)^n$, we have $u \geq v$.
\end{corollary}

\begin{corollary}
   Let $\{\beta\}\in BC^{1,1}(X)$ be a nef Bott-Chern class such that $D(X,\beta)\neq\emptyset$ and $L(V_\beta)$ is a pluripolar set. Assume moreover that $\underline{\operatorname{Vol}}(\beta)>0$. Then for   $u,v\in \operatorname{PSH}(X,\beta)$ such that with minimal singularities satisfying
    $$ (\beta+dd^c u)^n \leq c (\beta+dd^c v)^n$$ for some constant $c>0$, then $c\geq 1$.
\end{corollary}

\begin{corollary}\label{contact MA 3}
    Let $\{\beta\}\in BC^{1,1}(X)$ be a nef Bott-Chern class such that  $ D(X,\beta)\neq\emptyset$.   Let $\lambda \geq 0$ be a constant and let $u, v\in \operatorname{PSH}(X,\beta)$ be $\beta$- psh functions with minimal singularity types . Fix two smooth $(1,1)$-forms $\beta_1, \beta_2$ such that $\beta_1 \geq \beta, \beta_2 \geq \beta$.
     
(i) If $\left(\beta_1+d d^c u\right)^n \leq e^{\lambda u} f \omega^n$ and $\left(\beta_2+d d^c v\right)^n \leq e^{\lambda v} g \omega^n$, then
$$
\left(\beta+d d^c P_\beta(\min (u, v))\right)^n \leq e^{\lambda P_\beta(\min (u, v))} \max (f, g) \omega^n ;
$$

(ii) If $\left(\beta+d d^c u\right)^n \geq e^{\lambda u} f \omega^n$ and $\left(\beta+d d^c v\right)^n \geq e^{\lambda v} g \omega^n$, then
$$
\left(\beta+d d^c \max (u, v)\right)^n \geq e^{\lambda \max (u, v)} \min (f, g) \omega^n .
$$
\end{corollary}
\begin{proof}
The second statement is a direct consequence of \cref{max_global2}. We now prove the first one. Set $\varphi:=P_\beta(\min (u, v))$, which is also $\beta$- psh with minimal singularity types. Using \cref{lem: contact for beta bounded MA} and \cref{max_global2} we can write
\begin{align*}
        (\beta+dd^c\varphi)^n&\leq\mathds{1}_{\{\varphi=u<v\}}(\beta+dd^c\varphi)^n +\mathds{1}_{\{\varphi=v\}}(\beta+dd^c\varphi)^n\\
        &\leq\mathds{1}_{\{\varphi=u<v\}} (\beta_1+dd^c\varphi)^n+\mathds{1}_{\{\varphi=v\}} (\beta_2+dd^c\varphi)^n\\
        &\leq\mathds{1}_{\{\varphi=u<v\}} (\beta_1+dd^cu)^n+\mathds{1}_{\{\varphi=v\}} (\beta_2+dd^cv)^n\\
        &\leq\mathds{1}_{\{\varphi=u<v\}} e^{\lambda u}f\omega^n+\mathds{1}_{\{\varphi=v\}} e^{\lambda v}g\omega^n\\
        &\leq e^{\lambda P_{\beta}(\min(u,v))}\max(f,g)\omega^n
    \end{align*}
\end{proof}
\subsubsection{The final solution}
Now, we are in a position to solve the non-twist CMA equations in nef classes:
\begin{theorem}\label{small unbounded locus non-twist equation}
     Let $\{\beta\}\in BC^{1,1}(X)$ be a nef Bott-Chern class satisfying $\underline{\operatorname{Vol}}(\beta)>0$ and $ D(X,\beta)\neq\emptyset$. Assume moreover that $L(V_\beta)$ is a pluripolar set. Then, for any $f\in L^p(X), p>1$ such that $\int_Xf\omega^n>0$, there exists a function $\varphi$ of minimal singularity type  and a unique constant $c>0$ solving the following equation:
    $$
    (\beta+dd^c \varphi)^n =c f\omega^n.
    $$
\end{theorem}
\begin{proof}
The uniqueness of the constant $c$ follows from the domination principle. Choose a sequence of smooth positive functions $f_j$ such that $f_j\rightarrow f$ in $L^p(X,\omega^n)$. For each $j \in \mathbb{N}^*$ we use the main theorem of \cite{TW10} to find $u_j \in \operatorname{PSH}(X, \beta+\frac{1}{j}\omega)\cap C^\infty(X)$ and $c_j>0$ solving
$$
\left(\beta+\frac{1}{j}\omega+d d^c u_j\right)^n=c_j f_j \omega^n,\quad\sup_Xu_j=0.
$$

Since it is standard that $\int_X\left(\beta+\frac{1}{j}\omega+d d^c u_j\right)^n\geq\underline{\operatorname{Vol}}(\beta+\frac{1}{j}\omega)\geq\underline{\operatorname{Vol}}(\beta)>0$, integrating both sides we get a uniform lower bound of $c_j$. On the other hand, the mixed type inequality gives an upper bound for $c_j$, as illustrated in \cref{main thm for m-positive beta}, so up to extracting a subsequence we may assume that $c_j\rightarrow c>0$.

By the a priori estimate \cref{a priori in nef class for MA}, there is a uniform constant $C$ independent of $j$ such that 
\begin{equation}\label{uniform estimate}
-C\leq u_j-V_{\beta+\frac{1}{j}\omega}\leq0.
\end{equation}
Assume that $u_j\rightarrow u\in\operatorname{PSH}(X,\beta)$ in $L^1(X)$ and consider the following sequences
$$
\varphi_j=\mathrm{P}_\beta\left(\inf _{k \geq j} u_k\right), \psi_j=(\sup _{k \geq j} u_k)^*.
$$
The uniform estimate \eqref{uniform estimate} yields that 
$$-C\leq\varphi_j-V_\beta\leq u-V_\beta\leq\psi_j-V_\beta\leq0.$$
The Hartogs' lemma implies that $\psi_j$ decreases to $\psi$ and we also assume that $\varphi_j$ increases almost everywhere to a function $\varphi\in\operatorname{PSH}(X,\beta)$.

For each fixed $j$ and $l\geq j$,
$$
(\beta+\frac{1}{j}\omega+dd^cu_l)^n\geq(\beta+\frac{1}{l}+dd^cu_l)^n=c_lf_l\omega^n.
$$
According to the maximum principle (\cref{max_global2}),
\begin{align*}
   (\beta+\frac{1}{j}\omega+dd^c\max(u_j, \cdots, u_{j+l}))^n\geq \min{(c_j,\cdots,c_{j+l})}\min(f_j,\cdots,f_{j+l})\omega^n\geq(\inf_{l\geq j}c_l)(\inf_{l\geq j}f_l)\omega^n.
\end{align*}
Taking limit of $l$ on the left hand side, by Bedford-Taylor's increasing convergence theorem we get on $X-L(V_\beta)$,
\begin{equation}\label{eq 22 MA}
    (\beta+\frac{1}{j}\omega+dd^c\psi_j)^n\geq(\inf_{l\geq j}c_l)(\inf_{l\geq j}f_l)\omega^n.
\end{equation}
Since both sides are positive Radon measures putting no mass on $L(V_\beta)$ (see \cref{non-pluripolar}), \eqref{eq 22 MA} holds true on the whole $X$. Now since $c_j\rightarrow c>0$ and $f_j\rightarrow f$ almost everywhere, we have $\inf_{l\geq j}c_l$ increases to $c$ and $\inf_{l\geq j}f_l$ increases to $f$ almost everywhere, which in turn implies the weak convergence $(\inf_{l\geq j}c_l)(\inf_{l\geq j}f_l)\omega^n\rightarrow cf\omega^n$. Letting $j\to\infty$ in \eqref{eq 22 MA},  the standard monotone convergence theorem yields that
\begin{equation}\label{eq 25 MA}
(\beta + dd^c \varphi)^n \geq cf\omega^n.
\end{equation}

On the other hand, since $\beta\leq\beta_j$ and $(\beta_j+dd^cu_j)^n=e^{u_j-u_j}c_jf_j\omega^n$, we can use \cref{contact MA 3} to write
$$
(\beta+dd^cP_{\beta}(u_j,u_{j+1}))^n\leq e^{P_{\beta}(u_j,u_{j+1})}\max(c_je^{-u_j}f_j,c_{j+1}e^{-u_{j+1}}f_{j+1})\omega^n.
$$
By using the decreasing convergence theorem and an easy induction argument, we can write
\begin{equation}\label{eq 23 MA}
    (\beta+dd^c\varphi_j)^n\leq e^{\psi_j-\inf_{l\geq j}\varphi_l}(\sup_{l\geq j}c_l)(\sup_{l\geq j}f_l)\cdot\omega^n.
\end{equation}
It is clear that $\sup_{l\geq j}c_l\rightarrow c$, $\sup_{l\geq j}f_l$ decreases to $f$ almost everywhere hence weakly. Letting $j\rightarrow\infty$ in \eqref{eq 23  MA} we get on $X-L(V_\beta)$,
$$
  (\beta+dd^c\varphi)^n\leq e^{\psi-\varphi}cf\omega^n,
$$
by Bedford-Taylor's increasing convergence theorem again. Since both sides are positive Radon measures on $X$ vanishing on $L(V_\beta)$, \eqref{eq 24 MA} holds true on the whole $X$. This combined with \eqref{eq 25 MA} yields that
$$
e^{-\psi}\beta_\psi^n\leq e^{-\varphi}cf\omega^n\leq e^{-\varphi}\beta_\varphi^n.
$$
Now the domination principle \cref{domination of MA}  gives that $\varphi\leq\psi$. The reverse inequality $\varphi\geq\psi$ is clear from the definition, thus we obtain $\varphi=\psi$ and 
\begin{equation}\label{eq 24 MA}
(\beta+dd^c\varphi)^n=cf\omega^n.
\end{equation}
The proof is therefore concluded.
\end{proof}
\begin{remark}
There are two directions to generalize \cref{small unbounded locus non-twist equation}. For one thing, it is very desirable to remove the assumption that $L(V_\beta)$ is a pluripolar set, as can be done in the twist setting (see \cref{small unbounded locus twist equation}). For this issue, we could probably try to generalize the convergence theorem \cref{lem: weak convergence1} to the setting in \cite{Ceg12}.

   For another, it seems that we can even drop the assumption $D(X,\beta)\neq\emptyset$ in \cref{small unbounded locus non-twist equation} and \cref{small unbounded locus twist equation} and just assume that $L(V_\beta)$ is pluripolar. That is, we define Monge-Amp\`ere measures of potentials with minimal singularities outside $L(V_\beta)$ and extend by zero to $X$. We leave these as a future project.
\end{remark}

\subsection{Applications}
By applying \cref{a priori in nef class for MA}, \cref{small unbounded locus twist equation} and \cref{small unbounded locus non-twist equation} within the framework of \cite{Ch16, Pop16, Ngu16, LWZ24, SW25}, partial resolutions to the Tosatti-Weinkove and Demailly-Păun conjectures are obtained. Their proofs are omitted and left to the interested reader.
\begin{theorem}\label{thm:TW}
Let $X$ be a compact Hermitian manifold.    Let $\{\beta\}\in BC^{1,1}(X)$ be a nef Bott-Chern class satisfying $\underline{\operatorname{Vol}}(\beta)>0$ and $ D(X,\beta)\neq\emptyset$. Assume moreover that $L(V_\beta)$ is a pluripolar set.   Let $x_1, \ldots, x_N$ be fixed points and $\tau_1, \ldots, \tau_N$ be positive real numbers such that
$$
\sum_{i=1}^N \tau_i^n<\underline{\operatorname{Vol}}(\beta).
$$
Then there exists a $\beta$-psh function $\varphi$ with logarithmic poles at $x_1, \ldots, x_N$ :
$$
\varphi(z) \leq c \tau_j \log |z|+O(1)
$$
in a coordinate ball $\left(z_1, \ldots, z_n\right)$ centered at $x_j$, where $c \geq 1$ is a uniform constant.
\end{theorem}

\begin{theorem}\label{thm:DP}
     Let ($X, \omega$) is a compact Hermitian manifold, with $\omega$ a pluriclosed Hermitian metric, i.e. $d d^c \omega=0$. Let $\{\beta\} \in H^{1,1}(X, \mathbb{R})$ be a closed real nef class with smooth representative $\beta$, such that $D(X,\beta))\neq\emptyset$ and $\underline{Vol}(\beta)>0$. Assume moreover that $L(V_\beta)$ is a pluripolar set. Then $\{\beta\}$ contains a Kähler current. 
\end{theorem}

\end{document}